\documentclass[a4paper]{amsart}
\usepackage{tikz}
\usetikzlibrary{calc, arrows.meta}
\usepackage[english]{babel}
\usepackage[utf8]{inputenc}
\usepackage{latexsym}
\usepackage{graphicx}
\usepackage{array}
\usepackage{amsthm}
\usepackage{amsmath}
\usepackage{amssymb}
\usepackage{amsfonts}
\usepackage{mathrsfs}
\usepackage{mathtools}
\usepackage{tikz-cd} 
\usepackage{faktor}
\usepackage{xfrac}
\usepackage{tensor}
\usepackage{fancyhdr} %para el encabezado y pie de página
\usepackage{enumitem} %para que los => y <= no se crucen con la raya vertical
\usepackage{mdframed} %para que los teoremas queden con la raya vertical
\usepackage{amssymb}
\usepackage[hidelinks]{hyperref}
\usepackage{scalerel}
\usepackage{pgfplots}
\usepackage{pdfpages}
\usepackage{bbm}
\usepackage{listings}
\usepackage{tgbonum}
\usepackage[T1]{fontenc}
\mathtoolsset{showonlyrefs} %Solo pone numeritos a las ecuaciones que uso.
\usetikzlibrary{decorations.markings}
\tikzset{tangent/.style={postaction={decorate,decoration={markings,
				mark=at position #1 with {
					\draw[red] (0pt,0pt) coordinate (Tang) -- (0pt,0pt);)}}}}}

\allowdisplaybreaks
\title{Dispersion for the wave equation with Robin boundary conditions on a 2d convex model domain}
\author{Diego Sánchez Sanz$^1$}
\thanks{$^1$Sorbonne Université, CNRS, LJLL, Paris, France}
\date{}%\today}

\newtheorem{defi}{Definition}
\numberwithin{defi}{section}
\newtheorem{teo}[defi]{Theorem}
\newtheorem{cor}[defi]{Corollary}
\newtheorem{lema}[defi]{Lemma}
\newtheorem{prop}[defi]{Proposition}
\theoremstyle{definition}
\newtheorem{obs}[defi]{Remark}

\newtheorem{nota}[defi]{Notation}
\DeclareMathOperator{\arccot}{arccot}

\theoremstyle{definition}

\theoremstyle{plain}

\usepackage{etoolbox}
\patchcmd{\section}{\normalfont}{\normalfont\Large}{}{}
\begin{document}
\begin{abstract}
		This paper presents a study of dispersion for the two-dimensional Friedlander model with Neumann and Robin boundary conditions, thereby extending previous works treating only Dirichlet boundary conditions. By considering an Airy-type special function, we obtain a parametrix construction similar to the Dirichlet case but with additional terms requiring careful control.
		By studying the dispersion of this parametrix, we obtain the same dispersive estimates as in previous results for the Dirichlet case, and therefore the same Strichartz estimates.
		\end{abstract}
		
		\maketitle
\section{Introduction}
We consider the wave equation on a Riemannian manifold $\Omega$ with non-empty boundary $\partial\Omega$:
\[ \begin{cases} 
	(\partial_t^2-\Delta)u(t,x) =0& x\in \Omega\\
	u(0,x)=u_0(x),	\partial_t u(0,x)=u_1(x)\\
	Bu(t,x)=0& x\in\partial\Omega.
\end{cases}
\]
Here $\Delta$ denotes the Laplace-Beltrami operator in $\Omega$. The operator $B$ usually encodes either Dirichlet or Neumann boundary conditions.\\
In our case, however, we are interested in a linear combination of the Dirichlet and Neumann conditions, which we call Robin boundary conditions, that is, for fixed $a_r,b_r\in \mathbb{R}$, $Bu(t,x)=a_r u(t,x)+b_r \partial_{n} u(t,x)$.
\begin{obs}
Since we are considering $Bu=0$, we can normalize and consider that $|(a_r,b_r)|=1$, that is, instead of considering $a_r,b_r\in\mathbb{R}$, we consider $Bu(t,x)=\cos(\iota) u(t,x)-\sin(\iota)\partial_{n} u(t,x)$, $\iota\in (0,\pi)$, excluding the Dirichlet case and with the Neumann case being $\iota=\pi/2$.
\end{obs} 
For this problem with Dirichlet conditions, solutions are known to have local-in-time Strichartz estimates \cite{blair_strichartz_2009} with some losses in the admissible pairs range \cite{ivanovici_new_2021}.
A central problem in the field is to determine the smallest loss exponent (which we will call $\gamma(n)$ with $n$ being the dimension of the manifold) for which local-in-time Strichartz estimates hold for a given geometrical setting. This is usually approached by estimating dispersive properties of the wave. Under some possible additional conditions on the boundary, or on specific toy models \cite{positive}, strict convexity is classical \cite{ivanovici_dispersion_2023}.\\
That is, the general aim is to find the smallest possible $\gamma(n)$ such that the following statement holds:\\
	Let $q,r\in[2,\infty]$ such that
	\begin{equation}
		\frac{1}{q}\le \left(\frac{1}{2}-\gamma(n)\right)\left(\frac{1}{2}-\frac{1}{r}\right).
	\end{equation}
	Let $t_{max}\in\mathbb{R}_+$. Then, there exists a constant $C>0$ independent of $u_0,u_1$ such that
	\begin{equation}
		||u||_{L^q([0,t_{max}],L^r(\Omega))}\le C(||u_0||_{H^\beta(\Omega)}+||u_1||_{H^{\beta-1}(\Omega)}),
	\end{equation}
	where $\beta=2\left(\frac{1}{2}-\frac{1}{r}\right)-\frac{1}{q}$, and where $u$ is the solution to the wave equation with Dirichlet, Neumann or Robin boundary conditions and initial data $(u_0,u_1)$ and $\Omega$ a specific or general manifold of dimension $n$.\\
By standard reductions, we may localize in frequencies and trade the spectral localization for a localization with respect to the time frequency using the equation. This, together with the fact that the estimates are local in time, allows us to deduce the previous statement from the local high frequency estimates:\\
	Let $q,r\in[2,\infty]$ such that
	\begin{equation}
		\frac{1}{q}\le \left(\frac{1}{2}-\gamma(n)\right)\left(\frac{1}{2}-\frac{1}{r}\right).
	\end{equation}
	Let $t_{max}\in\mathbb{R}_+$ and $\psi_0\in \mathcal{C}^\infty_0([\frac{1}{2},\frac{3}{2}])$. Then, there exists a constant  $C>0$ independent of $u_0,u_1$ such that
	\begin{equation}
		||\psi_0(h^2 D_t^2)u||_{L^q([0,t_{max}],L^r(\Omega))}\le h^{1-\beta}C(||u_0||_{H^1(\Omega)}+||u_1||_{L^2(\Omega)}),
	\end{equation}
	where $\beta=2\left(\frac{1}{2}-\frac{1}{r}\right)-\frac{1}{q}$, and where $u$ is the solution to the wave equation with Dirichlet, Neumann or Robin boundary conditions and with initial data $(u_0,u_1)$ and $\Omega$ a specific or general manifold of dimension $n$.\\
It is known that for the Friedlander model and Dirichlet boundary conditions, $\gamma(2)\le \frac{1}{9}$ the result holds \cite{positive}. The main goal of this paper is to prove that the result holds also for Robin boundary conditions (and Neumann as a particular case for us), that is, for the Friedlander model and $\gamma(2)\le\frac{1}{9}$, the previous estimates hold.
\subsection{Motivation}
We are interested in the behavior of the wave equation with Robin boundary conditions in a convex domain, so we will start by considering the simplest convex domain. In principle, that would be the unit disk, however, the Friedlander model is simpler and it encapsulates many of the features of a convex domain (and the disk in particular). We provide a heuristic motivation.\\ 
We consider the following problem with $\boldsymbol{w}\in\mathbb{R}^2$:
\[ \begin{cases} 
	(\partial_t^2-\Delta)u(t,\boldsymbol{w}) =0& |\boldsymbol{w}|\le 1 \\
	u(0,\boldsymbol{w})=u_0(\boldsymbol{w}),	\partial_t u(0,\boldsymbol{w})=u_1(\boldsymbol{w})\\
	Bu(t,\boldsymbol{w})=0& |\boldsymbol{w}|=1.
\end{cases}
\]
This is the boundary value problem in the disk, with $B$ corresponding to the Dirichlet, Neumann or Robin boundary conditions. We change the problem to polar coordinates $(r,\theta)$, then we change the radial variable $r$ so that $r=1-\frac{x}{2}$, and we set $y=\theta$. We then change the sign to make the problem on the upper half-plane. Finally we linearize the metric coefficients near the boundary $\frac{1}{(1-\frac{x}{2})^2}\sim 1+x$ and we discard the lower order term $\frac{1}{2-{x}} \partial_x$. After all these transformations, we obtain the Friedlander equation:
\[ \begin{cases} 
	(\partial_t^2-\partial_x^2-(1+x)\partial_y^2)u(t,x,y) =0& x>0 \\
	u(0,x,y)=u_0(x,y),	\partial_t u(0,x,y)=u_1(x,y)\\
	Bu(t,x,y)=0& x=0.
\end{cases}
\]
This indicates that the Friedlander equation is a model for the wave equation near the boundary of the disk and locally in time. Thus, understanding the Friedlander equation gives insight into the convex boundary problem of the wave equation using a simpler model.\\
Furthermore, since the circle is the most basic case of convexity, the Friedlander model is a useful toy model for understanding convex domains in general (see \cite{me_taylor_boundary_nodate} section 1.4).\\\\
\begin{obs}
	This is a heuristic motivation, the exact disk problem should be studied independently and we cannot a priori prove properties for the disk by studying the Friedlander model. However, it provides useful intuition about which mechanisms are expected to persist.
\end{obs}
\subsection{Setting}\label{setting}
We therefore consider the two-dimensional Friedlander model:
Given $\iota\in(0,\pi),  u_0,u_1\in \mathcal{D}'(\Omega_2)$, we define $\Delta_F:=\partial_x^2+(1+x)\partial_y^2$ and we have the following problem:
\begin{align} 
	\label{baseprob}
	\begin{cases} 
		(\partial_t^2-\Delta_F)u(t,x,y) =0,& x>0 \\
		u(0,x,y)=u_0(x,y),	\partial_t u(0,x,y)=u_1(x,y)\\
		\cos(\iota)u(t,x,y)-\sin(\iota) \partial_x  u(t,x,y)=0,& x=0.
	\end{cases}
\end{align}
We will refer to this equation as the Friedlander equation.
\begin{nota}
	\label{firstnot}
	We will denote $b:=\cot(\iota)$.\\
	For simplicity, we write $\Omega_2:=\mathbb{R}_{+}\times \mathbb{R}$.\\
	We use the following vector notation:
	$\boldsymbol{t}:=(t,x,y)$, $\boldsymbol{T}:=(T,X,Y)$, $\boldsymbol{\sigma}=(\sigma,s)$.
\end{nota}
\begin{obs}
	The Neumann boundary condition is included in our study, since it is the case $\iota=\frac{\pi}{2}$, or $b=0$. We note here that Neumann is very natural for our case, since our Green function tends to the Neumann Green function, as $h\rightarrow 0$. Furthermore, our case does include the Neumann condition but does not include the Dirichlet condition. The Robin condition is then more closely related to Neumann than to Dirichlet in this high-frequency regime.
\end{obs}
We prove the following result:
\begin{teo}
	\label{mainteo}
	Let $q,r\in[2,\infty]$ such that
	\begin{equation}
		\frac{1}{q}\le \left(\frac{1}{2}-\frac{1}{9}\right)\left(\frac{1}{2}-\frac{1}{r}\right).
	\end{equation}
	Let $t_{max}\in\mathbb{R}_+$, $\psi_0\in \mathcal{C}^\infty_0(1-2^{-9},1+2^{-9})\subset \mathcal{C}^\infty_0(\frac{1}{2},\frac{3}{2})$. Then, there exists a constant $C>0$ such that, for every solution $u$  to the Friedlander equation with Robin boundary conditions \eqref{baseprob} and initial data $(u_0,u_1)\in H^1\times L^2(\Omega_2)$, we have the following Strichartz estimates
	\begin{equation}
		||\psi_0(h^2 D_t^2)u||_{L^q([0,t_{max}],L^r(\Omega_2))}\le h^{1-\beta}C(||u_0||_{H^1(\Omega_2)}+||u_1||_{L^2(\Omega_2)}),
	\end{equation}
	where $\beta=2\left(\frac{1}{2}-\frac{1}{r}\right)-\frac{1}{q}$, $0<h<h_0$ and the constant $C$ is independent of $h,u_0,u_1$.
\end{teo}
\begin{obs}
	We record the following remarks about the theorem:
	\begin{enumerate}
		\item The constant $C$ in the theorem depends on $\iota$, $t_{max}$ and the specific cut-offs chosen $\psi_0, \psi_1$, but it does not depend on $h, u_0,u_1$.
		\item The exponent $\beta$ does not depend on the Robin condition $\iota$.
		\item Our Robin setting does include the Neumann case $\iota=\frac{\pi}{2}$, but does not include the Dirichlet case $\iota=0,\pi$.
	\end{enumerate}
\end{obs}
\begin{obs}
	We make here some remarks about possible extensions which could be derived from this work in the future.
	\begin{enumerate}
		\item The constant $C$ of this theorem depends on the specific Robin problem chosen, that is, it depends on $\iota$, however, it is natural to ask whether there is continuity of $C(\iota)$, since we have $\iota\in[0,\pi]$ in a compact domain (with $\iota=0,\iota=\pi$ being the Dirichlet case, which holds by \cite{positive}). One may therefore conjecture that there is a constant $C$ independent of $\iota$ for which the result holds. This question is not treated in this paper and remains open.
		\item This model, while it is studied in 2D, may be generalized to higher dimensional isotropic model domains by standard methods (see \cite{ivanovici_new_2021},\cite{ivanovici_dispersion_2014}).
	\end{enumerate}
\end{obs}
\begin{obs}
	The relevant part of the wave for dispersion is the part of the rays which are tangential to the boundary $hD_y\sim 1$, so the rays transverse to the boundary only have one reflection at most and their contribution is handled by  standard transverse reflection estimates which we do not repeat. The tools for dealing with these rays are the ones used in \cite{blair_strichartz_2009}, or \cite{OanaIvanovici2021DiscreteandContinuousDynamicalSystems}.
This means that we can add a harmless cut-off $\psi_1(hD_y)$ with $\psi_1\in \mathcal{C}^\infty_0(\sqrt{1-2^{-10}},\sqrt{1+2^{-10}})\subset\mathcal{C}^\infty_0(\frac{1}{2},\frac{3}{2})$ to the function $u$ without altering the problem.
\end{obs}
\subsection{Structure of the proof and comparison with the Dirichlet case \cite{positive}.}

In this paper, we follow the same general strategy as in \cite{positive}: we construct the Green function of the Friedlander operator after Fourier transform in the $y$ variable, and then analyze its dispersion in order to obtain dispersive estimates, which are later used to prove Strichartz estimates.

The main new difficulty compared with the Dirichlet case \cite{positive} is that the Airy zero parameter depends on the frequency through
\[
\beta(\iota,\theta)=\operatorname{arccot}(\cot(\iota)\theta^{-2/3}).
\]
Therefore, derivatives in $\eta$ produce additional terms which are absent in the Dirichlet analysis. The main technical part of the paper is to prove that these frequency-dependent Robin remainders remain perturbative in each dispersive regime. On the other hand, the parts of the proof which rely only on derivatives in variables other than $\eta$ remain unchanged, provided that $B_{\beta(\theta)}$ satisfies the same asymptotic estimates. This is established in proposition \ref{boundsb}.

We start in Section \ref{spectralfriedlander} by constructing the Green function through a spectral analysis of the Fourier transformed Friedlander model, and using this analysis to write the Green function as a sum of eigenmodes. This already shows the first difference with the Dirichlet case: the zeroes of the Robin function have an additional dependence on $\theta$, which is not present in the Dirichlet setting. Nevertheless, we are able to repeat a similar argument and obtain a Hilbert basis of eigenmodes in proposition \ref{robin}. We then obtain two equivalent forms of the Green function, after using an adapted Poisson formula, in Definitions \ref{green} and \ref{green2}.

For the function defined in Definition \ref{green2}, we obtain a structure very similar to equation (25) in \cite{positive}. The main difference is the function $B$.

In our case, the function $B$ depends on $\eta$ in two places: through its argument and through its subscript. In \cite{positive}, by contrast, it depends only on its argument. This additional dependence is the main issue that has to be addressed in order to replicate the Dirichlet argument.

In Section \ref{sectionphase}, we carry out a detailed analysis of this new term and its asymptotics. In particular, in proposition \ref{boundsb} we prove that, in the regime of interest, namely when $\alpha^{3/2}\eta h^{-1}\lesssim \eta h^{-1}$, the asymptotics of our new function $B_\beta$ are the same as those of the function $B$ in the Dirichlet case, as are its derivatives with respect to the argument. This implies that, whenever an argument requires only derivatives in variables other than $\eta$, the corresponding proof from the Dirichlet case translates directly to our setting, without requiring any further adaptation. However, when we differentiate $B_\beta$ with respect to $\eta$, new terms appear which are absent in the Dirichlet case. These terms are introduced in Definition \ref{restrobsnewnew}, in the parts of the proof where they occur, we give new proofs and verify that these terms remain controlled. The necessary bounds are proved in proposition \ref{restrobs}.

Continuing with the strategy of \cite{positive}, we consider the case $\gamma>h^{2/3-\varepsilon}$ in Section \ref{forlambda}. The goal of this section is to prove that our Green function satisfies the same dispersive estimates as in the Dirichlet case. As explained above, for the propositions which do not require derivatives with respect to $\eta$, the corresponding Dirichlet arguments also prove the Robin case, once proposition \ref{boundsb} has established the required asymptotics. Indeed, the only information about $B$ used in those proofs is its asymptotic behavior, and our function $B_{\beta(\eta/h)}$ satisfies the same type of estimates uniformly in $\eta/h$. This applies to proposition \ref{boundgreat}, parts (1), (2), and (3) of proposition \ref{disptransverse}, and parts (1), (2), (3), and (4) of proposition \ref{bigtangenctial}, which correspond respectively to the proofs of lemma 4 and propositions 9, 10, 4, 5, 6, and 7 (1) in \cite{positive}. By contrast, proposition \ref{disptransverse} (4) and proposition \ref{bigtangenctial} (5) require new proofs in the Robin case. For proposition \ref{prop1}, we prove the bound on the number of waves, while the fact that the contribution is negligible outside these numbers follows as in the Dirichlet case in the tangential regime.

After that, in Section \ref{sectiongall}, we prove the Strichartz estimates in the regime $\gamma<h^{2/3-\varepsilon}$, corresponding to the whispering-gallery modes. We use the techniques of \cite{Ivanovici2010Counterexamples} and \cite{positive}, but we provide proofs adapted to the Robin setting.

Finally, in Section \ref{strichartz}, we use the dispersive estimates obtained in Section \ref{forlambda} to prove the Strichartz estimates outside the gallery-mode regime.

To obtain the properties needed throughout the paper, we rely on the Airy function, as in related works such as \cite{OanaIvanovici2021DiscreteandContinuousDynamicalSystems}, \cite{ivanovici_dispersion_2014}, \cite{positive}, and \cite{me_taylor_boundary_nodate}. We recall some of its properties in Appendix A. These properties are classical and are mostly taken from \cite{airy}. We also introduce a new auxiliary function adapted to the Robin problem, which we study in detail in Appendix \ref{approbin}, providing the necessary proofs. The only work related to the study of this special function that we found is \cite{rus}. However, we were not able to check all details, since the available version is in Russian. It seems to contain a more detailed analysis of the localization of the zeroes of $\operatorname{Rob}_\iota$ than the one needed here, but it does not appear to prove the additional properties of the function used in this paper.
\subsection{Acknowledgments}
I would like to thank the guidance of my thesis advisor Fabrice Planchon for his recommendations and feedback on this paper, as well as Luca Fanelli and Olivier Lafitte for their thesis reports which allowed me to improve some parts of the redaction of this paper. I also thank Oana Ivanovici for some technical recommendations and for helping me to understand better some of the works this paper is based on.\\
The author was partially supported by the European Research Council (ERC) under the European Unions Horizon 2020 research and innovation program (gran agreement N0. 757996, ANADEL).\\
\section{Spectral analysis of the Friedlander model}
\label{spectralfriedlander}
We consider 
$$\Omega_2=\{(x,y)\in\mathbb{R}^2\mid x>0,\,y\in\mathbb{R}\},\quad
\Delta_F=\partial^{2}_{x}+(1+x)\partial^{2}_{y}.$$

After a Fourier transform in the $y$ variable, the operator $-\Delta_{F}$ is now
$-\partial^2_x+(1+x)\theta^2$.\\\\
\begin{defi}
	We define the differential operator from $\mathcal{C}^\infty(\mathbb{R_+})\cap L^2(\mathbb{R}_+)$ to $\mathcal{C}^\infty(\mathbb{R_+})$
	\begin{equation}
		D_\theta f(x)=-\partial^2_x f(x)+(1+x)\theta^2 f(x) .
	\end{equation}
\end{defi}
From now on, we assume $\theta\ge c$ where $c$ is a constant depending on $\iota$. This causes no loss because of the $hD_y\sim 1$ cutoff.\\
\begin{obs}
	Since later we will have the spectral localization $\theta \sim \frac{1}{h}$, we can assume $\theta$ larger than a constant (depending on $\iota$ which is fixed in the model) without altering the argument.
\end{obs}
\begin{defi}
	We define the (unique) Green's operator $T_{\iota,\theta}$ such that, for every $f\in\mathcal{C}^{\infty} (\mathbb{R_+})\cap L^2(\mathbb{R}_+)$, we have
	\begin{equation}
		D_\theta T_{\iota,\theta} f=f.
	\end{equation}
	subject to the boundary conditions
	\begin{align}
		&\cos(\iota)(T_{\iota,\theta} f)(0)-\sin(\iota)(T_{\iota,\theta} f)'(0)=0, \\
		&\lim_{x\rightarrow \infty}T_{\iota,\theta} f(x)=0.
	\end{align}
	\begin{obs}
		One may define a Green's kernel $G_{\iota,\theta}$
		such that, for every $x\in\mathbb{R_+}$, 
		\begin{equation}
			T_{\iota,\theta} f(x)=\int_{0}^{\infty}G_{\iota,\theta}(x,s)f(s)ds.
		\end{equation}
		Written in this form, we can extend the $T_{\iota,\theta}$ operator to $L^2(\mathbb{R_+})$.
	\end{obs}
	\begin{obs}
		If $f\in\mathcal{C}^{\infty} (\mathbb{R_+})\cap L^2(\mathbb{R}_+)$ is an eigenvector of the operator, then it is an eigenvector of $D_\theta$ which satisfies the corresponding boundary condition, but the converse is also true:
		if we extend $D_\theta$ to distributions, it follows that $f\in L^2(\mathbb{R}_+)$ is an eigenfunction of $D_\theta$ satisfying the Robin boundary condition if and only if $f$ is an eigenfunction of $T_{\iota,\theta}$.
	\end{obs}
\end{defi}
\begin{prop}
	For $\theta> c$ (with the constant depending on $\iota$), the $T_{\iota,\theta}$ operator is a positive, compact, self-adjoint operator
	on $L^2(\mathbb{R}_+)$.
\end{prop}
\begin{proof}
For $\theta> c$, the quadratic form associated with $D_\theta$ is strictly positive, where the possibly negative Robin contribution is controlled by the standard trace inequality. Moreover, by standard Sturm-Liouville theory, the operator $D_\theta$ is self-adjoint with compact resolvent, since the potential $V(x)=\theta^2(1+x)\rightarrow \infty$, $x\rightarrow \infty$. Therefore, $T_{\iota,\theta}$ is self-adjoint and compact.
\end{proof}
By the spectral theorem, we can conclude that the eigenfunctions form a Hilbert basis in $L^2(\mathbb{R}_+)$.

\begin{prop}\label{robin}
	A sequence of orthonormal eigenfunctions for the operator $D_\theta$ is
	\begin{equation}
		e_k(x,\theta,\iota)=\sqrt{\frac{2\pi |\theta|^\frac{2}{3}}{L'_{{\beta(\iota,\theta)}}(w_{{\beta(\iota,\theta),k}})}}Ai(\theta^{\frac{2}{3}}x-w_{{\beta(\iota,\theta),k}})
	\end{equation}
	with $\lambda_k(\theta)=\theta^2+w_{k,\beta(\iota,\theta)}\theta^{4/3}$ being the corresponding eigenvalues of $D_\theta$, and\\ $\beta(\iota,\theta)=\arccot(\cot(\iota)\theta^{-2/3})\in [0,\pi)$. This sequence of eigenfunctions forms a Hilbert basis for every $\theta\ge c$.
\end{prop}
\begin{obs}
	$-w_{{\beta},k}$ is the $k$-nth zero of $Rob_{\beta}(x):=\cos(\beta)Ai(x)-\sin(\beta)Ai'(x)$ (proposition \ref{zerosrob}), $L_\beta'(w_{k,\beta})=2\pi \int_{0}^{\infty} Ai^2(x-w_{\beta,k})dx$ (Lemma \ref{LemaL}).
\end{obs}
\begin{proof}
By writing the eigenvalue problem, we obtain an Airy-like function, so we obtain
	\begin{align}
		& f''(x)-(1+x)\theta^2f(x)=-\lambda f(x),\\
		&f(x)=CAi\left(\theta^{\frac{2}{3}}\left(x+\frac{\theta^2-\lambda}{\theta^2}\right)\right).
	\end{align}
	The $Bi$ component vanishes, since the $Bi$ part is not in $L^2(\mathbb{R_+})$, nor in $\mathcal{S}'(\mathbb{R})$, so it must vanish.\\
	Imposing the Robin condition gives
	\begin{align}
		&\cos(\iota)f(0)-\sin(\iota)f'(0)=0.\\
		&-\theta^{\frac{2}{3}}\sin(\iota)Ai'\left(\theta^{\frac{2}{3}}\left(\frac{\theta^2-\lambda}{\theta^2}\right)\right)+\cos(\iota)Ai\left(\theta^{\frac{2}{3}}\left(\frac{\theta^2-\lambda}{\theta^2}\right)\right)=0.
	\end{align}
	With $\beta(\iota,\theta)=\arccot(\cot(\iota)\theta^{-2/3})$, the preceding equation becomes
	\begin{equation}
		-\sin(\beta(\iota,\theta))Ai'\left(\theta^{\frac{2}{3}}\left(\frac{\theta^2-\lambda}{\theta^2}\right)\right)+\cos(\beta(\iota,\theta))Ai\left(\theta^{\frac{2}{3}}\left(\frac{\theta^2-\lambda}{\theta^2}\right)\right)=0.
	\end{equation}
	Thus $\theta^{\frac{2}{3}}\left(\frac{\theta^2-\lambda}{\theta^2}\right)=-w_{k,\beta(\iota,\theta)}$, that is $\lambda_k(\theta)=\theta^2+w_{{\beta(\iota,\theta),k}} \theta^{4/3}$, and
	\begin{equation}
		e_k(x,\theta,\iota)=Ai(\theta^{\frac{2}{3}}x-w_{{\beta(\iota,\theta),k}}).
	\end{equation}
	We now normalize using lemma \ref{LemaL}:
	\begin{equation}
		\int_{0}^{\infty} Ai(\theta^{\frac{2}{3}}x-w_{{\beta(\iota,\theta),k}})^2dx=\frac{L'_{{\beta(\iota,\theta)}}(w_{{\beta(\iota,\theta),k}})}{\theta^{\frac{2}{3}}2\pi}.
	\end{equation}
	\begin{equation}
		e_k(x,\theta,\iota)=\sqrt{\frac{2\pi |\theta|^\frac{2}{3}}{L'_{{\beta(\iota,\theta)}}(w_{{\beta(\iota,\theta),k}})}}Ai(\theta^{\frac{2}{3}}x-w_{{\beta(\iota,\theta),k}}).
	\end{equation}
	Hence, except when zero is an eigenvalue (which we excluded, since we assume that $\theta\ge c$), we know by applying the spectral theorem to the Green's operator that these functions form a Hilbert basis.
\end{proof}
This section motivates the study of the Airy functions and also the function that we call the Robin function, $Rob_\iota(x)=\cos(\iota)Ai(x)-\sin(\iota)Ai'(x)$, in the appendix.
\section{Constructing the parametrix}
We now construct the Green functions for the problem. \\
Consider the problem
\[ \begin{cases} 
	(\partial_t^2-\partial_x^2-(1+x)\theta^2)\tilde{u}(t,x,\theta) =0,& x>0 \\
	\tilde{u}(0,x,\theta)=\tilde{u_0}(x,\theta),	\partial_t \tilde{u}(0,x,
	\theta)=0\\
	\cos(\iota)\tilde{u}(t,x,\theta)-\sin(\iota) \partial_x  \tilde{u}(t,x,\theta)=0,& x=0.
\end{cases}
\]
We assume $\theta\ge c$, where $c$ is chosen so that $w_{\beta(\iota,c)}\ge 0$ with $\tilde{u}_0\in L^2(\mathbb{R}_+)$ compatible with the boundary condition, that is $\cos(\iota)\tilde{u}_0(0,\theta)-\sin(\iota) \partial_x  \tilde{u}_0(0,\theta)=0$ and $supp\{\tilde{u}_0\}\subset\{x\ge0\}$.\\
This is the original problem \eqref{baseprob} but having done Fourier transform in the $y$ variable.\\
Then, by the spectral decomposition, there exist coefficients $\tilde{u}_{0,k}(\theta)$ such that, for almost every $\theta>0$,
\begin{equation}
	\tilde{u}_0(x,\theta)=_{L^2}\sum_{k=1}^{\infty} \tilde{u}_{0,k}(\theta) e_k(x,\theta,\iota),
\end{equation}
where the sum converges in the sense of $L^2(\mathbb{R}_+)$.
Furthermore,\\ $\tilde{u}_{0,k}(\theta)=\int \tilde{u}_0(x,\theta)e_k(x,\theta,\iota)dx$.
The solution $\tilde{u}$ is therefore given by
\begin{align}
	&\tilde{u}(t,x,\theta)=_{L^2} \sum_{k=1}^{\infty} e_k(x,\theta,\iota)\left(\frac{e^{-ti\sqrt{\lambda_k(\theta)}} +e^{ti\sqrt{\lambda_k(\theta)}}}{2}\right)\tilde{u}_{0,k}(\theta)\\
	&=\sum_{k=1}^{\infty} \int e_k(x,\theta,\iota)\left(\frac{e^{-ti\sqrt{\lambda_k(\theta)}} +e^{ti\sqrt{\lambda_k(\theta)}}}{2}\right)\tilde{u}_0(a,\theta)e_k(a,\theta,\iota)da.
\end{align}
The object of interest is not $\tilde{u}$ itself, but $\psi_0(h^2 D_t^2)\psi_1(h\theta)\tilde{u}$. We use the following elementary identity.
\begin{lema}
	For every $\lambda\in\mathbb{C}$, if $\psi_0\in \mathcal{C}^\infty_0([1/2,3/2])$ then $\psi_0(h^2D_t^2) e^{\lambda\text{Id}}=\psi_0(-h^2\lambda^2) e^{\lambda\text{Id}}$. Equivalently, with the usual abuse of notation, we say $\psi_0(h^2D_t^2)e^{\lambda t}=\psi_0(-h^2\lambda^2) e^{\lambda t}$.\\
\end{lema}
With this consideration, we have
\begin{align}
	&\psi_0(h^2 D_t^2)\psi_1(h\theta)\tilde{u}(t,x,\theta)\\
	&=_{L^2} \sum_{k=1}^{\infty}\int  e_k(x,\theta,\iota)\psi_0(h^2 \lambda_k(\theta))\psi_1(h\theta)\left(\frac{e^{-ti\sqrt{\lambda_k(\theta)}} +e^{ti\sqrt{\lambda_k(\theta)}}}{2}\right)\tilde{u}_0(a,\theta)e_k(a,\theta,\iota)da.
\end{align}
After inserting the cutoff, the conditions $h^2\theta^{4/3}w_{{\beta(\iota,\theta),k}}\le \frac{3}{2}$ and $w_{{\beta(\iota,\theta),k}}\ge w_{k-1}\rightarrow_{k\rightarrow \infty} \infty$ imply that the series is actually a finite sum, which in turn means that the equality is not only an equality in $L^2$ but an equality almost everywhere and we may interchange the sum and the integral.\\
Using the second cutoff, we also have that $h\theta\ge \frac{1}{2}$ in addition to $h^2\theta^{4/3}w_{{\beta(\iota,\theta),k}}\le \frac{3}{2}$, which allows us to conclude that $\frac{w_{k-1}}{h^{2/3}} \le \frac{w_{{\beta(\iota,\theta),k}}}{h^{2/3}}\le 3^{2/3}2^{2/3}$. Also, since $k^{2/3}\lesssim w_{k}$, the series is not only a finite sum, but the number of terms is $\lesssim C h^{-1}$ uniformly in $\theta$.\\
\begin{align}
	&\psi_1(h\theta)\psi_0(h^2 D_t^2)\tilde{u}(t,x,\theta)\\
	&= \sum_{k\lesssim h^{-1}}\int  e_k(x,\theta,\iota)\psi_0(h^2 \lambda_k(\theta))\psi_1(h\theta)\left(\frac{e^{-ti\sqrt{\lambda_k(\theta)}} +e^{ti\sqrt{\lambda_k(\theta)}}}{2}\right)\tilde{u}_0(a,\theta)e_k(a,\theta,\iota)da.
\end{align}
Also, since our estimates are local in time, only the full sum is relevant for the estimates. The contribution of finitely many low modes individually is the same as the free wave, so it will not be problematic, Remark \ref{lowmodes} provides a more detailed justification. In particular, we will consider only $k\ge 2$.
With all these considerations, we can finally consider the Green's function for our problem. We will consider the Green's function for the half wave operator 
\begin{equation}
	\tilde{G}^{\pm}_h(t,x,\theta,a)= \sum_{k=2}^{\infty}  e_k(x,\theta,\iota)\psi_0(h^2 \lambda_k(\theta))\psi_1(h\theta){e^{\pm ti\sqrt{\lambda_k(\theta)}} }e_k(a,\theta,\iota).
\end{equation}
\begin{obs}
	Considering the Green's function of the half wave operator instead of the full wave equation is technically convenient, since we need the complex exponential  for stationary phase arguments, and it also allows us to recover easily our case $\psi_1(h\theta)\psi_0(h^2 D_t^2)\tilde{u}(t,x,\theta)=\frac{1}{2}\int( G^{-}_h(t,x,\theta,a)+G^{+}_h(t,x,\theta,a))\tilde{u}_0(a,\theta)da$ for $\tilde{u_1}= 0$. Furthermore, when $\tilde{u_1}\ne 0$ and $u_0=0$, we have
	\begin{align}
		&\psi_1(h\theta)\psi_0(h^2 D_t^2)\tilde{u}(t,x,\theta)\\
		&= \sum_{k\lesssim h^{-1}}\int  e_k(x,\theta,\iota)\frac{\psi_0(h^2 \lambda_k(\theta))}{\sqrt{\lambda_k(\theta)}}\psi_1(h\theta)\left(\frac{e^{-ti\sqrt{\lambda_k(\theta)}} -e^{ti\sqrt{\lambda_k(\theta)}}}{2}\right)\tilde{u}_1(a,\theta)e_k(a,\theta,\iota)da.
	\end{align}
	Thus, the contribution is governed by  $\frac{h}{2}\int( G^{-}_h(t,x,\theta,a)-G^{+}_h(t,x,\theta,a))\tilde{u}_1(a,\theta)$ after replacing $\psi_0$ by another cut-off around $1$ with the same properties (specifically $\frac{\psi_0}{\sqrt{|-|}}\in\mathcal{C}^{\infty}(\mathbb{R})$).
\end{obs}
We now do the following decomposition: Define $\tilde{\psi}_2\in \mathcal{C}^\infty_0([0,3/2])\cap \mathbbm{1}([0,1])$. We define $\psi_2(\xi):=\tilde{\psi_2}(\xi)-\tilde{\psi_2}(2\xi)$, hence, we can define the following:
\begin{defi}
	\label{DefiGgamatilde}
	\begin{align}
		\tilde{G}_{h,\gamma}^{\pm}(t,x,\theta,a):=  \sum_{k=2}^{\infty}  \psi_0(h^2 \lambda_k(\theta)) 	\psi_1(h\theta)\psi_2\left(\frac{w_{{\beta(\iota,\theta),k}}}{\theta^{2/3}\gamma}\right)e_k(x,\theta,\iota)e_k(a,\theta,\iota){e^{\pm it\sqrt{\lambda_k(\theta)}} },
	\end{align}
\end{defi}
with $\psi_2\in \mathcal{C}^\infty_0(1/2,3/2)$. Thus we have a localization on a dyadic annulus.\\\\
With this construction, we have the important property that\\ $\tilde{G}_{h}^{\pm}(t,x,\theta,a)=\sum_{m=1}^{\infty}\tilde{G}_{h,2^{-m}}^{\pm}(t,x,\theta,a)$ and the sum is finite, in particular $h^{2/3}\lesssim \gamma \lesssim 1$.\\\\
\begin{obs}
	For $\gamma$ not small, the corresponding rays undergo at most one reflection. Thus, we can assume that $\gamma\le \gamma_0\le {2^{-7}}$. \\
	In particular, if $\gamma\ge 2^{-7}$, we have
	\begin{align}
		&\frac{w_{{\beta(\iota,\theta),k}}}{\theta^{2/3}\gamma}\le \frac{2^7w_{{\beta(\iota,\theta),k}}}{\theta^{2/3}}\le \frac{2^7h^2\theta^{4/3}w_{{\beta(\iota,\theta),k}}}{h^2\theta^{2}}\le \frac{2^7(1+2^{-9}-h^2\theta^2)}{\theta^2h^2}\\
		& \le\frac{2^7(1+2^{-9}-1+2^{-10})}{1-2^{-10}}\le \frac{\frac{1}{4}+\frac{1}{8}}{1-{2^{-10}}}\le \frac{\frac{3}{8}}{\frac{3}{4}}\le 1.
	\end{align}
	Thus, if $\gamma\ge 2^{-7}$, for every $h,\theta,w_{{\beta(\iota,\theta),k}}$, it follows that $\psi_0(h^2 \lambda_k(\theta)) 	\psi_1(h\theta)\psi_2\left(\frac{w_{{\beta(\iota,\theta),k}}}{\theta^{2/3}\gamma}\right)=0$.
\end{obs}
We now return to the original problem $(\partial_t^2-\partial_x^2-(1+x)\partial_y^2)u(t,x,y) =0$ by taking the inverse Fourier transform (where we can interchange the sum and the integral because the sum is finite and there is an upper bound to the number of terms which do not depend on $\theta$), we obtain the Green's function which will be the main object of study:
\begin{defi}
	\label{green}
	\begin{align}
		&\label{defih}
		G^\pm_h(t,x,y,a):=\mathcal{F}^{-1}_\theta(\tilde{G}^\pm_h(t,x,\theta,a))\\
		&=\frac{1}{(2\pi)} \sum_{k\ge 2}  \int \psi_0(h^2 \lambda_k(\theta)) 	\psi_1(h\theta)e_k(x,\theta,\iota)e_k(a,\theta,\iota){e^{iy\theta}e^{\pm it\sqrt{\lambda_k(\theta)}} } d\theta.
	\end{align}
	\begin{align}
		\label{Ggammadefi}
		G^\pm_{h,\gamma}(t,x,y,a):=\frac{1}{(2\pi)} \sum_{k\ge 2}  \int \psi_0(h^2 \lambda_k(\theta)) 	\psi_1(h\theta)\psi_2\left(\frac{w_{{\beta(\iota,\theta),k}}}{\theta^{2/3}\gamma}\right)e_k(x,\theta,\iota)e_k(a,\theta,\iota){e^{iy\theta}e^{\pm it\sqrt{\lambda_k(\theta)}} } d\theta.
	\end{align}
	Moreover, we have 
	$G_{h}(t,x,y,a)=\sum_{m=1}^{\infty}{G}_{h,2^{-m}}^{\pm}(t,x,y,a)$ and the sum is finite, in particular $h^{2/3}\lesssim \gamma \le 2^{-7}$.
\end{defi}
\begin{obs}
We are studying high frequencies, that is $h$ small. But for $h$ small enough, and since $h\sim \frac{1}{\theta}$, we have $\beta({\iota,\theta})$ tends to $\frac{\pi}{2}$, that is, the higher the frequency, the closer we are to Neumann and the further we are from Dirichlet ($\beta(\iota,\theta)\rightarrow \frac{\pi}{2}$ as $\theta\rightarrow \infty$). That is, the Neumann is the model towards which the Robin problem tends naturally rather than Dirichlet.
\end{obs}
\begin{obs}
	\label{lowmodes}
	As we stated before, we can eliminate any quantity of fixed modes without altering the result (that's why we can start in $k\ge 2$). 
	Specifying a little bit more, consider, the Green's function for the mode $k$,
	$$G_{h,k}(t,x,y,a)=\frac{1}{2\pi h}  \int \psi_0(h^2 \lambda_k(\eta/h)) 	\psi_1(\eta)e_k(x,\eta/h,\iota)e_k(a,\eta/h,\iota){e^{iy\eta/h}e^{\pm it\sqrt{\lambda_k(\eta/h)}} } d\eta.$$
	For this oscillatory integral, since $k$ is fixed, $e_k(\cdot,\eta/h,\iota)$ are symbols, and $|e_k(\cdot,\eta/h,\iota)|\lesssim h^{-1/3}$. Furthermore, the phase $\phi_h(\eta)=\eta\sqrt{1+w_{\beta(\frac{\eta}{h}),k}(h/\eta)^{2/3}}+\frac{y}{t}\eta$, satisfies $|\phi_h''(\eta)|\gtrsim h^{2/3}$. (It is the same phase as in lemma \ref{otherlemmagallery}). We get then that $\sup_{x,y,a}|G_{h,k}(t,x,y,a)|\lesssim \frac{h^{2/3}}{h^2}\frac{1}{|t|^{1/2}}\lesssim \frac{h^{1/2}}{h^2}\frac{1}{|t|^{1/2}}$. This is the classical free-wave dispersion, hence any isolated mode gives the same Strichartz as the free wave.
\end{obs}
We now use the Poisson formula in order to get a new expression for the Green's function. We first decompose the function in the eigenmodes by recalling that
\begin{align}
	&e_k(x,\theta,\iota)=\frac{\sqrt{2\pi}\theta^{1/3}}{\sqrt{L'_{\beta(\iota,\theta)}(w_{{\beta(\iota,\theta),k}})}}Ai(\theta^{2/3}x-w_{{\beta(\iota,\theta),k}})\\
	&=\frac{\sqrt{2\pi}^{-1}\theta^{1/3}}{\sqrt{L'_{\beta(\iota,\theta)}(w_{{\beta(\iota,\theta),k}})}}\int e^{i\left(\frac{\Xi_1^3}{3}+\Xi_1 \theta^{2/3}x-\Xi_1 w_{{\beta(\iota,\theta),k}} \right)}d\Xi_1.
\end{align}
We also prove here an initial bound on $G_{h,\gamma}$ which will be useful in the gallery modes regime.
\begin{lema}
	\label{grossboundgallmodes}
	We have the following bound: for every $t,x,y,a$, we obtain
	\begin{align}
		|G_{h,\gamma}(t,x,y,a)|\lesssim \frac{\sqrt{\gamma}}{h^2}.
	\end{align}
\end{lema}
\begin{proof}
	Lemma 3.5 in \cite{ivanovici_dispersion_2014} proves that for $L\ge 1$
	\begin{align}
		\sup_z \left(\sum_{1\le k\lesssim L} |Ai(z-w_{k})|^2k^{-1/3}\right)\lesssim L^{1/3}.
	\end{align}
	We now prove a similar statement for our Robin setting:
	\begin{align}
		\sup_z \left(\sum_{1\le k \lesssim L} |Ai(z-w_{{\beta(\iota,\eta/h),k}})|^2k^{-1/3}\right)\lesssim L^{1/3}.
	\end{align}
In order to do so, we prove the following general lemma:
\begin{lema}
	\label{lemmaaux}
	Let $\{x_n\}_{n\in\mathbb{N}}$ be a positive strictly increasing sequence of real numbers such that, for every $n$, $w_{n-1} \le x_n\le w_{n}$. For every $L\ge 1$ the following follows:
		\begin{align}
	\sup_{z\in\mathbb{R}}\left(\sum_{1\le n \lesssim L} |Ai(z-x_n)|^2n^{-1/3}\right)\lesssim L^{1/3}.
	\end{align}
\end{lema}
\begin{proof}[Proof of lemma \ref{lemmaaux}]
	The expression $w_{n-1} \le x_n\le w_{n}$ implies that for every $n\in\mathbb{N}$, we have $x_{n+2}-x_n\sim n^{-1/3}$ and $x_n \sim n^{2/3}$. We divide the sum into odd and even numbers.
			\begin{align}
		\left(\sum_{1\le k \lesssim L} |Ai(z-x_k)|^2k^{-1/3}\right)=\sum_{\substack{1\le k \lesssim L\\ k \; odd}} |Ai(z-x_k)|^2k^{-1/3}+\sum_{\substack{1\le k \lesssim L\\ k \; even}} |Ai(z-x_k)|^2k^{-1/3}.
	\end{align}
We will prove the property for $k$ odd, for $k$ even is analogue.\\
Let $z\in\mathbb{R}$, $k\in\mathbb{N}$. For every $s\in [x_{k}, x_{k+2}]$,
\begin{align}
	|z-x_k|^{1/2}=|z-s+s-x_k|^{1/2}\ge |z-s|^{1/2}-|s-x_k|^{1/2}\ge |z-s|^{1/2}-|x_{k+2}-x_k|^{1/2}\ge |z-s|^{1/2}-C.
\end{align}
For a certain constant $C\ge 0$. Thus we obtain
	$1+|z-s|^{1/2}\le (1+C) (1+|z-x_k|^{1/2})$. 
We obtain then
\begin{align}
	\frac{k^{-1/3}}{1+|z-x_k|^{1/2}}\lesssim\int_{x_k}^{x_{k+2}} \frac{ds}{1+|z-s|^{1/2}}.
\end{align}
Now, by using that $|Ai(x)|\lesssim \frac{1}{1+|x|^{1/4}}$, we conclude that
\begin{align}
	\sum_{\substack{1\le k \lesssim L\\ k \; odd}} |Ai(z-x_k)|^2k^{-1/3}\lesssim \int_{x_1}^{x_{L+3}} \frac{ds}{1+|z-s|^{1/2}}\lesssim x_{L+3}^{1/2}\lesssim L^{1/3}
\end{align}
\end{proof}
With all this information, we can conclude
	\begin{align}
	&|G^\pm_{h,\gamma}(t,x,y,a)|:=|\frac{1}{(2\pi)} \sum_{2\le k \lesssim \frac{\gamma^{3/2}
		}{h}}  \int \psi_0(h^2 \lambda_k(\theta)) 	\psi_1(h\theta)\psi_2\left(\frac{w_{{\beta(\iota,\theta),k}}}{\theta^{2/3}\gamma}\right)e_k(x,\theta,\iota)e_k(a,\theta,\iota){e^{iy\theta}e^{\pm it\sqrt{\lambda_k(\theta)}} } d\theta|\lesssim\\
	&\frac{1}{h}\sum_{2\le k \lesssim \frac{\gamma^{3/2}
		}{h}}  \int 	\psi_1(\eta)\frac{1}{h^{2/3}k^{1/3}}|Ai(\eta x/h-w_{{\beta(\iota,\eta/h),k}})Ai(\eta a/h-w_{{\beta(\iota,\eta/h),k}})|d\eta\lesssim\\
	& \frac{1}{h^{5/3}}\int \psi_1(\eta)\sup_{z\in\mathbb{R}}\left(\sum_{2\le k\lesssim \frac{\gamma^{3/2}}{h}} |Ai(z-w_{{\beta(\iota,\eta/h),k}})|^2k^{-1/3}\right)d\eta\lesssim \frac{1}{h^{5/3}}\left(\frac{\gamma^{3/2}}{h}\right)^{1/3}=\frac{\sqrt{\gamma}}{h^2}.
	\end{align}
\end{proof}
Since $ w_{{\beta(\iota,\theta),k}}<w_k$ and we start the sum at $k=2$, we can add a function $\psi_3\in \mathcal{C}^\infty(3/2,\infty)\cap \mathbbm{1}((2,\infty))$ without altering anything, having
\begin{align}
	G^\pm_{h,\gamma}(t,x,y,a):=\frac{1}{(2\pi)^2}\int \sum_{k\ge 2}\frac{ \theta^{2/3}}{{L'_{\beta(\iota,\theta)}(w_{{\beta(\iota,\theta),k}})}} \psi_0(h^2 \lambda_k(\theta)) 	\psi_1(h\theta)\psi_2\left(\frac{w_{{\beta(\iota,\theta),k}}}{\theta^{2/3}\gamma}\right)\psi_3(w_{{\beta(\iota,\theta),k}})\\
	e^{i\left(\frac{\Xi_1^3+\Xi_2^3}{3}+\theta^{2/3}(\Xi_1 x+\Xi_2 a)- w_{{\beta(\iota,\theta),k}}(\Xi_1+\Xi_2) \pm t\sqrt{\lambda_k(\theta)}+y\theta\right)}d\theta d\Xi_1 d\Xi_2.
\end{align}
After applying all the cut-offs $\psi_0,\psi_1,\psi_2,\psi_3$, we have $w_{{\beta(\iota,\theta),k}}\gtrsim 1$, so we apply our version of the Poisson formula \ref{AiryPoisson}.
\begin{align}
	G^\pm_{h,\gamma}(t,x,y,a):=\frac{1}{(2\pi)^3}\sum_{N\in \mathbb{Z}}{ \theta^{2/3}}\int \psi_0(h^2 \theta^2+w h^2\theta^{4/3}) 	\psi_1(h\theta)\psi_2\left(\frac{w}{\theta^{2/3}\gamma}\right)\psi_3(w)\\
	\cdot e^{i\left(\frac{\Xi_1^3+\Xi_2^3}{3}+\theta^{2/3}(\Xi_1 x+\Xi_2 a)- w(\Xi_1 +\Xi_2) \pm t\sqrt{h^2 \theta^2+w h^2\theta^{4/3}}+y\theta-NL_{\beta(\iota,\theta)}(w)\right)}d\theta d\Xi_1 d\Xi_2dw.
\end{align}
Now, we perform the following changes of variables: 
\begin{align}
	\Xi_1=\theta^{1/3}\sigma, \Xi_2=\theta^{1/3}s, h\theta=\eta, \alpha=\frac{w}{\theta^{2/3}}=w \frac{h^{2/3}}{\eta^{2/3}},
\end{align}
\begin{align}
	G^\pm_{h,\gamma}(t,x,y,a):=\frac{1}{(2\pi)^3 h^3}\sum_{N\in \mathbb{Z}}\int \psi_0(\eta^2(1+\alpha)) 	\eta^2\psi_1(\eta)\psi_2\left(\frac{\alpha}{\gamma}\right)\psi_3\left(\alpha \eta^{2/3}h^{-2/3}\right)\\
	e^{\frac{i\eta}{h}\left(\frac{\sigma^3+s^3}{3}+(\sigma x+s a)- \alpha(\sigma +s ) \pm t\sqrt{1+\alpha}+y\right)-iNL_{\beta(\iota,\eta/h)}(\alpha\eta^{2/3}h^{-2/3})}d\eta d\sigma ds d\alpha.
\end{align}
Or, in a more compact manner, (and, since now $\psi_0$ is redundant),
\begin{align}
	{G}_{h,\gamma}(t,x,y,a)=\frac{1}{(2\pi h)^3}\sum_{N\in\mathbb{Z}}\int e^{\frac{i}{h}\phi_{N,a}(\boldsymbol{t},\sigma,s,\alpha,\eta,h)}\eta^2\psi_1(\eta)\psi_2(\alpha/\gamma)\psi_3(\eta^{2/3}\alpha/h^{2/3})dsd\sigma d\alpha d\eta,\\
	\label{gammapoisson}
\end{align}
with
\begin{align}
	&\label{firstphase}
	\phi_{N,a}(\boldsymbol{t},\sigma,s,\alpha,\eta,h)\\
	&=\eta y+\frac{\eta\sigma^3}{3}+\eta\sigma(x-\alpha)+\frac{\eta s^3}{3}+\eta s(a-\alpha)+\eta t\sqrt{1+\alpha}-NhL_{\beta(\eta/h)}((\frac{h}{\eta})^{-2/3}\alpha),
\end{align}
which leads to the following definition
\begin{defi}
	\label{green2}
	We define $V_{h,\gamma,N}(\boldsymbol{t},a)$ as 
	\begin{equation}
		\label{Vdef}
		V_{h,\gamma,N}(\boldsymbol{t},a)=\frac{1}{(2\pi h)^3}\int e^{\frac{i}{h}\phi_{N,a}(\boldsymbol{t},\sigma,s,\alpha,\eta,h)}\eta^2\psi_1(\eta)\psi_2(\alpha/\gamma)\psi_3(\eta^{2/3}\alpha/h^{2/3})dsd\sigma d\alpha d\eta.
	\end{equation}
	With this definition,
	\begin{align}
		\label{writtingG}
		{G}_{h,\gamma}(t,x,y,a)=\sum_{N\in\mathbb{Z}} V_{h,\gamma,N}(t,x,y,a).
	\end{align}
\end{defi}
\begin{obs}
	Since $b=\cot(\iota)$ is fixed, we make the following abuse of notation $\beta({\theta})=\arccot(b{\theta^{-2/3}})=\beta(\iota,{\theta})$.
\end{obs}
\section{Studying the phase}
\label{sectionphase}
We have the following phase
\begin{align}
	&\label{phasedef}
	\phi_{N,a}(\boldsymbol{t},\sigma,s,\alpha,\eta,h)=\eta \left(y+\frac{\sigma^3}{3}+\sigma (x-\alpha)+\frac{s^3}{3}+s(a-\alpha)+t\sqrt{1+\alpha}\right)-NhL_{\beta(\frac{\eta}{h})}(\alpha (\frac{\eta}{h})^{2/3})\\
	&	\label{phasedef2}
	=\eta \left(y+\frac{\sigma^3}{3}+\sigma (x-\alpha)+\frac{s^3}{3}+s(a-\alpha)+t\sqrt{1+\alpha}-\frac{4}{3}N\alpha^{3/2}\right)-NhB_{\beta(\frac{\eta}{h})}(\alpha^{3/2} (\frac{\eta}{h})).
\end{align}
\subsection{The function $B_\beta$}
We now define
\begin{defi}
	\label{defib}
	The function $B_\beta$ is defined in the following manner:
	\begin{equation}
		L_\beta(w)=\frac{4}{3}w^{3/2}+B_{\beta}(w^{3/2}).
	\end{equation}
\end{defi}
\begin{cor}
	\label{L'}
	\begin{equation}
		L'_\beta(w)={2w^{1/2}}(1+\frac{3}{4}B'_\beta(w^{3/2})).
	\end{equation}
\end{cor}
\begin{obs}
	This definition of $B_{\beta}$ is not ideal for the Robin case but allows us to maintain coherence with the Dirichlet result.
\end{obs}
We now introduce some auxiliary statements used to study the function $B_{\beta}$.
We note here the following properties, immediate consequences of the properties of $\arccot$:
\begin{prop}
	\label{propbeta}
	$\beta'(\theta)=\frac{2}{3}\frac{b}{\theta^{5/3}+\theta^{1/3}b^2}$, $\cos^2(\beta(\theta))=\frac{b^2\theta^{-4/3}}{1+b^2\theta^{-4/3}}=\frac{b^2}{b^2+\theta^{4/3}}$,
	$\sin^2(\beta(\theta))=\frac{1}{1+b^2\theta^{-4/3}}=\frac{\theta^{4/3}}{b^2+\theta^{4/3}}$, $\sin\cos(\beta(\theta))=\frac{b\theta^{-2/3}}{1+b^2\theta^{-4/3}}=\frac{b\theta^{2/3}}{b^2+\theta^{4/3}}$.
\end{prop}
\begin{defi}\label{stableO1}
	We say that a real smooth function $f$ defined around a neighborhood of zero is a stable big $O$ of $x$ in zero, denoted $f=O^s_0$ if and only if for every $n\in\mathbb{N}$, $f^{(n)}$ is a big $O$ of $1$ in zero ($f^{(n)}=O_0(1)$). In particular, if $f$ is analytic around zero, $f=O^s_0$.\\
	The definition of $f=\tilde{O}^s_0$ follows (for every $n\in\mathbb{N}$, $f^{(n)}=\tilde{O}_0(1)$).
\end{defi}

We also make the definition in several variables (we will use it in two variables).
\begin{defi}\label{stableO2}
	Let $f$ be a smooth function of two variables $x,y$ defined around a neighborhood of zero. We say that $f$ is a stable big  $O$ of $x,y$ in zero, denoted $f(x,y)=O^s_0(x,y)$ if and only if, for every $\alpha\in\mathbb{N}^2$, there exist constants $c_\alpha,K_\alpha>0$, such that, if $(x,y)\in B(0,c_\alpha)$, then $|\partial_{\alpha}f(x,y)|\le K_\alpha$.\\
	We say that $f(x,y)=\tilde{O}^s_0(x,y)$ iff for every $\alpha\in\mathbb{N}^2$, there exist constants $c_\alpha,K_\alpha,k_\alpha>0$, such that, if $(x,y)\in B(0,c_\alpha)$, then $k_\alpha\le|\partial_{\alpha}f(x,y)|\le K_\alpha$.\\
\end{defi}
We shall use the following elementary stability property:\\
\begin{prop}
	Let $f_1,f_2$ be functions, $v,w,s,p>0$, such that $f_1(x,y)=x^{-v}\tilde{O}^s_0(x^{-s},y^{-p})$, $f_2(x,y)=x^{-v}y^{-w}\tilde{O}^s_0(x^{-s},y^{-p})$. Then, for every $N,M\in \mathbb{N}$, we have\\
$\partial^N_x f_1(x,y)=x^{-v-N}\tilde{O}^s_0(x^{-s},y^{-p})$, $\partial^M_y\partial^N_x f_2(x,y)=y^{-w-M}x^{-v-N}\tilde{O}^s_0(x^{-s},y^{-p})$. Thus, stable $\tilde{O}$ terms behave as constants under differentiation.
\end{prop}
\begin{prop}
	\label{boundsb}
	\begin{align}
		B'_{\beta(\theta)}(w)=\tilde{O}_0^s(\theta^{-2/3}b,w^{-2/3})w^{-4/3}\theta^{-2/3}b+\tilde{O}_0^s(\theta^{-2/3}b,w^{-2/3})w^{-2}.
	\end{align}
	In particular, for $w\le k \theta$, for some constant $k$, we have 
	\begin{align}
		\label{grossboundsb}
		B'_{\beta(\theta)}(w)=\tilde{O}^\infty(w^{-2}).
	\end{align}
	Here, the $\tilde{O}^\infty(w^{-2})$ denotes a function which depends on $\theta$, but the constants in the definition of $\tilde{O}^\infty$ do not depend on $\theta$.
\end{prop}
\begin{proof}
	First, by definition,
	\begin{align}
		\frac{3}{4}B'_{\beta(\theta)}(w)=\frac{L'_{\beta({\theta})}(w^{2/3})}{2w^{1/3}}-1=\frac{\cos^2(\beta(\theta))+\sin^2(\beta(\theta))w^{2/3}}{4\pi \rho^2_{\beta(\theta)}(w^{2/3})w^{1/3}}-\frac{4\pi \rho^2_{\beta(\theta)}(w^{2/3})w^{1/3}}{4\pi \rho^2_{\beta(\theta)}(w^{2/3})w^{1/3}}.
	\end{align}
	By using the asymptotic expansion of the phase and dividing everything by $\sin^2(\beta(\theta))$ (which we know is not zero), we obtain
	\begin{align}
		&\frac{3}{4}B'_{\beta(\theta)}(w)=\frac{-2b\theta^{-2/3}\frac{1}{24}w^{-2/3}-\tilde{O}^{\infty}(w^{-4/3})-b^2\theta^{-4/3}\tilde{O}^{\infty}(w^{-2})}{b^2\theta^{-4/3}(1+\tilde{O}^{\infty}(w^{-2}))+w^{2/3}(1+\tilde{O}^{\infty}(w^{-2}))+b\theta^{-2/3}\frac{1}{12}w^{-2/3}}\\
		&=\frac{\tilde{O}_0^s(\theta^{-2/3}b,w^{-2/3})w^{-4/3}\theta^{-2/3}b+\tilde{O}_0^s(\theta^{-2/3}b,w^{-2/3})w^{-2}}{b^2\theta^{-4/3}w^{-2/3}(1+\tilde{O}^{\infty}(w^{-2}))+1+\tilde{O}^{\infty}(w^{-2})+b\theta^{-2/3}\frac{1}{12}w^{-4/3}}\\
		&=\frac{\tilde{O}_0^s(\theta^{-2/3}b,w^{-2/3})w^{-4/3}\theta^{-2/3}b+\tilde{O}_0^s(\theta^{-2/3}b,w^{-2/3})w^{-2}}{1+O_0^s(\theta^{-2/3}b,w^{-2/3})b\theta^{-2/3}w^{-2/3}}\\
		&=\left(\tilde{O}_0^s(\theta^{-2/3}b,w^{-2/3})w^{-4/3}\theta^{-2/3}b+\tilde{O}_0^s(\theta^{-2/3}b,w^{-2/3})w^{-2}\right)\tilde{O}^s_0(O_0^s(\theta^{-2/3}b,w^{-2/3})b\theta^{-2/3}w^{-2/3}).\\
		&B'_{\beta(\theta)}(w)=\tilde{O}_0^s(\theta^{-2/3}b,w^{-2/3})w^{-4/3}\theta^{-2/3}b+\tilde{O}_0^s(\theta^{-2/3}b,w^{-2/3})w^{-2}.
	\end{align}
\end{proof}

\subsubsection{Derivatives of the Function $B_\beta$} \label{boundingb}
With this, we next study the derivatives of the phase with respect to the variables involved in the integral (which will be relevant since we want to study the dispersion by stationary phase methods).\\
In particular, we focus on the new terms involved in the Robin case which do not appear in the Dirichlet case and why these terms are small.
By \ref{partialalpha}, we know that
\begin{align}
	\partial_\beta B_\beta(w)=\frac{2\left(1+\frac{3}{4}B_\beta'(w)\right)}{\cos^2(\beta)w^{-1/3}+w^{1/3}\sin^2(\beta)}.
\end{align}
By using this expression, we obtain
\begin{prop}
	\label{superasimptotics}
	\begin{align}
		\partial_\theta B_{\beta(\theta)}(w)=w^{-1/3}b\theta^{-5/3}\tilde{O}^s_0(\theta^{-4/3}b^2,w^{-2/3}).
	\end{align}
\end{prop}
\begin{nota}
	This study of the function $B_{\beta(\theta)}(w)$ is especially useful in the regime $\gamma>h^{\frac{2}{3}-\varepsilon}$, therefore, we will use the notation for the large parameter in that section $\lambda_{\gamma,h}=\frac{\gamma^{3/2}}{h}$.
\end{nota}
We now consider the following remainders.
\begin{defi}\label{restrobsnewnew}
	\begin{align}
		&Resrob_1(\alpha,\eta,h)=\frac{d}{d\eta}\left[B_{\beta(\eta/h)}(\alpha^{3/2}\eta h^{-1})\right]-\lambda_{\gamma,h}\frac{\alpha^{3/2}}{\gamma^{3/2}}B'_{\beta(\eta/h)}(\eta\lambda_{\gamma,h}\frac{\alpha^{3/2}}{\gamma^{3/2}})\\
		&Resrob_2(\alpha,\eta,h)=\frac{d^2}{d\eta^2}\left[B_{\beta(\eta/h)}(\alpha^{3/2}\eta h^{-1})\right]-\frac{\alpha^{3}}{\gamma^{3}}\lambda_{\gamma,h}^2B''_{\beta(\eta/h)}(\lambda_{\gamma,h}\eta\frac{\alpha^{3/2}}{\gamma^{3/2}}).\\
		&Resrob_3(\alpha,\eta,h)=\frac{d^3}{d\eta^3}\left[B_{\beta(\eta/h)}(\alpha^{3/2}\eta h^{-1})\right]-\frac{\alpha^{9/2}}{\gamma^{9/2}}\lambda_{\gamma,h}^3B'''_{\beta(\eta/h)}(\lambda_{\gamma,h}\eta\frac{\alpha^{3/2}}{\gamma^{3/2}}).\\
	\end{align}
\end{defi}
\begin{prop}
	\label{restrobs}
	We have the following bounds for the remainders: $|Resrob_i(\alpha,\eta,h)|\lesssim \frac{h^{2/3}}{\lambda_{\gamma,h}^{1/3}}$, for $i=1,2,3$, $|\partial_\alpha Resrob_1(\alpha,\eta,h)|\lesssim \lambda_{\gamma,h}^{-1}$, $|\partial_\alpha Resrob_2(\alpha,\eta,h)|\lesssim \lambda_{\gamma,h}^{-1}$ and $|\partial^2_\alpha Resrob_1(\alpha,\eta,h)|\lesssim \lambda_{\gamma,h}^{-5/3}h^{-2/3}$.
\end{prop}
\begin{obs}
	The  bounds $|Resrob_i(\alpha,\eta,h)|\lesssim \frac{h^{2/3}}{\lambda_{\gamma,h}^{1/3}}$, are expected to hold for every  $i\in \mathbb{N}$, but we only need it up to order $3$.
\end{obs}
\begin{proof}
	We compute the derivatives:
	\begin{align}
		&\frac{d B_{\beta(\eta/h)}(\alpha^{3/2}\eta h^{-1})}{d\eta}=\lambda_{\gamma,h}\frac{\alpha^{3/2}}{\gamma^{3/2}}B'_{\beta(\eta/h)}(\eta\lambda_{\gamma,h}\frac{\alpha^{3/2}}{\gamma^{3/2}})+\frac{1}{h}\partial_\theta B_{\beta(\theta)}(\eta\lambda_{\gamma,h}\frac{\alpha^{3/2}}{\gamma^{3/2}})|_{\theta=\frac{\eta}{h}}\\
		&=\lambda_{\gamma,h}\frac{\alpha^{3/2}}{\gamma^{3/2}}B'_{\beta(\eta/h)}(\eta\lambda_{\gamma,h}\frac{\alpha^{3/2}}{\gamma^{3/2}})+\frac{1}{h}\eta^{-1/3}\lambda_{\gamma,h}^{-1/3}\frac{\alpha^{-1/2}}{\gamma^{-1/2}}b\frac{\eta^{-5/3}}{h^{-5/3}}\tilde{O}^s_0(\frac{\eta^{-4/3}}{h^{-4/3}},(\eta\lambda_{\gamma,h}\frac{\alpha^{3/2}}{\gamma^{3/2}})^{-2/3})\\
		&=\lambda_{\gamma,h}\frac{\alpha^{3/2}}{\gamma^{3/2}}B'_{\beta(\eta/h)}(\eta\lambda_{\gamma,h}\frac{\alpha^{3/2}}{\gamma^{3/2}})+Resrob_1(\alpha,\eta,h).
	\end{align}
	With 
	\begin{align}
		|Resrob_1(\alpha,\eta,h)|\lesssim \frac{1}{h}\eta^{-1/3}\lambda_{\gamma,h}^{-1/3}\frac{\alpha^{-1/2}}{\gamma^{-1/2}}b\frac{\eta^{-5/3}}{h^{-5/3}}\lesssim \frac{h^{2/3}}{\lambda_{\gamma,h}^{1/3}}.
	\end{align}
	Furthermore, we have
	\begin{align}
		&\partial_\alpha Resrob_1(\alpha,\eta,h)=\frac{1}{h}\eta^{-1/3}\lambda_{\gamma,h}^{-1/3}\frac{\alpha^{-1/2}\alpha^{-1}}{\gamma^{-1/2}}b\frac{\eta^{-5/3}}{h^{-5/3}}\tilde{O}^s_0(\frac{\eta^{-4/3}}{h^{-4/3}},(\eta\lambda_{\gamma,h}\frac{\alpha^{3/2}}{\gamma^{3/2}})^{-2/3})\\
		&=\frac{h^{2/3}}{\gamma}\eta^{-1/3}\lambda_{\gamma,h}^{-1/3}\frac{\alpha^{-1/2}\alpha^{-1}}{\gamma^{-1/2}\gamma^{-1}}b\frac{\eta^{-5/3}}{h^{-5/3}}\tilde{O}^s_0(\frac{\eta^{-4/3}}{h^{-4/3}},(\eta\lambda_{\gamma,h}\frac{\alpha^{3/2}}{\gamma^{3/2}})^{-2/3}).
	\end{align}
	With this, we conclude that $|\partial_\alpha Resrob_1(\alpha,\eta,h)|\lesssim \frac{1}{\lambda_{\gamma,h}}$.\\
	The remaining derivatives are estimated by the same argument:
	\begin{align}
	&\frac{d^2 B_{\beta(\eta/h)}(\alpha^{3/2}\eta h^{-1})}{d\eta^2}=\lambda_{\gamma,h}^2\frac{\alpha^{3}}{\gamma^{3}}B''_{\beta(\eta/h)}(\eta\lambda_{\gamma,h}\frac{\alpha^{3/2}}{\gamma^{3/2}})\\
	&+\lambda_{\gamma,h}\frac{\alpha^{3/2}}{\gamma^{3/2}}\frac{1}{h}\partial_\theta B'_{\beta(\theta)}(w)|_{w=\eta\lambda_{\gamma,h}\frac{\alpha^{3/2}}{\gamma^{3/2}},\theta=\frac{\eta}{h}}+\partial_\eta Resrob_1(\alpha,\eta,h).\\
	\end{align}
	A direct computation shows that $|\partial_\eta Resrob_1(\alpha,\eta,h)|\lesssim \frac{h^{2/3}}{\lambda_{\gamma,h}^{1/3}}$, and for the other term, we just use \ref{superasimptotics}, getting 
	\begin{align}
		\partial_\theta B'_{\beta(\theta)}(w)=\tilde{O}_0^s(\theta^{-2/3}b,w^{-2/3})w^{-4/3}\theta^{-5/3}b+\tilde{O}_0^s(\theta^{-2/3}b,w^{-2/3})w^{-2}b\theta^{-5/3}.
	\end{align}
	Thus, for $w$ large (that is, $\lambda_{\gamma,h}$ large in our case), we obtain $|\partial_\theta B'_{\beta(\theta)}(w)|\lesssim b\theta^{-5/3}w^{-4/3}$. Applying $w=\eta\lambda_{\gamma,h}\frac{\alpha^{3/2}}{\gamma^{3/2}},\theta=\frac{\eta}{h}$ and the $h^{-1}$, we get that this term is $\lesssim \frac{h^{2/3}}{\lambda_{\gamma,h}^{1/3}}$ so the sum of both terms must be $\lesssim \frac{h^{2/3}}{\lambda_{\gamma,h}^{1/3}}$, concluding the bound.
\end{proof}
\section{The case $\gamma>h^{\frac{2}{3}-\varepsilon}$}\label{forlambda}
Set $\gamma>h^{\frac{2}{3}-\varepsilon}$. Defining $\lambda_{\gamma,h}=\frac{\gamma^{3/2}}{h}$, we have $\lambda_{\gamma,h}>h^{-\frac{3}{2}\varepsilon}\gg 1$, which means that we can use $\lambda_{\gamma,h}$ as the large parameter in stationary phase arguments. In particular $B'_{\beta(\eta/h)}(\alpha^{3/2}\eta h^{-1})\lesssim \lambda_{\gamma,h}^{-2}$, and $|Resrob_i(\alpha,\eta,h)|\lesssim \frac{h^{2/3}}{\lambda_{\gamma,h}^{1/3}}$ (by proposition \ref{restrobs}).
\begin{obs}
	\label{ignoreverything}
	The study of the phase shows that, as long as no derivatives with respect to $\eta$ are taken, the integrals have the same structure as those in \cite{positive} with a different but equivalent symbol. This is because the asymptotics of the new $B_{\beta(\theta)}$ are the same as the $B$ in \cite{positive} when we do not differentiate with respect to $\eta$. This means that, whenever a step of the Dirichlet proof depends only on the symbolic bounds and asymptotic structure of the oscillatory integral, we use the same argument after verifying that the Robin quantities satisfy the corresponding bounds. The only points at which the Robin problem differs from the Dirichlet one are those where derivatives fall on the frequency-dependent parameter $\beta(\iota,\eta/h)$. These derivatives produce the remainders ${Resrob}_i$, which are not present in \cite{positive} and are estimated in proposition \ref{restrobs}. Thus, in the sequel, references to \cite{positive} should be understood as references to the specific stationary phase or integration-by-parts argument used there, while the Robin-specific error terms are controlled separately whenever they occur.
\end{obs}
\begin{obs}
	In the regime $\gamma>h^{\frac{2}{3}-\varepsilon}$, having $O(\lambda_{\gamma,h}^{-\infty})$ implies $O(h^{\infty})$.
\end{obs}
In particular, the proof of lemma 4 in \cite{positive} does not involve derivatives in $\eta$. This allows us to repeat the proof of the corresponding lemma in \cite{positive} without modification, as detailed in remark \ref{ignoreverything}.
\begin{prop}[Lemma 4 in \cite{positive}]
	\label{boundgreat}
	The sum in \eqref{gammapoisson} is microlocally finite, that is, there exists a constant $K$ not depending on $h,\gamma,N$ such that, 
	\begin{equation}
		\sum_{K|N|>|t|\gamma^{-1/2}} V_{h,\gamma,N}=O(h^\infty).
	\end{equation}
	Thus only finitely many $N$ contribute to $G_{\gamma,h}$.\\
	Also, if $\gamma< a/2$ or $x>2\gamma$, $G_{\gamma,h}(\boldsymbol{t},a)=O(h^\infty)$.
\end{prop}
\begin{obs}
	Since we are working locally in time, $V_{h,\gamma,N}$, is only relevant for $\gamma ^{1/2}|N|\le K t_{max} \lesssim 1$.
\end{obs}
In our case, the analogue to proposition 1 in \cite{positive}, does not translate directly to our case (since it involves derivatives of $\eta$). The statement remains true, but the proof needs to control the new $\eta$-derivative terms coming from $\beta(\eta/h)$.\\
\begin{prop}
	\label{prop1}
	We define
	\begin{align}
		\label{firstdefiN}\mathcal{N}_{h,\gamma}(\boldsymbol{t})=\{N\in \mathbb{Z}: \text{there exists at least one critical point of } \phi_{N,a}(\boldsymbol{t},\cdot)\}.
	\end{align}
	that is
	\begin{align}
		\label{seconddefiN}\mathcal{N}_{h,\gamma}(\boldsymbol{t})=\{N\in \mathbb{Z}: \exists (\boldsymbol{\sigma},\alpha,\eta)\in \mathbb{R}^3\times \left[\frac{1}{2},\frac{3}{2}\right] : \nabla_{(\sigma,s,\alpha,\eta)}\phi_{N,a}(\boldsymbol{t},\sigma,s,\alpha,\eta)=0\}.
	\end{align}\\\\
	Now, we define 
	\begin{align}
		&\label{firslinen1}\mathcal{N}_{1,h,\gamma}(\boldsymbol{t})=\{N\in \mathbb{Z}: \exists (t',x',y')\in \mathbb{R}^3 : N\in\mathcal{N}(t',x',y') \text{ and }\\
		& \label{secondlinen2}|T(\gamma,\boldsymbol{t}')-T(\gamma,\boldsymbol{t})|\le 1,|X(\gamma,\boldsymbol{t}')-X(\gamma,\boldsymbol{t})|<1, |Y(\gamma,\boldsymbol{t}')-Y(\gamma,\boldsymbol{t})|<1\}.
	\end{align} 
	\begin{align}
		\mathcal{N}_{1,h,\gamma}(\boldsymbol{t})=\bigcup_{\{(t',x',y')\in \mathbb{R}^3 |T(\gamma,\boldsymbol{t}')-T(\gamma,\boldsymbol{t})|\le 1,|X(\gamma,\boldsymbol{t}')-X(\gamma,\boldsymbol{t})|<1, |Y(\gamma,\boldsymbol{t}')-Y(\gamma,\boldsymbol{t})|<1\}}\mathcal{N}_{h,\gamma}(t',x',y').\\
		\label{thirdlinen1}
	\end{align}
	Where $ t=\sqrt{\gamma}T(\gamma,\boldsymbol{t}), x=\gamma X(\gamma,\boldsymbol{t}),y+t\sqrt{1+\gamma}=\gamma^{3/2} Y(\gamma,\boldsymbol{t})$.
	Then, we have a bound for the distance between two elements $N_1,N_2$ of $\mathcal{N}_{1,h,\gamma}(\boldsymbol{t})$ and consequently a bound on the cardinal of $\mathcal{N}_{1,h,\gamma}(\boldsymbol{t})$:
	\begin{align}
		|N_1-N_2|\lesssim C+\frac{T(\gamma,\boldsymbol{t})}{\lambda_{\gamma,h}^2}=C+\frac{t}{\sqrt{\gamma}\lambda_{\gamma,h}^2},\\
		|\mathcal{N}_{1,h,\gamma}(\boldsymbol{t})|\lesssim C+\frac{T(\gamma,\boldsymbol{t})}{\lambda_{\gamma,h}^2}=C+\frac{t}{\sqrt{\gamma}\lambda_{\gamma,h}^2}.
	\end{align}
	Also, we have
	\begin{align}
		\label{finiten}
		\sum_{N\notin \mathcal{N}_{1,h,\gamma}(\boldsymbol{t})}V_{h,N,\gamma}(t,x,y)=O(h^{\infty}).
	\end{align}
	\begin{obs}
	This proposition tells us that the only relevant terms  $V_{h,\gamma,N}$ are those with, or close to a critical point, in a sense of closeness which depends on gamma.
	\end{obs}
\end{prop}
\begin{proof}
	Thus, a stationary point satisfies the system
	\begin{align}
		&\alpha= \sigma^2+x,\\
		& \alpha= s^2+a,\\
		&2N\alpha^{1/2}(1+\frac{3}{4}B'_{\beta(\eta/h)})(\alpha^{3/2}\frac{\eta}{h})=-\sigma-s+\frac{t}{2\sqrt{1+\alpha}} \label{goodeq},\\ 
		&y+\frac{\sigma^3}{3}+\sigma (x-\alpha)+\frac{s^3}{3}+s(a-\alpha)+t\sqrt{1+\alpha}\\
		&=\frac{4}{3}N\alpha^{3/2}\left(1+\frac{3}{4}B'_{\beta(\eta/h)}(\alpha^{3/2}\eta h^{-1})\right)+NhResrob_1(\alpha,\eta,h). \label{bigeq}
	\end{align}
	Let $N_1,N_2\in \mathcal{N}_{1,h,\gamma}(t,x,y)$. By definition, there exist two points $(t_1,x_1,y_1)$ and $(t_2,x_2,y_2)$, and ${N}_1\in \mathcal{N}_{h,\gamma}(t_1,x_1,y_1)$, and $N_2\in\mathcal{N}_{h,\gamma}(t_2,x_2,y_2)$, with $||\textbf{T}(\gamma,\boldsymbol{t}_1)-\textbf{T}(\gamma,\boldsymbol{t})||_{\infty}\le 1$.\\
	Thus, by definition, there exist $\alpha_i,\eta_i,\sigma_i,s_i$, $i=1,2$ such that, for 
	$(t_i,x_i,y_i,N_i)$, $(\alpha_i,\eta_i,\sigma_i,s_i)$ is a stationary point.\\
	Also, we note that, because of \eqref{goodeq}, we have
	\begin{align}
		&2(N_2-N_1)\\ &=-\frac{3}{4}\left(N_2(B'_{\beta(\eta_2/h)})(\alpha_2^{3/2}\frac{\eta_2}{h})-N_1(B'_{\beta(\eta_1/h)})(\alpha_1^{3/2}\frac{\eta_1}{h})\right)\label{theB}\\
		&-\frac{\sigma_2+s_2}{\alpha_2^{1/2}}+\frac{\sigma_1+s_1}{\alpha_1^{1/2}}+\label{gammanads}\\
		&\frac{t_2}{2\sqrt{1+\alpha_2}\alpha_2^{1/2}}-\frac{t_1}{2\sqrt{1+\alpha_1}\alpha_1^{1/2}}\label{thehardone}.\\
	\end{align}
	To estimate \eqref{theB}, we use \eqref{grossboundsb}, plus the fact that we are in the regime $h<\lambda^{-1}$ (and $\alpha\sim\gamma$) to get that
	\begin{equation}
		|B'_{\beta(\eta_i/h)}(\alpha_i^{3/2}\frac{\eta_i}{h})|\lesssim \frac{h^2}{\alpha_i^3}\lesssim \lambda_{\gamma,h}^{-2}.
	\end{equation}
	Thus, we obtain that \eqref{theB} is bounded by $\frac{N_i}{\lambda_{\gamma,h}^2}\lesssim\frac{t}{\sqrt{\gamma}\lambda_{\gamma,h}^2}$
	and the desired bound follows.\\
	By definition $|t_i-t|\le \gamma^{1/2}$, and by proposition \ref{boundgreat}, $N_i\gamma^{1/2}\lesssim t_i$, which justifies $\frac{N_i}{\lambda_{\gamma,h}^2}\lesssim\frac{t}{\sqrt{\gamma}\lambda_{\gamma,h}^2}$.\\\\
	Since in the stationary points, we have $|\sigma_i|\le\sqrt{\alpha_i-x_i}\le \alpha_i^{1/2}$, and $|s_i|\le\sqrt{\alpha_i-a}\le \alpha_i^{1/2}$, then $\left|\frac{\sigma_i+s_i}{\alpha_i^{1/2}}\right|\lesssim 1$, so the bound holds.\\\\
	It remains to study \eqref{thehardone}.\\
	We compute as follows:
	\begin{align}
		\left|\frac{t_2}{2\sqrt{1+\alpha_2}\alpha_2^{1/2}}-\frac{t_1}{2\sqrt{1+\alpha_1}\alpha_1^{1/2}}\right|\lesssim (1+\frac{t_1}{\gamma^{1/2}}\frac{|\alpha_1-\alpha_2|}{\gamma}).
	\end{align}
	We now have to prove that $1+\frac{t_1}{\gamma^{1/2}}\frac{|\alpha_1-\alpha_2|}{\gamma}=O(1)$.\\
	Since it suffices to prove an $O(1)$ bound, and $\frac{|\alpha_1-\alpha_2|}{\gamma}\lesssim 1$, we can assume that $\frac{t_1}{\gamma}$ is large. Since $|t_1-t|<\gamma$ and $|t_2-t|\le\sqrt{\gamma}$, $t_1\sim t_2\sim t$.\\
	If we put together the equations \eqref{bigeq} and \eqref{goodeq}, we obtain 
	\begin{align}
		&y_i+\frac{\sigma_i^3}{3}+\sigma_i(x_i-\alpha_i)+\frac{s_i^3}{3}+s_i(a-\alpha_i)+t_i\sqrt{1+\alpha_i}\\
		&=\frac{2}{3}\alpha_i(-\sigma_i-s_i+\frac{t_i}{2\sqrt{1+\alpha_i}})+NhResrob_1(\alpha_i,\eta_i,h).
	\end{align}
	It is useful to rewrite $x,y,t$ into $X,Y,T$, this gives us
	\begin{align}
		Y(\gamma,\boldsymbol{t}_j)+\frac{\sigma_j^3}{3\gamma^{3/2}}+\sigma_j(\gamma X(\gamma,\boldsymbol{t}_j)-\alpha_j)\gamma^{-3/2}+\frac{s^3_j}{3\gamma^{3/2}}+s_j(a-\alpha_j)\gamma^{-3/2}+\frac{2\alpha_j}{3\gamma^{3/2}}(\sigma_j+s_j)\\
		=\gamma^{-1}T(\gamma,\boldsymbol{t}_j)\left(\frac{\alpha_j}{3\sqrt{1+\alpha_j}}+\frac{(\gamma-\alpha_i)}{\sqrt{1+\alpha_j}+\sqrt{1+\gamma}}\right)+N_j\lambda_{\gamma,h}^{-1}Resrob_1(\alpha_i,\eta_i,h).
	\end{align}
Using proposition \ref{boundgreat}, we note that $0<X<2$. Thus we obtain 
	\begin{align}
		\frac{Y_j}{T_j}=F_\gamma(\frac{\alpha_j}{\gamma})+O(\frac{1}{T_j})+N_j\lambda_{\gamma,h}^{-1}Resrob_1(\alpha_j,\eta_j,h), \:\:\;\;\;\;\;\ F_\gamma(A)=\frac{A}{3\sqrt{1+\gamma A}}-\frac{A-1}{\sqrt{1+\gamma A}+\sqrt{1+\gamma}}.
	\end{align}
	We have $|N_i\lambda_{\gamma,h}^{-1}Resrob_1(\alpha_i,\eta_i,h)|\lesssim N_i h^{2/3} \lambda_{\gamma,h}^{-4/3}\lesssim \gamma^{-1/2} h^{2/3}\lambda_{\gamma,h}^{-4/3} \lesssim \lambda_{\gamma,h}^{-1/3}h^{-1/3}h^{2/3}\lambda_{\gamma,h}^{-4/3}=h^{1/3}\lambda_{\gamma,h}^{-5/3}<<1$.\\
	It remains to use the following observation. The function $F$, for $\gamma$ small enough and in $1/4<A<4$ is decreasing, in particular $\partial_AF_\gamma(A)\le (1+a)\gamma-\frac{1}{6}$. Thus we obtain
	\begin{align}
		\left|\left(\frac{Y_2}{T_2}-\frac{Y_1}{T_1}\right)+O(\frac{1}{T_1})+O(\frac{1}{T_2})\right|=\left|(\frac{\alpha_2-\alpha_1}{\gamma})\int_{0}^{1} \partial_AF_\gamma(\frac{\alpha_1}{\gamma}+w(\frac{\alpha_2-\alpha_1}{\gamma}))dw\right|\ge \left|(\frac{\alpha_2-\alpha_1}{\gamma})\frac{1}{12}\right|.
	\end{align}
	This implies that
	\begin{align}
		\frac{t_1}{\gamma^{1/2}}\frac{|\alpha_1-\alpha_2|}{\gamma}\lesssim |Y_2-Y_1|+Y_1|1-\frac{T_2}{T_1}|+O(\frac{T_2}{T_1})+O(1)=O(1).
	\end{align}
	\begin{obs}
		The only difference with the Dirichlet problem is that in Robin we have to make sure that the term $Resrob_1$ is harmless, and that $B_\beta(\eta h^{-1})$ has the correct asymptotics.
	\end{obs}
\end{proof}
The proof that the contribution outside $\mathcal{N}_1$ is negligible depends on whether the wave is transverse or tangential, we therefore postpone it until after introducing the appropriate terms in propositions \ref{neg foir tan} and \ref{neg for trans}.\\\\
Furthermore, we can also improve the bound in \eqref{boundgreat} to state that, not only $|t|\gtrsim \sqrt{\gamma}|N|$, but $|t|\sim \sqrt{\gamma}|N|$.
\begin{cor}\label{TsimN}
	Let $|N|\ge 1$. There exists a constant $\mathcal{K}$ such that,
	\begin{equation}
		\sum_{\mathcal{K}|N|\le|t|\gamma^{-1/2}} V_{h,\gamma,N}=O(h^\infty).
	\end{equation}
	That is, the sum in $G_{\gamma,h}$ is only relevant for $|N|\sim|t|\gamma^{-1/2}$, $T\sim N$.
\end{cor}
\begin{proof}
	We  prove that, if $200|N|\le |t|\gamma^{-1/2}$, then, $N\notin \mathcal{N}_{1,h,\gamma}(\boldsymbol{t})$.\\
	If $N\in \mathcal{N}_{h,\gamma}(\boldsymbol{t}')$, then there is a stationary point $(\boldsymbol{\sigma},\alpha,\eta)$, where $\alpha=\sigma^2+x'$, $\alpha=s^2+a$, $2N\alpha^{1/2}(1+\frac{3}{4}B'_{\beta(\eta/h)})(\alpha^{3/2}\frac{\eta}{h})=-\sigma-s+\frac{t'}{2\sqrt{1+\alpha}}$. The first two equations imply that $|\sigma|,|s|\le 3\sqrt{\gamma}$. The third one (assuming $\lambda_{\gamma,h}$ large enough, which is only to assume that $h$ is small enough in this regime), implies that $|N|\ge \frac{-12\sqrt{\gamma}+2\frac{t'}{4}}{\alpha^{1/2}}\ge -12+\frac{|t'|}{10\sqrt{\gamma}}$, that is: $10|N|+120\ge \frac{|t'|}{\sqrt{\gamma}}$. Since $|N|\ge 1$, we have $130|N|\ge \frac{|t'|}{\sqrt{\gamma}}$. Since $\frac{|t'-t|}{\sqrt{\gamma}}\le 1$, $130|N|\ge \frac{|t|}{\sqrt{\gamma}}-1$. Thus, $131|N|\ge \frac{|t|}{\sqrt{\gamma}}$. This concludes the proof. In order for $N$ to be in $\mathcal{N}_{1,h,\gamma}(\boldsymbol{t})$, we need $131|N|\ge \frac{|t|}{\sqrt{\gamma}}$.
\end{proof}
\begin{obs}
	We could also prove the corollary directly with the tools of proposition \ref{boundgreat}, but since we are not replicating that proof here,it is simpler in our case to use the argument above.
\end{obs}
It now makes sense to distinguish the cases of $\gamma\sim a$ which for geometrical reasons, we  refer to as tangential waves, and $\gamma \gg a$ which we  refer to as transverse waves. (By proposition \ref{boundgreat}, we do not have to consider the case $\gamma<<a$).
\subsection{Transverse waves $\gamma\gg a$}
We consider now the case $\gamma>4a$.\\
For this case, we introduce the following change of variables
\begin{align}
	\label{transchange}
	x=\gamma{X},\;\;\alpha=\gamma A, \;\; t=\sqrt{\gamma}{T}, \;\; s=\sqrt{\gamma}{S}, \;\;\;\; \sigma=\sqrt{\gamma}\Sigma,\;\; y+t\sqrt{1+\gamma}=\gamma^{3/2} {Y}.
\end{align}
We define 
\begin{defi}
	\label{change1}
	\begin{align}
		\label{defiw}
		&{W}_{h,N,\gamma,a}({T},{X},{Y})=V_{h,N,\gamma,a}(t(\gamma,{T},{X},{Y}),x(\gamma,{T},{X},{Y}),y(\gamma,{T},{X},{Y}))\\
		&=\frac{\gamma^2}{(2\pi h)^3}\int e^{i\lambda_{\gamma,h}  \psi_{N,a,h}(\boldsymbol{ T},\gamma,{\Sigma},{S},{A},\eta)} \eta^2 \psi_1(\eta)\psi_3(\lambda_{\gamma,h}^{2/3}\eta^{2/3}{A})\psi_2({A}) d{S}d{\Sigma}d{A}d\eta,
	\end{align}
	with 
	\begin{align}
		\label{defipsi}
		& \psi_{N,a,h}(\boldsymbol{ T},\gamma,{\Sigma},{S},{A},\eta)=\gamma^{-3/2}\phi_{N,a}(\boldsymbol{t}(\gamma, \boldsymbol{T}),\sqrt{\gamma}\Sigma,\sqrt{\gamma} S,\gamma A,\eta,h).
	\end{align}
	Or by writing it explicitly, we obtain
	\begin{align}
		\label{defipsiB}
		& \psi_{N,a,h}(\boldsymbol{ T},\gamma,{\Sigma},{S},{A},\eta)\\
		&=\eta \left({Y}+\frac{{\Sigma}^3}{3}+\frac{{S}^3}{3}+{X}{\Sigma}+\frac{a}{\gamma}{S}-({S}+{\Sigma}){A}+{T}\frac{1}{\gamma}(\sqrt{1+\gamma{A}}-\sqrt{1+\gamma})-N\frac{4}{3}{A}^{3/2}\right)\\
		&-\frac{N}{\lambda_{\gamma,h}}B_{\beta(\eta/h)}(\eta \lambda_{\gamma,h}{A}^{3/2}).
	\end{align}
\end{defi}
\begin{obs}
	Sometimes we will make the abuse of notation of suppressing the dependence on $\gamma$ on the phase and doing $ \psi_{N,a,h}(\boldsymbol{ T},\gamma,{\Sigma},{S},{A},\eta)= \psi_{N,a,h}(\boldsymbol{ T},{\Sigma},{S},{A},\eta)$.
\end{obs}
\begin{obs}
	Inspecting the phase, we see that $\psi_2({A})=0$ if ${A}\ge 6$ and also if ${A}\le 1/20$, which means that for $\lambda_{\gamma,h}^{2/3}>100$, we obtain $\psi_3(\lambda_{\gamma,h}^{2/3}\eta^{2/3}{A})=1$, and since we are considering $\lambda$ a large parameter, $\psi_3$ is equal to $1$ on the relevant support and remove it.\\
	Also since $\partial_{{S}} \psi_{N,a,h}=\eta ({S}^2+\frac{a}{\gamma}-{A})$ and $\partial_{{\Sigma}} \psi_{N,a,h}=\eta ({\Sigma}^2+{X}-{A})$, if either $|{S}|>\sqrt{6}$ or $|{\Sigma}|>\sqrt{6}$, the phase is non-stationary, and ${W}=O(\lambda_{\gamma,h}^{-\infty})$. Therefore we can add a cut-off $\psi_4$ on ${S}$ and ${\Sigma}$ and assume that ${S}$ and ${\Sigma}$ are less than $3$.\\
\end{obs}
Our goal is to prove the following dispersive estimates, which are the same as in the Dirichlet case these are the analogues of propositions 8,9,10,11 in \cite{positive} on the Dirichlet case.
\begin{prop}
	\label{disptransverse}
	We have the following dispersion:
	\begin{enumerate}
		\item For $|T|\le 1$, we have
		\begin{align}
			|({G}_{\gamma,h}(\boldsymbol{t}(\gamma,{\boldsymbol{T}})))|\lesssim \gamma^{-5/2} \lambda_{\gamma,h}^2 \inf\{1,\left(\frac{1}{\lambda_{\gamma,h} {|T|}}\right)^{1/2}\}.
		\end{align}
		\item For $1\le |T|\le 9$, we have
		\begin{align}
			\sum_{|N|\le 2} |W_{h,N,\gamma,a}(\boldsymbol{T})|\lesssim \gamma^{-5/2} \lambda_{\gamma,h}^2 \inf\{1,\left(\frac{1}{\lambda_{\gamma,h} {|T|}}\right)^{1/3}\}.
		\end{align}
		\item For $9\le |T|\le \lambda_{\gamma,h}^2$, we have
		\begin{align}
			|W_{h,N,\gamma,a}(\boldsymbol{|T|})|\lesssim \frac{\gamma^{-5/2}\lambda_{\gamma,h}^{5/3}}{\sqrt{T}}.
		\end{align}
		\item For $|T|\ge  \lambda_{\gamma,h}^2$, we have
		\begin{align}
			\label{imptrans}
			|W_{h,N,\gamma,a}(\boldsymbol{T})|\lesssim \frac{\lambda_{\gamma,h}^{13/6}}{\gamma^{5/2}|N|}=\frac{\gamma^2}{h^3 \lambda_{\gamma,h}^{5/6}|N|}.
		\end{align}
	\end{enumerate}
\end{prop}
\begin{prop}[\eqref{finiten} for the transverse waves]
	\label{neg for trans}
	\begin{align}
		\sum_{N\notin \mathcal{N}_{1,h,\gamma}(\boldsymbol{t}(\gamma,\boldsymbol{{T}))}}{W}_{h,N,\gamma,a}(\boldsymbol{{T}})=O(h^{\infty}).
	\end{align}
\end{prop}
\begin{proof}
	Our objective is to prove that, if $N\notin\mathcal{N}_{1,h,\gamma}(\boldsymbol{t}(\gamma,\boldsymbol{{T}))}$, then, ${W}_{h,N,\gamma,a}(\boldsymbol{{T}})=O(h^{\infty})$ because, since we know that the sum is only relevant for $|N|\lesssim\gamma^{-1/2}\lesssim h^{-1/3}$, (The first inequality by proposition \ref{boundgreat}, the second is a condition of our regime of study $\gamma>h^{2/3-\varepsilon}$), then, ${W}_{h,N,\gamma,a}(\boldsymbol{{T}})=O(h^{\infty})$ implies that
	$\sum_{N\notin \mathcal{N}_{1,h,\gamma}(\boldsymbol{t}(\gamma,\boldsymbol{{T}))}}{W}_{h,N,\gamma,a}(\boldsymbol{{T}})=O(h^{\infty})+h^{-1/3}O(h^{\infty})=O(h^\infty)$.\\
	We  prove the smallness of ${W}_{h,N,\gamma,a}$ by proving that, if $S$ is stationary, then \\$|\partial_\Sigma \psi_{N,a,h}(\boldsymbol{ T},\gamma,{\Sigma},{S_c},{A},\eta)|+|\partial_A \psi_{N,a,h}(\boldsymbol{ T},\gamma,{\Sigma},{S_c},{A},\eta)|+|\partial_\eta \psi_{N,a,h}(\boldsymbol{ T},\gamma,{\Sigma},{S_c},{A},\eta)|\ge \frac{1}{10}$ and thus we can apply non-stationary phase theorem.\\
	We proceed then to obtain the lower bound announced.\\
	Let $N\notin\mathcal{N}_{1,h,\gamma}(\boldsymbol{t}(\gamma,\boldsymbol{{T}))}$, this means that, for every $\boldsymbol{T}'\in B(\boldsymbol{T},1)$,\\ $\nabla_{(\Sigma,S,A,\eta)} \psi_{N,a,h}(\boldsymbol{ T'},\gamma,{\Sigma},{S},{A},\eta)\ne 0$.
	Now, we note that there are exactly two critical points for $S$ and they only depend on $a,\gamma,A$, $S_{\pm c}(A):=\pm\sqrt{A-\frac{a}{\gamma}}$, and, since $\frac{a}{\gamma}\le \frac{1}{4}$ and $A\ge\frac{1}{2}$ we have $|S_{\pm c}(A)|\ge |\frac{1}{2}|$.\\
	\begin{obs}
		The second derivative in $S$ which is $2\eta S$ is large at the critical point and the stationary points are well separated from each other, in particular, these are precisely the facts that fail in the tangential case, which is why Airy analysis is needed.
	\end{obs}
	Define $W_{h,N,\gamma,a,\varepsilon}(\boldsymbol{T})$, that is, the result of applying stationary phase and separating the two $S_{c,\varepsilon}$, $\varepsilon=\pm 1$. $W_{h,N,\gamma,a,\varepsilon}(\boldsymbol{T})$ have as phase 
	\begin{align}
		\label{defipsieps}
		& \psi_{N,a,h,\varepsilon}(\boldsymbol{ T},\gamma,{\Sigma},{A},\eta)=\eta \Big({Y}+\frac{{\Sigma}^3}{3}+{X}{\Sigma}+\varepsilon\frac{2}{3}\left(A-\frac{a}{\gamma}\right)^{3/2}-{\Sigma}{A}+{T}\frac{1}{\gamma}(\sqrt{1+\gamma{A}}-\sqrt{1+\gamma})\\
		&-N\frac{4}{3}{A}^{3/2}\Big)-\frac{N}{\lambda_{\gamma,h}}B_{\beta(\eta/h)}(\eta \lambda_{\gamma,h}{A}^{3/2}).
	\end{align}
Define now the integral curves corresponding to $ \psi_{N,a,h,\varepsilon}$
	\begin{align}
		&\tilde{X}_\varepsilon(\Sigma,A):=A-\Sigma^2,\\
		&\tilde{T}_\varepsilon(\Sigma,A,\eta):=2\sqrt{1+\gamma A}\left(\Sigma+\varepsilon\left(A-\frac{a}{\gamma}\right)^{1/2}+2N\sqrt{A}\left(1+\frac{3}{4}B'_{\beta(\eta/h)}(\eta \lambda_{\gamma,h}A^{3/2})\right)\right),\\
		&\tilde{Y}_\varepsilon(\Sigma,A,\eta):=-2\sqrt{1+\gamma A}\frac{(\Sigma+\varepsilon(A-\frac{a}{\gamma})^{1/2})(A-1)}{\sqrt{1+\gamma A}+\sqrt{1+\gamma}}+\frac{2}{3}\Sigma^3-\varepsilon\left(A-\frac{a}{\gamma}\right)^{3/2}\\
		&+4N\sqrt{A}\left(1+\frac{3}{4}B'_{\beta(\eta/h)}(\eta\lambda_{\gamma,h}A^{3/2})\right)\left(\frac{A}{3}-\frac{(A-1)\sqrt{1+\gamma A}}{\sqrt{1+\gamma A}+\sqrt{1+\gamma}}\right)-N\lambda_{\gamma,h}^{-1}Resrob_1(\gamma A,\eta,h).
	\end{align}
	Now, if we take $\boldsymbol{T}'$, we have 
	\begin{align}
		\label{partialintegralcurves}
		&\partial_\Sigma \psi_{N,a,h,\varepsilon}(\boldsymbol{ T'},\gamma,{\Sigma},{A},\eta)=X'-\tilde{X}_\varepsilon(\Sigma,A),\\
		&\partial_A \psi_{N,a,h,\varepsilon}(\boldsymbol{ T'},\gamma,{\Sigma},{A},\eta)=\frac{T'-\tilde{T}_\varepsilon(\Sigma,A,\eta)}{2\sqrt{1+\gamma A}},\\
		&\partial_\eta \psi_{N,a,h,\varepsilon}(\boldsymbol{ T'},\gamma,{\Sigma},{A},\eta)=\frac{(T'-\tilde{T}_\varepsilon(\Sigma,A,\eta))(A-1)}{\sqrt{1+\gamma A}+\sqrt{1+\gamma}}+Y'-\tilde{Y}_\varepsilon(\Sigma,A,\eta)+\Sigma(X'-\tilde{X}_\varepsilon(\Sigma,A,\eta)).\\
	\end{align}
	This implies that in the coordinates $\boldsymbol{T}'=\tilde{\boldsymbol{T}}_\varepsilon(\Sigma,A,\eta)$, $(\Sigma,S_{\varepsilon c}(A),A,\eta)$ is a stationary point. In other words, 
	$N\in\mathcal{N}_{h,\gamma}(\boldsymbol{t}(\tilde{\boldsymbol{T}}_\varepsilon(\Sigma,A,\eta)))$, but since $N\notin\mathcal{N}_{1,h,\gamma}(\boldsymbol{t}(a,\boldsymbol{T}))$, by definition
	$||\boldsymbol{T}-\tilde{\boldsymbol{T}}_\varepsilon(\Sigma,A,\eta)||_1\ge 1$.\\
	Another implication of \eqref{partialintegralcurves} is that, by bounding the terms involved,
	\begin{align}
		||\nabla_{(\Sigma,A,\eta)} \psi_{N,a,h,\varepsilon}(\boldsymbol{ T'},\gamma,{\Sigma},{A},\eta)||_1\le 10||\boldsymbol{T}-\tilde{\boldsymbol{T}}_\varepsilon(\Sigma,A,\eta)||_1\le 100 ||\nabla_{(\Sigma,A,\eta)} \psi_{N,a,h,\varepsilon}(\boldsymbol{ T'},\gamma,{\Sigma},{A},\eta)||_1.
	\end{align}
	These two remarks allow us to conclude, we consider
	\begin{align}
		||\nabla_{(\Sigma,A,\eta)} \psi_{N,a,h,\varepsilon}(\boldsymbol{ T'},a,{\Sigma},{A},\eta)||_1\ge \frac{1}{10}.
	\end{align}
	Now, we apply non-stationary phase and conclude.
	\begin{align}
		W_{h,N,\gamma}(\boldsymbol{T}) =W_{h,N,\gamma,+}(\boldsymbol{T})+W_{h,N,\gamma,-}(\boldsymbol{T})=O(\lambda_{\gamma,h}^{-\infty}).
	\end{align}
\end{proof}
\begin{obs}
	In that case, we focus especially on \eqref{imptrans}, since the other bounds do not require differentiating with respect to $\eta$, which means that the dispersion will work exactly the same as in \cite{positive}, given that $B_{\beta}$ has the required asymptotics (Remark \ref{ignoreverything}).
\end{obs}
As we stated before, the main goal of this section is to prove \eqref{imptrans}.
\begin{proof}[Proof of \eqref{imptrans}]
	The main idea of the proof is to see that, in this regime, we can apply stationary phase in $S,A,\eta$ (with $\lambda_{\gamma,h}$ as large parameter), and after doing so, the remaining integral on $\Sigma$ has a positive third derivative, so we are able to finish it with Van der Corput.\\\\
	We compute then the first and second derivatives of the phase.\\
	In the first derivative, we have:
	\begin{align}
		&\partial_{S} \psi_{N,a,h}(\boldsymbol{ T},\gamma,{\Sigma},{S},{A},\eta)=\eta(S^2-\left(A-\frac{a}{\gamma}\right)),\\
		&\partial_A \psi_{N,a,h}(\boldsymbol{T},\Sigma,A,\eta)=\eta\left(-S-\Sigma+\frac{T}{2\sqrt{1+\gamma A}}-2N\sqrt{A}\left(1+\frac{3}{4}B'_{\beta(\eta/h)}(\eta \lambda_{\gamma,h} A^{3/2})\right)\right),\\
		&\partial_\eta \psi_{N,a,h}(\boldsymbol{ T},\gamma,{\Sigma},{S},{A},\eta)\\
		&=Y+\frac{\Sigma^3}{3}+\Sigma(X-A)+\frac{S^3}{3}+S(\frac{a}{\gamma}-A)+T\frac{\sqrt{1+\gamma A}-\sqrt{1+\gamma}}{\gamma}-N\frac{4}{3}A^{3/2}\\
		&-NA^{3/2}B'_{\beta(\eta/h)}(\eta \lambda_{\gamma,h} A^{3/2})+N\lambda_{\gamma,h}^{-1}Resrob_1(\gamma A,\eta,h).
	\end{align}
	By having $|T|\sim |N|$ and being in the regime $|T|\ge \lambda_{\gamma,h}^2$, we have \eqref{TsimN}, $\lambda_{\gamma,h}^2 \lesssim|N|$.\\
	We compute now the second derivative
	\begin{align}
		&|\partial^2_A \psi_{N,a,h}(\boldsymbol{ T},\gamma,{\Sigma},{S},{A},\eta)|=|\eta\left(-\frac{\gamma T}{4\sqrt{1+\gamma A}^3}\right)\\
		&-\eta\frac{N}{\sqrt{A}}\left(1+\frac{3}{4}B'_{\beta(\eta/h)}(\eta \lambda_{\gamma,h} A^{3/2})+\frac{9}{4}\eta\lambda_{\gamma,h}A^{3/2}B''_{\beta(\eta/h)}(\eta \lambda_{\gamma,h} A^{3/2})\right)|\sim N,\\\\
		&|\partial^2_{A,\eta} \psi_{N,a,h}(\boldsymbol{ T},\gamma,{\Sigma},{S},{A},\eta)|\\
		&=|\eta^{-1} \partial_A \psi_{N,a,h}(\boldsymbol{ T},\gamma,{\Sigma},{S},{A},\eta)-\frac{3}{2}NA^2\eta \lambda_{\gamma,h}B''_{\beta(\eta/h)}(A^{3/2}\eta\lambda_{\gamma,h})-N\lambda_{\gamma,h}^{-1} \gamma \partial_A Resrob_1(\gamma A,\eta,h)|\\
		&\sim \frac{N}{\lambda_{\gamma,h}^2}+K|\partial_A \psi_{N,a,h}(\boldsymbol{ T},\gamma,{\Sigma},{S},{A},\eta)|,\\
		&|\partial^2_{\eta} \psi_{N,a,h}(\boldsymbol{ T},\gamma,{\Sigma},{S},{A},\eta)|=|-NA^3\lambda_{\gamma,h}B''_{\beta(\eta/h)}(\lambda_{\gamma,h}\eta A^{3/2})+N\lambda_{\gamma,h}^{-1}Resrob_2(\alpha,\eta,h)|\sim \frac{N}{\lambda_{\gamma,h}^2},\\
		&|\partial^2_{S} \psi_{N,a,h}(\boldsymbol{ T},\gamma,{\Sigma},{S},{A},\eta)|=|2S\eta|, |\partial^2_{S,A} \psi_{N,a,h}(\boldsymbol{ T},\gamma,{\Sigma},{S},{A},\eta)|=\eta,\\
		& |\partial^2_{S,\eta} \psi_{N,a,h}(\boldsymbol{ T},\gamma,{\Sigma},{S},{A},\eta)|=S^2-\left(A-\frac{a}{\gamma}\right).
	\end{align}
	With this, since $|\partial^2_A \psi_{N,a,h}(\boldsymbol{ T},\gamma,{\Sigma},{S},{A},\eta)|$ and $|\partial^2_{\eta} \psi_{N,a,h}(\boldsymbol{ T},\gamma,{\Sigma},{S},{A},\eta)|$ are always bounded away from zero, we can conclude by the implicit functions theorem, provided that the Hessian in the critical points is nonzero, that the critical points $A_c$ and $\eta_c$ are always isolated. Also, $S_c$ is easily computed $S_c=\pm \sqrt{A-\frac{a}{\gamma}}$ and in the transverse regime $\frac{a}{\gamma}\le \frac{1}{4}$, that gives us always two isolated stationary points and $S_c(\Sigma)\sim 1$ (Remembering that the support of the symbol is $\frac{1}{2}\le A\le \frac{3}{2}$).\\
	If we compute the determinant of the Hessian in the critical points, we obtain\\ $\det|H_{S,A,\eta} \psi_{N,a,h_{\gamma,h},\varepsilon}(\boldsymbol{T},S_c(\Sigma),A_c(\Sigma),\eta_c(\Sigma))|\sim \frac{N^2}{\lambda_{\gamma,h}^2}$. Also, the smaller eigenvalue is  $\sim \frac{N}{\lambda_{\gamma,h}^2}\gtrsim 1$, so we can apply stationary phase directly.
	By applying stationary phase, 
	\begin{align}
		&{W}_{h,N,\gamma,a}({T},{X},{Y})=V_{h,N,\gamma,a}(t(\gamma,{T},{X},{Y}),x(\gamma,{T},{X},{Y}),y(\gamma,{T},{X},{Y}))\\
		&=\lambda_{\gamma,h}^{-3/2}\frac{\lambda_{\gamma,h}}{N}\frac{\gamma^2}{(2\pi h)^3}\int e^{i\lambda_{\gamma,h}  \psi_{N,\gamma,\lambda}(\boldsymbol{ T},\gamma,{\Sigma},{S_c(\Sigma)},{A_c(\Sigma)},\eta_c(\Sigma))} \chi_4(\Sigma,N,\lambda_{\gamma,h})d{\Sigma}.
	\end{align}
	Where $\chi_4(\Sigma,N,\lambda_{\gamma,h})$ has a harmless dependence on $N,\lambda_{\gamma,h}$ and is of compact support in $\Sigma$.\\\\
	Now it remains to check that $\Psi(\Sigma):=\psi_{N,\gamma,\lambda}(\boldsymbol{ T},\gamma,{\Sigma},{S_c(\Sigma)},{A_c(\Sigma)},\eta_c(\Sigma))$, behaves correctly.\\\\
	We compute now the derivatives. For our purposes, we suppress the variables in $\psi$ in this section. Set $\psi|_\Sigma:=\psi_{N,\gamma,\lambda}(\boldsymbol{T},\gamma,\Sigma,{S_c}(\Sigma),{A_c(\Sigma)},\eta_c(\Sigma))$, we  use this notation to consider the derivatives of $\psi_{N,\gamma,\lambda}$, but applying then to the critical points. For example \\$ \partial_A \psi|_\Sigma:= \partial_A \psi_{N,\gamma,\lambda}(\boldsymbol{T},\gamma,\Sigma,S,{A},\eta)|_{(\Sigma,S,{A},\eta)=(\Sigma,{S_c}(\Sigma),{A_c(\Sigma)},\eta_c(\Sigma)}$. We now compute:
	\begin{align}
		&\Psi'(\Sigma)=A_c'(\Sigma) \partial_A \psi|_\Sigma +\eta_c'(\Sigma)\partial_\eta \psi|_\Sigma +S_c'(\Sigma) \partial_S \psi|_\Sigma+\partial_\Sigma\psi|_\Sigma=\partial_\Sigma\psi|_\Sigma=(\Sigma^2+X-A_c(\Sigma))\eta_c(\Sigma).
	\end{align}
	( $ \partial_A \psi|_\Sigma,\partial_\eta \psi|_\Sigma,\partial_S \psi|_\Sigma$, are zero by definition of critical point). Since this is the first derivative, we can assume from now on that $|\Sigma^2+X-A_c(\Sigma)|\lesssim \lambda_{\gamma,h}^{-3/4}$. Otherwise, we could apply non-stationary phase and conclude.
	\begin{align}
		&\Psi''(\Sigma)=\eta_c(\Sigma)(2\Sigma-A_c'(\Sigma))+\eta_c'(\Sigma)(\Sigma^2+X-A_c(\Sigma)),\\
		& \Psi'''(\Sigma)=(2-A_c''(\Sigma))\eta_c(\Sigma)+\eta_c'(\Sigma)(2\Sigma-A_c'(\Sigma))+\eta_c''(\Sigma)(\Sigma^2+X-A_c(\Sigma)).
	\end{align}
	Now, we see that we need to compute $A_c''(\Sigma)$. In order to do so, we use the condition of critical points.
	\begin{align}
		&\partial_A \psi|_\Sigma=0, 	\partial_\eta \psi|_\Sigma=0, S_c(\Sigma)=\pm \sqrt{A_c(\Sigma)-\frac{a}{\gamma}},\\
		&\label{firstA}\frac{d\partial_A \psi|_\Sigma}{d\Sigma}=0=A_c'(\Sigma)\partial^2_{A} \psi|_\Sigma+\eta_c'(\Sigma)\partial^2_{A,\eta} \psi|_\Sigma+\partial^2_{A,S} \psi|_\Sigma S_c'(\Sigma)+\partial^2_{A,\Sigma}  \psi|_\Sigma,\\
		&\label{firstAsinS}\frac{d\partial_A \psi|_\Sigma}{d\Sigma}=0=A_c'(\Sigma)\partial^2_{A} \psi|_\Sigma+\eta_c'(\Sigma)\partial^2_{A,\eta} \psi|_\Sigma\pm\frac{1}{2}\eta_c(\Sigma)A'_c(\Sigma)\sqrt{A_c(\Sigma)-\frac{a}{\gamma}}^{-1}+\partial^2_{A,\Sigma}  \psi|_\Sigma,\\
		&\label{firsteta}\frac{d\partial_\eta \psi|_\Sigma}{d\Sigma}=0=A_c'(\Sigma)\partial^2_{\eta,A} \psi|_\Sigma+\eta_c'(\Sigma)\partial^2_{\eta} \psi|_\Sigma+(\Sigma^2+X-A_c(\Sigma)).\\
	\end{align}
	We can solve this system easily, having
	\begin{align}
		\begin{pmatrix}
			&\partial^2_{A}\psi|_\Sigma \pm\frac{1}{2}\eta_c(\Sigma)\sqrt{A_c(\Sigma)-\frac{a}{\gamma}}^{-1} &\partial^2_{A,\eta} \psi|_\Sigma\\
			&\partial^2_{\eta,A} \psi|_\Sigma & \partial^2_{\eta} \psi|_\Sigma
		\end{pmatrix}
		\begin{pmatrix}
			A'_c(\Sigma) \\ \eta'_c(\Sigma)
		\end{pmatrix}=
		\begin{pmatrix}
			\eta_c(\Sigma)\\
			-(\Sigma^2+X-A_c(\Sigma))
		\end{pmatrix}.
	\end{align}
	This gives us, since 
	$
	\begin{pmatrix}
		a&b\\
		b&c
	\end{pmatrix}^{-1}=\frac{1}{ac-b^2}	\begin{pmatrix}
		c&-b\\
		-b&a
	\end{pmatrix}
	$ and the determinant of the matrix is $\sim \frac{N^2}{\lambda_{\gamma,h}^2}$, the asymptotics of the derivatives:
	\begin{align}
		\begin{pmatrix}
			A'_c(\Sigma) \\ \eta'_c(\Sigma)
		\end{pmatrix}=\frac{1}{\det(H \psi_{N,a,h_{\gamma,h},\varepsilon})}
		\begin{pmatrix}
			&  \partial^2_{\eta} \psi|_\Sigma&-\partial^2_{A,\eta} \psi|_\Sigma\\
			&-\partial^2_{\eta,A} \psi|_\Sigma &\partial^2_{A}\psi|_\Sigma \pm\frac{1}{2}\eta_c(\Sigma)\sqrt{A_c(\Sigma)-\frac{a}{\gamma}}^{-1}
		\end{pmatrix}
		\begin{pmatrix}
			\eta_c(\Sigma)\\
			-(\Sigma^2+X-A_c(\Sigma))
		\end{pmatrix}.
	\end{align}
	\begin{align}
		&A_c'(\Sigma)\sim \frac{1}{N},& \eta'_c(\Sigma)\lesssim \frac{\lambda^2}{N}(\Sigma^2+X-A_c(\Sigma))\lesssim \lambda^{-3/4}.
	\end{align}
	We do the same with the second derivatives:
	\begin{align}
		&\label{secondA}0=\frac{d^2\partial_A \psi|_\Sigma}{d\Sigma^2}=A_c''(\Sigma)\partial^2_{A} \psi|_\Sigma+\eta_c''(\Sigma)\partial^2_{A,\eta} \psi|_\Sigma+\partial^2_{A,S} \psi|_\Sigma S_c
		''(\Sigma)\\
		&+(A_c'(\Sigma))^2\partial^3_{A} \psi|_\Sigma+(\eta_c'(\Sigma))^2\partial^3_{A,\eta,\eta} \psi|_\Sigma+\partial^3_{A,S,S} \psi|_\Sigma (S_c'(\Sigma))^2\\
		&+2\partial^3_{A,A,\eta} \psi|_\Sigma A_c'(\Sigma)\eta_c'(\Sigma)+2\partial^3_{A,A,S}\psi|_\Sigma A_c'(\Sigma)S_c'(\Sigma)+2\partial^3_{A,S,\eta}\psi|_\Sigma \eta_c'(\Sigma)S_c'(\Sigma)-\eta_c'(\Sigma),\\
		&\label{secondeta}0=\frac{d^2\partial_\eta \psi|_\Sigma}{d\Sigma^2} =A_c''(\Sigma)\partial^2_{A,\eta} \psi|_\Sigma+\eta_c''(\Sigma)\partial^2_{\eta,\eta} \psi|_\Sigma+\partial^2_{\eta,S} \psi|_\Sigma S_c
		''(\Sigma)\\
		&+(A_c'(\Sigma))^2\partial^3_{A,A,\eta} \psi|_\Sigma+(\eta_c'(\Sigma))^2\partial^3_{\eta,\eta,\eta} \psi|_\Sigma+\partial^3_{\eta,S,S} \psi|_\Sigma (S_c'(\Sigma)^2)\\
		&+2\partial^3_{\eta,A,\eta} \psi|_\Sigma A_c'(\Sigma)\eta_c'(\Sigma)+2\partial^3_{\eta,A,S}\psi|_\Sigma A_c'(\Sigma)S_c'(\Sigma)+2\partial^3_{\eta,S,\eta}\psi|_\Sigma \eta_c'(\Sigma)S_c'(\Sigma)+(2\Sigma-A_c'(\Sigma)).
	\end{align}
	\begin{align}
		&	\begin{pmatrix}
			&\partial^2_{A}\psi|_\Sigma \pm\frac{1}{2}\eta_c(\Sigma)\sqrt{A_c(\Sigma)-\frac{a}{\gamma}}^{-1} &\partial^2_{A,\eta} \psi|_\Sigma\\
			&\partial^2_{\eta,A} \psi|_\Sigma & \partial^2_{\eta} \psi|_\Sigma
		\end{pmatrix}
		\begin{pmatrix}
			A''_c(\Sigma) \\ \eta''_c(\Sigma)
		\end{pmatrix}=	\begin{pmatrix}
			r_1(\Sigma)\\ r_2(\Sigma)
		\end{pmatrix},
	\end{align}
	with
	\begin{align}
		&r_1(\Sigma)=\mp\eta_c(\Sigma)\frac{A'_c(\Sigma)}{4}\sqrt{A_c(\Sigma)-\frac{a}{\gamma}}^{-3}+(A_c'(\Sigma))^2\partial^3_{A} \psi|_\Sigma+(\eta_c'(\Sigma))^2\partial^3_{A,\eta,\eta} \psi|_\Sigma\\
		&+\partial^3_{A,S,S} \psi|_\Sigma (S_c'(\Sigma))^2+2\partial^3_{A,A,\eta} \psi|_\Sigma A_c'(\Sigma)\eta_c'(\Sigma)+2\partial^3_{A,A,S}\psi|_\Sigma  A_c'(\Sigma)S_c'(\Sigma)+2\partial^3_{A,S,\eta}\psi|_\Sigma  \eta_c'(\Sigma)S_c'(\Sigma)-\eta_c'(\Sigma)\\
		&\lesssim \lambda^{-3/4},\\
		&r_2(\Sigma)=(A_c'(\Sigma))^2\partial^3_{A,A,\eta} \psi|_\Sigma+(\eta_c'(\Sigma))^2\partial^3_{\eta,\eta,\eta} \psi|_\Sigma+\partial^3_{\eta,S,S} \psi|_\Sigma (S_c'(\Sigma)^2)\\
		&+2\partial^3_{\eta,A,\eta} \psi|_\Sigma A_c'(\Sigma)\eta_c'(\Sigma)+2\partial^3_{\eta,A,S}\psi|_\Sigma  A_c'(\Sigma)S_c'(\Sigma)+2\partial^3_{\eta,S,\eta}\psi|_\Sigma  \eta_c'(\Sigma)S_c'(\Sigma)+(2\Sigma-A_c'(\Sigma))
		\lesssim 1.
	\end{align}
	In order to get then $A''_c(\Sigma),\eta''_c(\Sigma)$, it remains to estimate the order of the third derivatives, and we obtain the following bounds for third derivatives (See Appendix \ref{derivatives})
	\begin{align}
		|\partial^3_{A,A,\eta} \psi|_\Sigma|\lesssim {N},\;\;\;\; |\partial^3_{\eta,\eta,\eta}\psi|_\Sigma\lesssim \frac{N}{\lambda_{\gamma,h}^2},\;\;\;\;
		|\partial^3_{A,\eta,\eta}\psi_\Sigma|\lesssim \frac{N}{\lambda_{\gamma,h}^2},\;\;\;\;
		|\partial^3_{A}\psi|_\Sigma|\lesssim {N}.
	\end{align}
	By using these bounds in the system, we conclude $A_c''(\Sigma)\lesssim \frac{1}{N}$  and $\eta''_c(\Sigma)\lesssim 1$. We can complete the analysis now:
	\begin{enumerate}
		\item If $(2\Sigma-A_c'(\Sigma))\ge \lambda_{\gamma,h}^{-1/3}$, then we can use  Van der Corput lemma of order two, since $|\lambda_{\gamma,h}\Psi''(\Sigma)|\ge \lambda_{\gamma,h}^{2/3}$.
		\item If $(2\Sigma-A_c'(\Sigma))\le \lambda_{\gamma,h}^{-1/3}$, then we can use Van der Corput lemma of order three, since $|\Psi'''(\Sigma)|\ge \frac{1}{4}+O(\lambda_{\gamma,h}^{-1/12})\gtrsim 1$. \\
	\end{enumerate}
	By van der Corput lemma then, 
	\begin{align}
		&|{W}_{h,N,\gamma,a}({T},{X},{Y})|\\
		&=|\lambda_{\gamma,h}^{-3/2}\frac{\lambda_{\gamma,h}}{N}\frac{\gamma^2}{(2\pi h)^3}\int e^{i\lambda_{\gamma,h}  \psi_{N,\gamma,\lambda}(\boldsymbol{ T},\gamma,{\Sigma},{S_c(\Sigma)},{A_c(\Sigma)},\eta_c(\Sigma))} \chi_4(\Sigma,N,\lambda_{\gamma,h})d{\Sigma}|\\
		& \lesssim\frac{\lambda_{\gamma,h}^{-3/2}\lambda_{\gamma,h}\lambda_{\gamma,h}^{-1/3}}{N}\frac{\gamma^2}{h^3}=\frac{\lambda_{\gamma,h}^{-5/6}}{N}\frac{\gamma^2}{h^3}.
	\end{align}
\end{proof}
\subsection{Tangential waves $\gamma \sim a$}
We consider the case in this section $\gamma/4\le a \le 2\gamma$.\\
For this section, we introduce the following change of variables
\begin{align}
	\label{tangenctialchange}
	x=a\hat{X},\;\alpha=a\hat A, \; t=\sqrt{a}\sqrt{1+a}\hat{T}, \; s=\sqrt{a}\hat{S},\; \sigma=\sqrt{a}\hat \Sigma,\; y+t\sqrt{1+a}=a^{3/2} \hat{Y}
\end{align}
\begin{obs}
	This change of variables depends on $a$, so we cannot use it when adding all the $G_\gamma$.
\end{obs}
We define then the functions $\hat W$ as the functions $V$ after the change of variables.
\begin{defi}
	\label{defitan}
	\begin{align}
		\label{defihatw}
		&\hat{W}_{h,N,\gamma,a}(\hat{T},\hat{X},\hat{Y})=V_{h,N,\gamma,a}(t(a,\hat{T},\hat{X},\hat{Y}),x(a,\hat{T},\hat{X},\hat{Y}),y(a,\hat{T},\hat{X},\hat{Y}))\\
		&=\frac{a^2}{(2\pi h)^3}\int e^{i\lambda_{a,h} \hat \psi_{N,a,h}(\boldsymbol{\hat T},a,\hat{\Sigma},\hat{S},\hat{A},\eta)} \eta^2 \psi_1(\eta)\psi_3(\lambda_{a,h}^{2/3}\eta^{2/3}\hat{A})\psi_2(\frac{a}{\gamma}\hat{A}) d\hat{S}d\hat{\Sigma}d\hat{A}d\eta,
	\end{align}
	\begin{align}
		\label{defihatpsiB}
		&\hat \psi_{N,a,h}(\boldsymbol{\hat T},a,\hat{\Sigma},\hat{S},\hat{A},\eta)\\
		&=\eta \left(\hat{Y}+\frac{\hat{\Sigma}^3}{3}+\frac{\hat{S}^3}{3}+\hat{X}\hat{\Sigma}+\hat{S}-(\hat{S}+\hat{\Sigma})\hat{A}+\hat{T}\frac{\sqrt{1+a}}{a}(\sqrt{1+a\hat{A}}-\sqrt{1+a})-N\frac{4}{3}\hat{A}^{3/2}\right)\\
		&-\frac{N}{\lambda_{a,h}}B_{\beta(\eta/h)}(\eta \lambda_{a,h}\hat{A}^{3/2}).
	\end{align}
\end{defi}
The main goal of the section is to prove that, in the Robin case, the functions $\hat{W}$ satisfy the same estimates as in the Dirichlet case ([Propositions 4,5,6,7 in \cite{positive} for the Dirichlet case]), as well as prove \eqref{finiten} for the tangential waves.
\begin{prop}
	\label{bigtangenctial}
	There exist a $C_i\in  \mathbb{R}$, $i=1,2,3$ such that:
	\begin{enumerate}
		\item  For $|\hat T|\le C_1$, we have
		\begin{align}
			|({G}_{\gamma,h}(\boldsymbol{t}(a,\hat{\boldsymbol{T}})))|\lesssim  a^{-5/2} \lambda_{a,h}^2 \inf\{1,\left(\frac{1}{\lambda_{a,h} {|\hat{T}|}}\right)^{1/2}\}.
		\end{align}
		\item For $1\le |N|<\lambda_{a,h}^{1/3}$ and $|\hat{T}-4N|\le C_2/N$, we have
		\begin{align}
			\label{swallowtail}
			|\hat{W}_{h,N,\gamma,a}(\hat{\boldsymbol{T}})|\lesssim \frac{a^{1/2}\lambda_{a,h}^2}{a^3\lambda_{a,h}^{1/3}((\frac{N}{\lambda_{a,h}^{1/3}})^{1/4}+|N(\hat{T}-4N)|^{1/6})}.
		\end{align}
		\item For $1\le |N|<\lambda_{a,h}^{1/3}$ and $|\hat{T}-4N|\ge C_2/N$, we have
		\begin{align}
			|\hat{W}_{h,N,\gamma,a}(\hat{\boldsymbol{T}})|\lesssim \frac{a^{1/2}\lambda_{a,h}^2}{a^3\lambda_{a,h}^{1/3}(1+|N(\hat{T}-4N)|^{1/4})}.
		\end{align}
		\item For $\lambda_{a,h}\ge|N|\ge \lambda_{a,h}^{1/3}$, we have
		\begin{align}
			|\hat{W}_{h,N,\gamma,a}(\hat{\boldsymbol{T}})|\lesssim \frac{a^{1/2}\lambda_{a,h}^2}{a^3\lambda_{a,h}^{1/3}((\frac{N}{\lambda_{a,h}^{1/3}})^{1/4}+|N(\hat{T}-4N)|^{1/4})}.
		\end{align}
		\item For $\lambda_{a,h}\le|N|\le C_3 a^{-1/2}$, we have
		\begin{align}
			\label{thegood}
			|\hat{W}_{h,N,\gamma,a}(\hat{\boldsymbol{T}})|\lesssim \frac{a^{1/2}\lambda_{a,h}^{1/3}\lambda_{a,h}^2}{a^3N}.
		\end{align}
	\end{enumerate}
\end{prop}
\begin{prop}[\eqref{finiten} for the tangential waves]
	\label{neg foir tan}
	\begin{align}
		\sum_{N\notin \mathcal{N}_{1,h,\gamma}(\boldsymbol{t}(a,\boldsymbol{\hat{T}))}}\hat{W}_{h,N,\gamma,a}(\boldsymbol{\hat{T}})=O(h^{\infty}).
	\end{align}
\end{prop}
The proof of this case does not require derivatives on $\eta$ except we are back on the case of proposition \ref{neg for trans} and so, it's the same case as \cite{positive}, [Proposition 1], as explained by remark \ref{ignoreverything}.
\begin{obs}
	We will prove in detail the case \eqref{thegood}. For the other cases, since in the proofs of the equivalent statements on \cite{positive} the dispersion estimates do not involve derivatives on $\eta$ we would be dealing with the exact same phases with different symbols, so the only thing we will need to prove is that the symbols we obtain are equivalent to the ones the Dirichlet case has.\\
	In particular, the case where the swallowtail appears is \eqref{swallowtail}, in which the phase will be exactly the same as in the Dirichlet case.
\end{obs}
We start to study the function $\hat{W}$:\\\\
Inspecting the phase, we can note that for $\psi_2(\frac{a}{\gamma}\hat{A})=0$ if $\hat{A}\ge 6$ (Also if $\hat{A}\le 1/20$ ), which means that for $\lambda_{a,h}^{2/3}>100$ $\psi_3(\lambda_{a,h}^{2/3}\eta^{2/3}\hat{A})=1$, and since we are considering $\lambda$ a large parameter, we can assume that $\psi_3$ is always one and eliminating it.\\
Also since $\partial_{\hat{S}} \psi_{N,a,h}=\eta (\hat{S}^2+1-\hat{A})$ and $\partial_{\hat{\Sigma}} \psi_{N,a,h}=\eta (\hat{\Sigma}^2+\hat{X}-\hat{A})$, if either $|\hat{S}|>\sqrt{6}$ or $|\hat{\Sigma}|>\sqrt{6}$, the phase is non-stationary, and $\hat{W}=O(\lambda_{a,h}^{-\infty})$. Therefore we can add a cut-off $\psi_4$ on $\hat{S}$ and $\hat{\Sigma}$ and assume that $\hat{S}$ and $\hat{\Sigma}$ are less than $3$.
Thus, we conclude that
\begin{align}
	\label{defihatwcut}
	&\hat{W}_{h,N,\gamma,a}(\hat{T},\hat{X},\hat{Y})\\
	&=\frac{a^2}{(2\pi h)^3}\int e^{i\lambda_{a,h} \hat \psi_{N,a,h}(\boldsymbol{\hat T},a,\hat{\Sigma},\hat{S},\hat{A},\eta)} \eta^2 \psi_1(\eta)\psi_4(\hat{S})\psi_4(\hat{\Sigma})\psi_2(\frac{a}{\gamma}\hat{A}) d\hat{S}d\hat{\Sigma}d\hat{A}d\eta+O(\lambda_{a,h}^{-\infty}).
\end{align}
We now distinguish two regimes, $|N|\le \lambda_{a,h}$, $|N|\ge \lambda_{a,h}$.
\subsubsection{Case $|N|\le \lambda_{a,h}$:}
If $|N|\le \lambda_{a,h}$ then, for all $n\in \mathbb{N}$ we have $NB_{\beta(\eta/h)}^{(n)}(\eta \lambda_{a,h} \hat{A}^{3/2})$ is bounded (for $\lambda$ large). This means that $e^{iNB_{\beta(\eta/h)}(\eta \lambda_{a,h} \hat{A}^{3/2})}$ will not oscillate, and may be absorbed into the symbol with respect to the derivatives in $A$. \\
Let us focus a bit more in that: In order for $e^{iNB_{\beta(\eta/h)}(\eta \lambda_{a,h} \hat{A}^{3/2})}$ to be treated as a symbol, (with respect to $A$) we need that $|\partial_{\hat{A}}^n\left(e^{iNB_{\beta(\eta/h)}(\eta \lambda_{a,h} \hat{A}^{3/2})}\right)|\lesssim 1$ for all $n$. But all the derivatives are always a linear combinations of terms like this: $$K \lambda_{a,h}^{p+q-1}\eta^{p+q-1}\hat{A}^{z}N^{q}\prod_{i=1}^q\left(B^{(1+j_i)}_{\beta(\eta/h)}(\eta \lambda_{a,h} \hat{A}^{3/2})\right)e^{iNB_{\beta(\eta/h)}(\eta \lambda_{a,h} \hat{A}^{3/2})},$$ where $K$ is a constant and $\sum_i^q j_i=p$, $p,q\in\mathbb{N}$, $z\in \mathbb{R}$.\\
Since $\left|e^{iNB_{\beta(\eta/h)}(\eta \lambda_{a,h} \hat{A}^{3/2})}\right|=1$, $|\hat{A}^{z}|\le \max\{20^{-z}, 6^z\}$ and $|\eta|\lesssim \frac{3}{2}$, these terms are harmless, since they are bounded. Furthermore, \\$B'_{\beta(\eta/h)}(\eta \lambda_{a,h} \hat{A}^{3/2})\lesssim \lambda_{a,h}^{-2/3} O^\infty(\hat{A}^{-4/3})+O^\infty(\hat{A}^{-2})$. That is, for all $p$, $$|B^{(p)}_{\beta(\eta/h)}(\eta \lambda_{a,h} \hat{A}^{3/2})|\lesssim  \lambda_{a,h}^{-2/3} \eta^{-4/3-p+1} \lambda_{a,h}^{-4/3-p+1} \hat{A}^{3/2-p+1}+\eta^{-1-p} \lambda_{a,h}^{-1-p} \hat{A}^{-1-p}\le K_p {\lambda_{a,h}}^{-1-p}.$$. Then, coming back to the term we obtain that $K \lambda_{a,h}^{p+q}\eta^{p+q}\hat{A}^{z}N^{q}\prod_{i=1}^q\left(B^{(1+j_i)}_{\beta(\eta/h)}(\eta \lambda_{a,h} \hat{A}^{3/2})\right)e^{iNB_{\beta(\eta/h)}(\eta \lambda_{a,h} \hat{A}^{3/2})}\\
\le K_p \lambda_{a,h}^{p+q} N^q \prod_{i=1}^{q}\lambda_{a,h}^{-2-j_i} 
\le K_p \lambda_{a,h}^{p+q} \lambda_{a,h}^{\sum_{i=1}^{q}(-2-N j_i)}\le K_p {N^q}\lambda_{a,h}^{p+q}\lambda_{a,h}^{-2q-p}\le K_p \frac{N^q}{\lambda_{a,h}^q}\lambda_{a,h}^{p+q-p-q}\le K_p \frac{N^q}{\lambda_{a,h}^q}$. Now, in the case there $|N|\le \lambda_{a,h}$, we conclude that this part is a symbol so it may be absorbed into the symbolic expression.\\\\
Since now we have relegated $B_{\beta}$ to the symbol, and since $B_{\beta}$ carries all the terms specific to the Robin problem, we now have the same phase structure as in the Dirichlet case with only different but same-behaved symbol, so we obtain directly the same estimates as in \cite{positive} since we are studying  the same oscillatory integral with different symbol. 
\begin{obs}
	We remark that this regime is the regime where the swallowtail appears, so is the worst case in terms of dispersion.
\end{obs}
\subsubsection{Case $|N|\ge \lambda_{a,h}$:} 
First, we  see that we can do stationary phase, as we did for \eqref{imptrans}, we compute the derivatives in our variables
\begin{align}
	&\partial_\eta \hat \psi_{N,a,h}(\boldsymbol{\hat T},a,\hat{\Sigma},\hat{S},\hat{A},\eta)\\
	&= \left(\hat{Y}+\frac{\hat{\Sigma}^3}{3}+\frac{\hat{S}^3}{3}+\hat{X}\hat{\Sigma}+\hat{S}-(\hat{S}+\hat{\Sigma})\hat{A}+\hat{T}\frac{\sqrt{1+a}}{a}(\sqrt{1+a\hat{A}}-\sqrt{1+a})-N\frac{4}{3}\hat{A}^{3/2}\right)\\
	&-\hat{A}^{3/2}{N}B'_{\beta(\eta/h)}(\eta \lambda_{a,h}\hat{A}^{3/2})-\frac{N}{\lambda_{a,h}}Resrob_1(\hat{A},\eta,h),\\
	&\partial_{\hat{A}}\hat \psi_{N,a,h}(\boldsymbol{\hat T},a,\hat{\Sigma},\hat{S},\hat{A},\eta)\\
	&=\eta\left(-\hat{\Sigma}-\hat{S}+\frac{\hat{T}}{2}\frac{\sqrt{1+a}}{{\sqrt{1+a\hat{A}}}}-2N\hat{A}^{1/2}\left(1+\frac{3}{4}B'_{\beta{(\eta/h)}}(\eta \lambda_{a,h}\hat{A}^{3/2})\right)\right),\\
	&\partial^2_{\hat{A}}\hat \psi_{N,a,h}(\boldsymbol{\hat T},a,\hat{\Sigma},\hat{S},\hat{A},\eta)\sim N,\\
	&\partial^2_{\eta}\hat \psi_{N,a,h}(\boldsymbol{\hat T},a,\hat{\Sigma},\hat{S},\hat{A},\eta)=-N\lambda_{a,h} \hat{A}^3B''_{\beta{(\eta/h)}}(\eta \lambda_{a,h}\hat{A}^{3/2})+N\lambda_{a,h}^{-1}Resrob_2(\hat{A},\eta,h)\sim \frac{N}{\lambda_{a,h}^2},\\
	&\partial^2_{\eta,\hat{A}}\hat \psi_{N,a,h}(\boldsymbol{\hat T},a,\hat{\Sigma},\hat{S},\hat{A},\eta)=\eta^{-1} \partial_{\hat{A}}\hat \psi_{N,a,h}(\boldsymbol{\hat T},a,\hat{\Sigma},\hat{S},\hat{A},\eta)-\frac{3}{2}N\hat{A}^2 \eta \lambda_{a,h}B''_{\beta{(\eta/h)}}(\eta \lambda_{a,h}\hat{A}^{3/2})\\
	&-N\lambda_{a,h}^{-1}{a}\partial_\alpha Resrob_1(\hat{A},\eta,h)\sim_{\hat{A}=A_c, \eta=\eta_c} \frac{N}{\lambda_{a,h}^2}.\\
\end{align}
The upper bound $|\partial^2_{\eta}\hat \psi_{N,a,h}(\boldsymbol{\hat T},a,\hat{\Sigma},\hat{S},\hat{A},\eta)|\gtrsim  \frac{N}{\lambda_{a,h}^2}$ comes from the fact that $N\lambda_{a,h}^{-1}Resrob_2(\hat{A},\eta,h)$ is negligible compared to $\frac{N}{\lambda_{a,h}}$, since $|N\lambda_{a,h}^{-1}Resrob_2(\hat{A},\eta,h)|\lesssim \frac{N h^{1/3}}{\lambda_{a,h}^{5/3}}$, since $N\ge \lambda_{\gamma,h}$, implies $\lambda_{a,h}^{-1/2}\ge h^{1/3}$.\\
Thus Hessian in the stationary points (${A}_c$,$\eta_c$) is $|H_{\hat{A},\eta}\hat \psi_{N,a,h}(\boldsymbol{\hat T},a,\hat{\Sigma},\hat{S},{A}_c(\hat{S},\hat{\Sigma}),\eta_c(\hat{S},\hat{\Sigma}))|\sim \frac{N^2}{\lambda_{a,h}^2}$, with the larger eigenvalue being $\sim N$ and the smaller one $\sim \frac{N}{\lambda_{a,h}^2}$.  
This means that we cannot guarantee the largeness of the small eigenvalue. However, if we consider the large parameter $\frac{N}{\lambda_{a,h}}$, then we can apply stationary phase, the only problem is that, in the direction of the small eigenvalue, and since we only have $|N|\gtrsim \lambda_{a,h}$, we  have $O\left(\frac{N}{\lambda_{a,h}}\right)$, that is $O\left(1\right)$, but since we only are using stationary phase for obtaining non-dangerous symbols and thus only using the first order expansion of the stationary phase, this is enough for us since the smallness of the Hessian compensate the $O\left(\frac{N}{\lambda_{a,h}}\right)$.\\ We have
\begin{align}
	\hat{W}_{h,N,\gamma,a}(\hat{T},\hat{X},\hat{Y})=\frac{a^2}{Nh^3} \int e^{i\lambda_{a,h} \hat \psi_{N,a,h}(\boldsymbol{\hat T},a,\hat{\Sigma},\hat{S},\hat{A}_c(\hat{S},\hat{\Sigma}),\eta_c(\hat{S},\hat{\Sigma}))} \chi_4(\hat{\Sigma},\hat{S},N,h,\lambda_{a,h})d\hat{S}d\hat{\Sigma},
\end{align}
where $\chi_4$ is a harmless symbol.\\
We recall here the following elementary theorem (which is only an exact expression of the Taylor remainder of order 1).
\begin{teo}
	For $f$ smooth function of several variables, we have
	\begin{equation}
		f(\boldsymbol{x})=f(0)+\boldsymbol{x}\cdot \int_{0}^1\nabla f(\theta\boldsymbol{x})d\theta.
	\end{equation}
\end{teo}
By applying this to the function $A_c(\hat{\Sigma},\hat{S})$, we obtain
\begin{align}
	A_c(\hat\Sigma,\hat S)=A_c|_0+A_{cr}(\hat \Sigma,\hat{S}).
\end{align}
Where 
\begin{align}
	A_{cr}(\hat \Sigma,\hat{S}):=(\hat S,\hat \Sigma)\cdot \int_{0}^{1} \nabla_{\hat S,\hat \Sigma}{A}_c(\theta \hat S,\theta \hat \Sigma)d\theta,
\end{align}
\begin{align}
	\eta_{cr}(\hat \Sigma,\hat{S}):=(\hat S,\hat \Sigma)\cdot \int_{0}^{1} \nabla_{\hat S,\hat \Sigma}\eta_c(\theta \hat S,\theta \hat \Sigma)d\theta.
\end{align}
Now we proceed to prove the proposition.\\
By using same heuristics as in tangential case, we have the following:
\begin{prop}
	We have the following asymptotics for the derivatives of the critical points:
	\begin{align}
		&|\partial_{\hat{\Sigma}} A_c({\hat{\Sigma}},{\hat{S}})|\lesssim \frac{1}{N},
		& |\partial_{\hat{S}} \eta_c({\hat{\Sigma}},{\hat{S}})|\lesssim \lambda_{a,h}|{\hat{S}}^2-1+A_c({\hat{\Sigma}},{\hat{S}})|,\\ &|\partial_{\hat{\Sigma}} \eta_c({\hat{\Sigma}},{\hat{S}})|\lesssim \lambda_{a,h}|{\hat{\Sigma}}^2-\hat{X}+A_c({\hat{\Sigma}},{\hat{S}})|,
		&|\partial_{\hat{S}} A_c({\hat{\Sigma}},{\hat{S}})|\lesssim \frac{1}{N}.
	\end{align}
\end{prop}
The proposition is again a consequence of the implicit functions theorem, since
\begin{align}
	\begin{pmatrix}
		&\partial^2_{\hat{A}}\psi|_{\hat{\Sigma},\hat{S}} &\partial^2_{\hat{A},\eta} \psi|_{\hat{\Sigma},\hat{S}}\\
		&\partial^2_{\eta,\hat{A}} \ \psi|_{\hat{\Sigma},\hat{S}}& \partial^2_{\eta}  \psi|_{\hat{\Sigma},\hat{S}}
	\end{pmatrix}
	\begin{pmatrix}
		\partial_{\hat{\Sigma}} A_c(\hat{\Sigma},\hat{S}) \\ \partial_{\hat{\Sigma}}\eta_c(\hat{\Sigma},\hat{S})
	\end{pmatrix}=
	\begin{pmatrix}
		\eta_c(\hat{\Sigma},\hat{S})\\
		-(\hat{\Sigma}^2-\hat{X}+A_c(\hat{\Sigma},\hat{S}))
	\end{pmatrix},
\end{align}
\begin{align}
	\begin{pmatrix}
		&\partial^2_{\hat{A}}\psi|_{\hat{\Sigma},\hat{S}} &\partial^2_{\hat{A},\eta} \psi|_{\hat{\Sigma},\hat{S}}\\
		&\partial^2_{\eta,\hat{A}} \ \psi|_{\hat{\Sigma},\hat{S}}& \partial^2_{\eta}  \psi|_{\hat{\Sigma},\hat{S}}
	\end{pmatrix}
	\begin{pmatrix}
		\partial_{\hat{S}} A_c(\hat{\Sigma},\hat{S}) \\ \partial_{\hat{S}}\eta_c(\hat{\Sigma},\hat{S})
	\end{pmatrix}=
	\begin{pmatrix}
		\eta_c(\hat{\Sigma},\hat{S})\\
		-(\hat{S}^2-1+A_c(\hat{\Sigma},\hat{S}))
	\end{pmatrix}.
\end{align}
This implies that
\begin{align}
	\begin{pmatrix}
		\partial_{\hat{\Sigma}} A_c(\hat{\Sigma},\hat{S}) \\ \partial_{\hat{\Sigma}}\eta_c(\hat{\Sigma},\hat{S})
	\end{pmatrix}=	\frac{1}{\det(H_{\hat{A},\eta}\hat \psi_{N,a,h}|_{\hat{\Sigma},\hat{S}})}\begin{pmatrix}
		& \partial^2_{\eta}  \psi|_{\hat{\Sigma},\hat{S}}&-\partial^2_{\hat{A},\eta} \psi|_{\hat{\Sigma},\hat{S}}\\
		&-\partial^2_{\eta,\hat{A}} \ \psi|_{\hat{\Sigma},\hat{S}}& \partial^2_{\hat{A}}\psi|_{\hat{\Sigma},\hat{S}}
	\end{pmatrix}
	\begin{pmatrix}
		\eta_c(\hat{\Sigma},\hat{S})\\
		-(\hat{\Sigma}^2-\hat{X}+A_c(\hat{\Sigma},\hat{S}))
	\end{pmatrix},
\end{align}
\begin{align}
	\begin{pmatrix}
		\partial_{\hat{S}} A_c(\hat{\Sigma},\hat{S}) \\ \partial_{\hat{S}}\eta_c(\hat{\Sigma},\hat{S})
	\end{pmatrix}=	\frac{1}{\det(H_{\hat{A},\eta}\hat \psi_{N,a,h}|_{\hat{\Sigma},\hat{S}})}\begin{pmatrix}
		& \partial^2_{\eta}  \psi|_{\hat{\Sigma},\hat{S}} &-\partial^2_{\hat{A},\eta} \psi|_{\hat{\Sigma},\hat{S}}\\
		&-\partial^2_{\eta,\hat{A}} \ \psi|_{\hat{\Sigma},\hat{S}}&  \partial^2_{\hat{A}}\psi|_{\hat{\Sigma},\hat{S}}
	\end{pmatrix}
	\begin{pmatrix}
		\eta_c(\hat{\Sigma},\hat{S})\\
		-(\hat{S}^2-1+A_c(\hat{\Sigma},\hat{S}))
	\end{pmatrix}.
\end{align}
\begin{proof}[Proof of proposition \ref{bigtangenctial}, item (5)]
	With all the set-up, what we have to prove is that 
	\begin{align}
		|I|=\left|\int e^{i\lambda_{a,h} \hat \psi_{N,a,h}(\boldsymbol{\hat T},a,\hat{\Sigma},\hat{S},\hat{A}_c(\hat{S},\hat{\Sigma}),\eta_c(\hat{S},\hat{\Sigma}))} \chi_4(\hat{\Sigma},\hat{S},N,h,\lambda_{a,h})d\hat{S}d\hat{\Sigma} \right|\lesssim \lambda_{a,h}^{-2/3}.
	\end{align}
	We will prove that this integral is two almost independent oscillatory integrals to which we can use Van der Corput lemma.\\\\
	First, we assume that $|({\hat{\Sigma}}^2-X+A_c({\hat{\Sigma}},{\hat{S}}))|,|({\hat{S}}^2-1+A_c({\hat{\Sigma}},{\hat{S}}))|\le \lambda_{a,h}^{-11/12}$. If that is not the case we have non-stationary phase.\\
	We  also assume that $|2{\hat{\Sigma}}-\partial_{\hat\Sigma} A_c({\hat{\Sigma}},{\hat{S}})|,|2{\hat{S}}-\partial_{\hat{S}} A_c({\hat{\Sigma}},{\hat{S}})|\le \lambda_{a,h}^{-1/3}$. If that was not the case in one of the variables, do stationary phase in that variable and van der Corput on the other, is a strictly easier case than the one we will study.
	\begin{obs}\label{smallness}
		Since we know that $|\partial_{\hat{\Sigma}} A_c({\hat{\Sigma}},{\hat{S}})|,|\partial_{\hat{S}} A_c({\hat{\Sigma}},S)|\lesssim \frac{1}{N}$, our conditions implies that $|\hat{S}|,|\hat{\Sigma}|\lesssim \lambda_{a,h}^{-1/3}$. This implies that $|\eta_{cr}(\hat{\Sigma},\hat{S})|\lesssim \lambda_{a,h}^{-1/3}\lambda_{a,h}\lambda_{a,h}^{-11/12}=\lambda_{a,h}^{-1/4}$.
	\end{obs}
	We now introduce the following change of variables and terms (and we do the abuse of notation of suppressing from the notation variables on which the expression does not depend, since Van der Corput lemma gives a universal constant):
	\begin{align}
		P:=\eta_c|_0^{2/3}\lambda_{a,h}^{2/3}(A_c|_0-X), Q:=\eta_c|_0^{2/3}\lambda_{a,h}^{2/3}(A_c|_0-1), \hat{\Sigma}=\eta_c|_0^{-1/3}\lambda_{a,h}^{-1/3}p, \hat{S}=\eta_c|_0^{-1/3}\lambda_{a,h}^{-1/3}q.
	\end{align}
	After the change of variables, and these considerations, we obtain
	\begin{align}
		I=\lambda_{a,h}^{-2/3}\int e^{i\left(\frac{p^3}{3}-Pp\right)}e^{i\left(\frac{q^3}{3}-Qq\right)}\chi_5(p,q)dpdq.
	\end{align}
	With the symbol $\chi_5$ which we will study later, being the following:\\
	\begin{align}
		&\tilde{\chi_5}(\hat{\Sigma},\hat{S})=\chi_5(\hat{\Sigma}\eta_c|_0^{1/3}\lambda_{a,h}^{1/3},\hat{S}\eta_c|_0^{1/3}\lambda_{a,h}^{1/3})\\
		&=\chi_4(\hat{\Sigma},\hat{S},N,h,\lambda_{a,h})e^{i\lambda_{a,h} \left(\hat \psi_{N,a,h}(\boldsymbol{\hat T},a,\hat{\Sigma},\hat{S},\hat{A}_c(\hat{S},\hat{\Sigma}),\eta_c(\hat{S},\hat{\Sigma}))-\eta_c|_0\left[\frac{\hat{\Sigma^3}}{3}+\frac{\hat{S}^3}{3}-(A_c|_0-X)\hat{\Sigma}-(A_c|_0-1)\hat{S}\right]\right)}.
	\end{align}
	By integration by parts in $p$ and $q$ (and since the support of $\chi_5(p,q)$ is compact), this gives
	\begin{align}
		&\lambda_{a,h}^{2/3}I=-\frac{1}{2}\int e^{i\left(\frac{q^3}{3}-Qq\right)}\partial_p \chi_5(p,q)\left(\int_{0}^{p} e^{i\left(\frac{z^3}{3}-Pz\right)} dz\right) dp dq\\
		&=\int \partial^2_{p,q} \chi_5(p,q) \left(\int_{0}^{p} e^{i\left(\frac{z^3}{3}-Pz\right)} dz\right)\left(\int_{0}^{q} e^{i\left(\frac{z^3}{3}-Qz\right)} dz\right)dpdq.
	\end{align}
	After the same change of variables as in the Airy integral $z=t\sqrt{P}$,$z=t\sqrt{Q}$ respectively, we obtain, by Van der Corput of order 2 around 1 or non-stationary phase away from 1, the following.
	\begin{align}
		&\lambda_{a,h}^{2/3}|I|\lesssim  ||\partial^2_{p,q}\chi_5(p,q)||_{L^1_{p,q}(\mathbb{R}^2)} \frac{1}{1+|P|^{1/4}}\frac{1}{1+|Q|^{1/4}}\\
		&\lesssim ||\partial^2_{p,q}\chi_5(p,q)||_{L^1_{p,q}(\mathbb{R}^2)}.
	\end{align}
	We conclude then if we prove that $||\partial^2_{p,q}\chi_5(p,q)||_{L^1_{p,q}(\mathbb{R}^2)}\lesssim 1$ uniformly. 
	We compute the derivatives of the phase to see that they do not oscillate.\\
	First we can introduce a new auxiliary notation for this phase
	\begin{align}
		&\hat \Psi_{N,a,\lambda}(\boldsymbol{\hat T},a,\hat{\Sigma},\hat{S})\\
		&:=\hat \psi_{N,a,h}(\boldsymbol{\hat T},a,\hat{\Sigma},\hat{S},\hat{A}_c(\hat{S},\hat{\Sigma}),\eta_c(\hat{S},\hat{\Sigma}))\\
		&-\eta_c|_0\left[\frac{\hat{\Sigma^3}}{3}+\frac{\hat{S}^3}{3}-(A_c|_0-X)\hat{\Sigma}-(A_c|_0-1)\hat{S}\right],\\
		&\partial_{\hat{S}}\left(\hat \Psi_{N,a,\lambda}(\boldsymbol{\hat T},a,\hat{\Sigma},\hat{S})\right)\\
		&=\partial_{\hat{S}}\hat \psi_{N,a,h}(\boldsymbol{\hat T},a,\hat{\Sigma},\hat{S},\hat{A}_c(\hat{S},\hat{\Sigma}),\eta_c(\hat{S},\hat{\Sigma}))-\eta_c|_0\left[\hat{S}^2-(A_c|_0-1)\right]\\
		&=\eta_c(\hat{\Sigma},\hat{S})\left[\hat{S}^2-(A_c(\hat{\Sigma},\hat{S})-1)\right]-\eta_c|_0\left[\hat{S}^2-(A_c|_0-1)\right],\\
		& \partial_{\hat{\Sigma}}\left(\hat \Psi_{N,a,\lambda}(\boldsymbol{\hat T},a,\hat{\Sigma},\hat{S})\right)\\
		&=\eta_c(\hat{\Sigma},\hat{S})\left[\hat{\Sigma}^2-(A_c(\hat{\Sigma},\hat{S})-X)\right]-\eta_c|_0\left[\hat{\Sigma}^2-(A_c|_0-X)\right],\\
		&\partial_{\hat{S},\hat{\Sigma}}\left(\hat \Psi_{N,a,\lambda}(\boldsymbol{\hat T},a,\hat{\Sigma},\hat{S})\right)\\
		&=\partial_{\hat{\Sigma}}\eta_c(\hat{\Sigma},\hat{S})\left[\hat{S}^2-(A_c(\hat{\Sigma},\hat{S})-1)\right]+[\eta_c|_0-\eta_c(\hat{\Sigma},\hat{S})]\partial_{\hat{\Sigma}}A_c(\hat{\Sigma},\hat{S}).
	\end{align}
	Now, we can bound these derivatives using the assumptions we have made
	\begin{align}
		&|\partial_{\hat{S}}\left(\hat \Psi_{N,a,\lambda}(\boldsymbol{\hat T},a,\hat{\Sigma},\hat{S})\right)|\\
		&=|\eta_c(\hat{\Sigma},\hat{S})\left[\hat{S}^2-(A_c(\hat{\Sigma},\hat{S})-1)\right]-\eta_c|_0\left[\hat{S}^2-(A_c|_0-1)\right]|\\
		&=|\eta_c(\hat{\Sigma},\hat{S})\left[\hat{S}^2-(A_c(\hat{\Sigma},\hat{S})-1)\right]-\eta_c|_0\left[\hat{S}^2-(A_c(\hat{\Sigma},\hat{S})-1)\right]\\
		&+\eta_c|_0\left[\hat{S}^2-(A_c(\hat{\Sigma},\hat{S})-1)\right]-\eta_c|_0\left[\hat{S}^2-(A_c|_0-1)\right]|\\
		&\lesssim \lambda_{a,h}^{-7/6}+O(1/N)\lesssim \lambda_{a,h}^{-1},\\
		&|\partial_{\hat{\Sigma}}\left(\hat \Psi_{N,a,\lambda}(\boldsymbol{\hat T},a,\hat{\Sigma},\hat{S})\right)|\lesssim \lambda_{a,h}^{-7/6}+O(1/N)\lesssim \lambda_{a,h}^{-1},\\
		&|\partial_{\hat{S},\hat{\Sigma}}\left(\hat \Psi_{N,a,\lambda}(\boldsymbol{\hat T},a,\hat{\Sigma},\hat{S})\right)|\lesssim \lambda_{a,h}^{-1/2}.
	\end{align}
\end{proof}
Now, we have:
\begin{align}
	&\partial_{\hat{S},\hat{\Sigma}}\tilde{\chi_5}(\hat{\Sigma},\hat{S})=e^{i\lambda_{a,h}\hat \Psi_{N,a,\lambda}(\boldsymbol{\hat T},a,\hat{\Sigma},\hat{S})}\Big(\partial_{\hat{S},\hat{\Sigma}}\chi_4(\hat{\Sigma},\hat{S},N,h,\lambda_{a,h})\\
	&+i\lambda_{a,h}\partial_{\hat{S}}\hat \Psi_{N,a,\lambda}(\boldsymbol{\hat T},a,\hat{\Sigma},\hat{S}) \partial_{\hat{\Sigma}}\chi_4(\hat{\Sigma},\hat{S},N,h,\lambda_{a,h})\\
	&+i\lambda_{a,h}\partial_{\hat{\Sigma}}\hat \Psi_{N,a,\lambda}(\boldsymbol{\hat T},a,\hat{\Sigma},\hat{S}) \partial_{\hat{S}}\chi_4(\hat{\Sigma},\hat{S},N,h,\lambda_{a,h})\\
	&+i^2\lambda_{a,h}^2\partial_{\hat{\Sigma}}\hat \Psi_{N,a,\lambda}(\boldsymbol{\hat T},a,\hat{\Sigma},\hat{S}) \partial_{\hat{S}}\hat \Psi_{N,a,\lambda}(\boldsymbol{\hat T},a,\hat{\Sigma},\hat{S}) \chi_4(\hat{\Sigma},\hat{S},N,h,\lambda_{a,h})\\
	&+\lambda_{a,h}i\chi_4(\hat{\Sigma},\hat{S},N,h,\lambda_{a,h})\partial_{\hat{\Sigma},\hat{S}}\hat \Psi_{N,a,\lambda}(\boldsymbol{\hat T},a,\hat{\Sigma},\hat{S}) \Big).
\end{align}
Now, with our bounds, we obtain $|\partial_{\hat{S},\hat{\Sigma}}\tilde{\chi_5}(\hat{\Sigma},\hat{S})|\lesssim \lambda_{a,h}^{1/2}$, that is, $|\partial_{p,q}{\chi_5}(p,q)|\lesssim  1$. Therefore the integral $||\partial_{p,q}{\chi_5}||_{L^1}\lesssim 1$. This is because the integral is in a fixed interval, since $|\Sigma|,|S|\lesssim \lambda_{a,h}^{-1/3}$, that is $|p|,|q|\lesssim 1$.
\begin{obs}
	Essentially this is a more explicit rewriting of the analogous proof in \cite{positive}, the only difference is in proving that $\hat{A}_{cr}$ is small by controlling the $Resrob$ terms.
\end{obs}

\section{Gallery modes}
\label{sectiongall}
The regime $\gamma<h^{2/3-\varepsilon}$, will be handled by studying the so-called gallery modes.
We start by considering the space of the gallery modes:
 \begin{defi}
 	For every $k\in\mathbb{N}$, $0<h<1$, we define the space $E^h_{k}$ as the closure in $L^2(\Omega_2)$ of the following set:
 	\begin{equation}
 		\left\{\int \psi_1(h\theta)\hat{g}(\theta)e^{iy\theta}e_k(x,\theta,\iota)d\theta : g\in  \mathcal{S}\left(\mathbb{R}\right)\right\}
 	\end{equation} 
 \end{defi}
\begin{prop}
	For every $j\ne k$, $E^h_{k}$ is orthogonal to $E^h_j$. \\
	For every $u\in L^2(\Omega_2)$, then $\psi_1(hD_y)u$, can be written as a series of elements of $\{E^h_k\}$.
\end{prop}
\begin{proof}
	Consider $v_k(x,y)=\int \psi_1(h\theta)\hat{g}_j(\theta)e^{iy\theta}e_k(x,\theta,\iota)d\theta\in E^h_k$, $v_j(x,y)=\int \psi_1(h\theta)\hat{g}_k(\theta)e^{iy\theta}e_j(x,\theta,\iota)d\theta\in E^h_j$. Then
	\begin{align}
		\int v_k(x,y) \overline{v}_j(x,y) dydx=\int \mathcal{F}_y{v}_k(x,\theta) \mathcal{F}_y{\overline{v}}_j(x,\theta) d\theta dx= \int \psi_1^2(h\theta) \hat{g}_j(\theta)\hat{g}_k(\theta) e_j(x,\theta,\iota)e_k(x,\theta,\iota)dxd\theta=0
	\end{align}
since $e_k$,$e_j$ are orthogonal in $x$ for every $\theta>c$. 
This proves orthogonality of the $E^h_k$.\\
We now prove that we can write $\psi_1(hD_y)u$ as a sum of elements of $E^h_k$. We prove by contradiction. If there exists a $u$ such that $\psi_1(hD_y)u$ cannot be written as a sum of elements of $E^h_k$, there also exists a $u$ such that $\psi_1(hD_y)u$ is orthogonal to all elements of $E^h_k$. We consider $\psi_1(hD_y)u$ as such and nonzero. We consider the Fourier transform of $\psi_1(hD_y)u$ in $y$, $w(x,\theta)=\psi_1(h\theta) \hat{u}(x,\theta)$. Since for every $\theta>c$, the family $e_k$ is complete, we can consider the spectral writing of this function $w(x,\theta)=\sum_{j=1}^\infty  \psi_1(h\theta)w_j(\theta)e_j(x,\theta,\iota)$. Consider $j\in\mathbb{N}$ arbitrary. By definition, for every $g\in \mathcal{S}(\mathbb{R})$, 
\begin{align}
	\int \psi_1^2(h\theta)g(\theta) w(x,\theta) e_k(x,\theta,\iota) d\theta dx=\int \psi_1^2(h\theta) g(\theta)w_j(\theta)=0.
\end{align}
Therefore, $\psi_1^2(h\theta)w_j(\theta)=0$ for almost all $\theta$, which implies that $\psi_1(h\theta)w_j(\theta)=0$. Also, since $j$ is arbitrary too, for almost all $\theta$, $w(x,\theta)$ is zero, that is $\psi_1(hD_y)u$ is zero, which is a contradiction.
\end{proof}
Given the form of $G$, it is useful to have the following estimate
\begin{teo}
	\label{stricharzgallery}
	Let $$u_{k,\pm,\iota}(t,x,y)=  \frac{1}{h}\int \psi_1(\eta) e^{\pm it\sqrt{\lambda_{k,\iota}(\eta/h)}}e^{iy\eta/h}e_{k}(x,\eta/h) d\eta.$$ Define $u_{k,0,\iota}(x,y)=u_{k,\pm,\iota}(0,x,y)$, then, for every $\iota,q,r,k$ satisfying $q\ge 4, \frac{1}{q}=\frac{1}{2}\left(\frac{1}{2}-\frac{1}{r}\right)$, $\psi_1\in\mathcal{C}^\infty_0\left(\frac{1}{2},\frac{3}{2}\right)$, $h^{2/3}w_k\le 10$ and $\beta_{Str}=2\left(\frac{1}{2}-\frac{1}{r}\right)-\frac{1}{q}$, there exists a $C$ independent of $q,r,k$ and $\psi_1$ such that 
	\begin{equation}
		||u_{k,\pm,\iota}||_{L^q([0,t_{max}],L^r(\Omega_2))}\le C h^{-\beta_{Str}}||u_{k,0,\iota}||_{L^2(\Omega_2)}.
	\end{equation}
\begin{obs}
This is a similar treatment to the one done in \cite{Ivanovici2010Counterexamples}, however, in our case we get a constant independent on $k$, unlike in the treatment of whispering gallery modes in \cite{Ivanovici2010Counterexamples}.
\end{obs}
	We will use stationary phase methods to solve the problem, so we begin by analyzing  $w_{\beta(\frac{\eta}{h}),k}$ and proving that its derivatives are "small enough" similarly to proposition \ref{restrobs}.
	\begin{prop}
		\label{boundszeroes}
		We have the following bounds: $|\partial_\eta w_{\beta(\frac{\eta}{h}),k}|,|\partial^2_\eta w_{\beta(\frac{\eta}{h}),k}|\lesssim \frac{h^{2/3}}{w_k}$, and $\partial_\eta w_{\beta(\frac{\eta}{h}),k}$ has the opposite sign from $b$.
	\end{prop}
\begin{obs}
	In the Neumann case $b=0$, there are no derivatives of $w_{\beta(\frac{\eta}{h}),k}$, they are all zero.
\end{obs}
	\begin{proof}
		By proposition \ref{zeros}, we have
		\begin{align}
			w_{\beta(\frac{\eta}{h}),k}=-Lai^{-1}(\cot(\beta(\frac{\eta}{h}))),
		\end{align}
		where $Lai^{-1}:\mathbb{R}\rightarrow(-w_{k},-w_{k-1})$. \\
		We recall that $Lai$ satisfies the differential equation $Lai'(x)=x-Lai^2(x)$, (by lemma \ref{lemalai}), so by the inverse function theorem, $Lai^{-1'}(x)=\frac{1}{Lai'(Lai^{-1}(x))}=\frac{1}{Lai^{-1}(x)-Lai^2(Lai^{-1}(x))}=\frac{1}{Lai^{-1}(x)-x^2}$, so
		\begin{align}
			&\partial_\eta w_{\beta(\frac{\eta}{h}),k}=-Lai^{-1'}(\cot(\beta(\frac{\eta}{h})))\cot'(\beta(\frac{\eta}{h}))\beta'(\frac{\eta}{h})\frac{1}{h}\\
			&=\frac{1}{\cot(\beta(\frac{\eta}{h}))^2-Lai^{-1}(\cot(\beta(\frac{\eta}{h})))}(\cot^2(\beta(\frac{\eta}{h}))-1)\beta'(\frac{\eta}{h})\frac{1}{h}\\
			&=-\frac{1}{b^2h^{4/3}\eta^{-4/3}+w_{\beta(\frac{\eta}{h}),k}}(b^2h^{4/3}\eta^{-4/3}+1)\beta'(\frac{\eta}{h})\frac{1}{h}.
		\end{align}
		From this expression we deduce the corresponding properties for the first derivative.\\
		We compute the second derivative
		\begin{align}
			&\partial^2_\eta w_{\beta(\frac{\eta}{h}),k}=-\left(\frac{b^2h^{4/3}\eta^{-4/3}+1}{b^2h^{4/3}\eta^{-4/3}+w_{\beta(\frac{\eta}{h}),k}}\beta'(\frac{\eta}{h})\frac{1}{h}\right)'\\
			&=-\frac{1}{h^2}\beta''(\frac{\eta}{h})\frac{b^2h^{4/3}\eta^{-4/3}+1}{b^2h^{4/3}\eta^{-4/3}+w_{\beta(\frac{\eta}{h}),k}}\\
			&+\beta'(\frac{\eta}{h})\frac{1}{h}\frac{b^2h^{4/3}\eta^{-7/3}\frac{4}{3}}{b^2h^{4/3}\eta^{-4/3}+w_{\beta(\frac{\eta}{h}),k}}+\frac{(b^2h^{4/3}\eta^{-4/3}+1)(\partial_\eta w_{\beta(\frac{\eta}{h}),k}-b^2h^{4/3}\eta^{-7/3}\frac{4}{3})}{(b^2h^{4/3}\eta^{-4/3}+w_{\beta(\frac{\eta}{h}),k})^2}\beta'(\frac{\eta}{h})\frac{1}{h}.
		\end{align}
		This yields the desired bounds.
	\end{proof}
\end{teo}
\begin{proof}[Proof of theorem \ref{stricharzgallery}]
	The first thing we note is that, if we define 
	\begin{equation}
		\Gamma_{x,w_{k}}(y):=\frac{1}{h}\int \frac{1}{\sqrt{L_{\beta(\eta/h)}'(w_{\beta(\frac{\eta}{h}),k})}} e^{i\frac{y\eta}{h}}(\eta/h)^{\frac{1}{3}}Ai((\eta/h)^{2/3}x-w_{\beta(\frac{\eta}{h}),k})\sqrt\psi_1(\eta)d\eta,
	\end{equation}
	and 
	\begin{equation}
		W_{k}(t,y):=\frac{1}{h}\int e^{i\frac{t}{h}\eta \sqrt{1+w_{\beta(\frac{\eta}{h}),k}(h/\eta)^{2/3}}}\sqrt\psi_1(\eta)e^{i\frac{y\eta}{h}} d\eta.
	\end{equation}
	(where we ignore the dependence on $b=\cot(\iota)$, since it is fixed and specific for the problem), then $u_{k,+,\iota}(t,x,y)=(\Gamma_{x,w_{k}}\ast_y W_{k}(t,\cdot))(y)$.\\
	Now, we can use Young's inequality in $y$ and get, 
	\begin{equation}
		||u_{k,b}(t,x,\cdot)||_{L^\infty_y}\le ||\Gamma_{x,w_{k}}||_{L^1_y}||W_{k}(t,\cdot)||_{L^\infty_y}.
	\end{equation}
	We can now conclude by a $TT^*$ argument if we prove the following lemmas (which are only a study of the dispersion of the related functions).
	\begin{obs}
		Throughout this proof, compared with the corresponding argument in \cite{positive}, we have the main added difficulty that $w_{\beta(\frac{\eta}{h}),k}$ depends on $\eta$. We will address it by comparing it with $w_{k}$.
	\end{obs}
	\begin{lema}
		\label{boundgamma}
		\begin{equation}
			\sup_x||\Gamma_{x,w_k}||_{L^1_y}\lesssim \frac{w_{k}^{1/4}}{h^{1/3}}.
		\end{equation}
	\end{lema}
	\begin{lema}
		\label{otherlemmagallery}
		There exists a $C$ such that, for every $t\ge Ch$, we have
		\begin{equation}
			|W_{k}(t,y)|\lesssim \frac{1}{h}\left(\frac{h}{t}\right)^{1/2}\frac{1}{w_{k}^{1/2}h^{1/3}}.
		\end{equation}
	\end{lema}
\end{proof}
\begin{proof}[Proof of lemma \ref{boundgamma}]
	We consider the change of variables $y=hw_{k}^{3/2}\Sigma$, $x=h^{2/3}w_{k}Z$, \\
	$H_{Z,w_k}(\Sigma)=w_k^{3/2} h \Gamma_{h^{2/3}w_{k}Z,w_k}(hw_k^{3/2}\Sigma)$. Then we have, 
	\begin{align}
		H_{Z,w_k}(\Sigma)=\frac{w_k^{3/2}}{h^{1/3}}\int \frac{\eta^{1/3}}{\sqrt{L_{\beta(\eta/h)}'(w_{\beta(\frac{\eta}{h}),k})}} e^{i{w_{k}^{3/2}\Sigma\eta}}Ai(Z\eta^{2/3}w_k-w_{\beta(\frac{\eta}{h}),k})\sqrt\psi_1(\eta)d\eta\\
		=\frac{w_k^2}{h^{1/3}}\int \frac{\eta^{1/3}}{\sqrt{L_{\beta(\eta/h)}'(w_{\beta(\frac{\eta}{h}),k})}} e^{i{w_{k}^{3/2}(\Sigma\eta+\sigma Z\eta^{2/3}-
				\sigma\frac{w_{\beta(\frac{\eta}{h}),k}}{w_k}+\sigma^3/3)}}\sqrt\psi_1(\eta)d\eta d\sigma.
	\end{align}
	With this change of variables, we can rewrite the problem in terms of $H$:
	\begin{align}
		&||\Gamma_{x,w_k}(y)||_{L^1}=\int |\Gamma_{x,w_k}(y)|dy=\int w_k^{3/2}h|\Gamma_{x,w_k}(hw_k^{3/2}\Sigma)|d\Sigma\\
		&=\int |H_{xh^{-2/3}w_k^{-1},w_k}(\Sigma)|d\Sigma=||H_{xh^{-2/3}w_k^{-1},w_k}(\Sigma)||_{L^1}
	\end{align}
	and
	\begin{equation}
		\sup_x||\Gamma_{x,w_k}||_{L^1_y}=\sup_{x}||H_{xh^{-2/3}w_k^{-1},k}(\Sigma)||_{L^1}=\sup_{xh^{-2/3}w_k^{-1}}||H_{xh^{-2/3}w_k^{-1},k}(\Sigma)||_{L^1}=\sup_{Z}||H_{Z,w_k}(\Sigma)||_{L^1}.
	\end{equation}
	Thus, the study of $\Gamma_{x,k}$ is equivalent to the study $H_{Z,w_k}$.\\
	We will apply stationary phase methods, so we need to study the phase and the symbol:\\
	For the symbol, we consider the following:
	\begin{align}
		&\frac{1}{\sqrt{L_{\beta(\eta/h)}'(w_{\beta(\frac{\eta}{h}),k})}}=\frac{1}{\sqrt{2}w_{\beta(\frac{\eta}{h}),k}^{1/4}\sqrt{1+\frac{3}{2}B_{\beta(\eta/h)}'(w_{\beta(\frac{\eta}{h}),k})}}\\
		&=\frac{\tilde{O}^s_0(b\frac{\eta}{h},w_{\beta(\frac{\eta}{h}),k}^{-1})}{\sqrt{2}w_{\beta(\frac{\eta}{h}),k}^{1/4}}.
	\end{align}
	Now, since the derivatives $\partial_\eta w_{\beta(\frac{\eta}{h}),k}$ are small, we have $\frac{1}{\sqrt{L_{\beta(\eta/h)}'(w_{\beta(\frac{\eta}{h}),k})}}\sim w_k^{-1/4}$ and it is harmless as a symbol.\\
	We study then the phase $\Phi_{h,k,Z,\Sigma}(\sigma,\eta)=(\Sigma\eta+\sigma Z\eta^{2/3}-
	\sigma\frac{w_{\beta(\frac{\eta}{h}),k}}{w_k}+\sigma^3/3)$:
	\begin{align}
		&\partial_\sigma\Phi_{h,k,Z}(\sigma,\eta)= Z\eta^{2/3}-
		\frac{w_{\beta(\frac{\eta}{h}),k}}{w_k}+\sigma^2,\\
		&\partial_\eta\Phi_{h,k,Z}(\sigma,\eta)=\Sigma+\frac{2}{3}\sigma Z \eta^{-1/3}-\sigma\frac{\partial_\eta w_{\beta(\frac{\eta}{h}),k}}{w_k},\\
		&\partial^2_\sigma\Phi_{h,k,Z}(\sigma,\eta)= 2\sigma,\\
		&\partial^2_{\sigma,\eta}\Phi_{h,k,Z}(\sigma,\eta)= \frac{2}{3}Z\eta^{-1/3}-
		\frac{\partial_\eta w_{\beta(\frac{\eta}{h}),k}}{w_k},\\
		&\partial_{\eta,\eta}\Phi_{h,k,Z}(\sigma,\eta)=-\frac{2}{9}\sigma Z \eta^{-4/3}-\sigma\frac{\partial_{\eta,\eta} w_{\beta(\frac{\eta}{h}),k}}{w_k}.\\
	\end{align}
	We consider three different cases
	\begin{enumerate}
		\item $100>Z>w_k^{-3/2}$.\\
		If $|\Sigma|\ge 10Z$ the phase is non-stationary in $\eta$, $|\partial_\eta \Phi_{h,k,Z}(\sigma,\eta)|\gtrsim |\Sigma|$. Integrating by parts twice, it follows that $|H_{Z,w_k}(\Sigma)|\lesssim \frac{1}{w_k^{3}\Sigma^2}$.\\
		For $|\Sigma|\lesssim Z$, 
		we consider the Hessian matrix of the phase on the critical points,  $H\Phi_{h,k,Z}(\sigma_c,\eta_c)$. We have $$\det(H\Phi_{h,k,Z})(\sigma_c,\eta_c)=-\frac{4}{9}\eta_c^{-4/3}Z-2\sigma_c^2 \frac{\partial_{\eta,\eta} w_{\beta(\eta_c/h),k}}{w_k}-\frac{4}{3}Z\eta_c^{-1/3}\frac{\partial_{\eta} w_{\beta(\eta_c/h),k}}{w_k}-\frac{(\partial_{\eta} w_{\beta(\eta_c/h),k})^2}{w_k^2}.$$
		Furthermore, if we compute the eigenvalues of the Hessian:
		\begin{align}
			\lambda_{\pm}=\frac{(\partial^2_{\sigma,\sigma}+\partial^2_{\eta,\eta})\Phi_{h,k,Z}(\sigma_c,\eta_c)\pm \sqrt{(\partial^2_{\sigma,\sigma}+\partial^2_{\eta,\eta})\Phi_{h,k,Z}(\Sigma,\eta)^2-4\det(H\Phi_{h,k,Z})(\sigma_c,\eta_c)}}{2}.
		\end{align}
		With all this, we can analyze the stationary phase.\\
		\begin{obs}\label{simplifications}
			Since the large parameter is $w_k$ and we know that $1\ge \frac{ w_{\beta(\frac{\eta}{h}),k}}{w_k}\rightarrow 1$, uniformly in $h,\eta$, we may assume that $w_k>10000$ and $\frac{(w_{\beta(\frac{\eta}{h}),k})^2}{w_k^2}\in (1/2,1)$.\\
			Finally, $|\sigma|\le1$ otherwise, the phase is non-stationary. Finally, we also use that $\eta \sim 1$ because of the cut-off.
		\end{obs}
		With all this, we obtain the desired bounds. First, since $ w_k^{-3/2}\lesssim Z\lesssim 100$ plus the bounds in proposition \ref{boundszeroes}, we have 
		$-100\lesssim\det(H\Phi_{h,k,Z})(\sigma_c,\eta_c)\lesssim -w_k^{-3/2}$. Now the trace $|(\partial^2_{\sigma,\sigma}+\partial^2_{\eta,\eta})\Phi_{h,k,Z}(\sigma_c,\eta_c)|\le 100$ although the sign is not fixed. With all this, we can bound the eigenvalues from below easily. We consider $(\partial^2_{\sigma,\sigma}+\partial^2_{\eta,\eta})\Phi_{h,k,Z}(\sigma_c,\eta_c)$ positive, the case negative is analogous swapping the roles of $\lambda_+$ and $\lambda_-$:
		\begin{align}
			&\lambda_+\ge \sqrt{|\det(H\Phi_{h,k,Z})(\sigma_c,\eta_c)}|\ge {\sqrt{Z}}.\\
			&|\lambda_-|=\frac{2|\det(H\Phi_{h,k,Z})(\sigma_c,\eta_c)|}{(\partial^2_{\sigma,\sigma}+\partial^2_{\eta,\eta})\Phi_{h,k,Z}(\sigma_c,\eta_c)+ \sqrt{(\partial^2_{\sigma,\sigma}+\partial^2_{\eta,\eta})\Phi_{h,k,Z}(\sigma_c,\eta_c)^2-4\det(H\Phi_{h,k,Z})(\sigma_c,\eta_c)}}\ge\\
			&\ge \frac{2|\det(H\Phi_{h,k,Z})(\sigma_c,\eta_c)|}{100+\sqrt{100^2-4\cdot 100}}\ge \frac{\cdot Z}{100(100+\sqrt{100^2-4\cdot 100})}\gtrsim Z.
		\end{align}
		Now, under the conditions of remark \ref{simplifications}, when $w_k^{3/2}Z\ge 1$, we can apply stationary phase and get 
		\begin{align}
			|H_{Z,w_k}(\Sigma)|\lesssim \frac{w_k^2h^{-1/3}}{{w_k}^{1/4}\sqrt{Z}}{w_k^{-3/2}}=h^{-1/3}w_k^{1/4}\frac{1}{\sqrt{Z}}.
		\end{align}
		This allows us to conclude about the integral
		\begin{align}
			&\int |H_{Z,w_k}(\Sigma)|d\Sigma=\int_{|\Sigma|\lesssim Z} |H_{Z,w_k}(\Sigma)|d\Sigma+\int_{|\Sigma|\gtrsim Z}|H_{Z,w_k}(\Sigma)|d\Sigma\\
			&\lesssim \int_{|\Sigma|\lesssim Z} h^{-1/3}w_k^{1/4}\frac{1}{\sqrt{Z}}+\int_{|\Sigma|\gtrsim Z}  \frac{w_k^2}{w_k^{1/4}h^{1/3}}\frac{1}{w_k^3\Sigma^2}\lesssim h^{-1/3}w_k^{1/4}\sqrt{Z}+\frac{w_k^2}{h^{1/3}w_k^{1/4}}\frac{1}{Z w_k^3}\\
			&\lesssim h^{-1/3}w_k^{1/4}+\frac{w_k^2}{h^{1/3}w_k^{1/4}}\frac{w_k^{3/2}}{w_k^3}\lesssim h^{-1/3}w_k^{1/4}.
		\end{align}
		\item If $Z>{100}$.\\
		We have $Z\eta^{2/3}-\frac{ w_{\beta(\frac{\eta}{h}),k}}{w_k}\ge 10$ so $\partial_\sigma\Phi_{h,k,Z}(\sigma,\eta)\ge 10$, and the phase is non-stationary.\\
		\item If $Z\le w_k^{-3/2}$.\\
		We have $|Z\eta^{2/3}-\frac{w_{\beta(\frac{\eta}{h}),k}}{w_k}|\ge \frac{1}{4}$, that is $|\sigma_c|\ge \frac{1}{2}$, and $|\partial^2_\sigma\Phi_{h,k,Z}(\sigma_c,\eta)|= 2|\sigma_c|\ge \frac{1}{2}$, and we can apply stationary phase on $\sigma$. Or more easily (and formally is the same idea), we can recover the Airy function and using that $\Psi_\pm(x)=\frac{1}{2}\frac{A_\pm(x)}{e^{\mp i|x|^{3/2}}}$, and we have $\Psi_\pm(x)=\frac{1}{2}\frac{Ai(x)}{e^{\mp i|x|^{3/2}}}={O}^\infty(\frac{1}{1+|x|^{1/4}})$ and $Ai(x)=\sum_{\varepsilon=\pm} A_\varepsilon(x)$.
		\begin{align}
			&H_{Z,w_k}(\Sigma)=\frac{w_k^{3/2}}{h^{1/3}}\int \frac{\eta^{1/3}}{\sqrt{L_{\beta(\eta/h)}'(w_{\beta(\frac{\eta}{h}),k})}} e^{i{w_{k}^{3/2}\Sigma\eta}}Ai(Z\eta^{2/3}w_k-w_{\beta(\frac{\eta}{h}),k})\sqrt\psi_1(\eta)d\eta\\
			&=\frac{w_k^{3/2}}{h^{1/3}}\sum_{\varepsilon=\pm}\int \frac{\sqrt\psi_1(\eta)\eta^{1/3}}{\sqrt{L_{\beta(\eta/h)}'(w_{\beta(\frac{\eta}{h}),k})}} e^{iw_{k}^{3/2}(\Sigma\eta+\varepsilon(\frac{w_{\beta(\frac{\eta}{h}),k}}{w_k}-Z\eta^{2/3})^{3/2})}\Psi_\varepsilon(Z\eta^{2/3}w_k-w_{\beta(\frac{\eta}{h}),k})d\eta.
		\end{align}
		The derivative of the new phase $\tilde\Phi_{\varepsilon,h,k,Z}(\eta)=(\Sigma\eta+\varepsilon(\frac{w_{\beta(\frac{\eta}{h}),k}}{w_k}-Z\eta^{2/3})^{3/2})$ is 
		\begin{align}
			\tilde{\Phi}_{\varepsilon,h,k,Z}'(\eta)=\Sigma-\varepsilon(\frac{w_{\beta(\frac{\eta}{h}),k}}{w_k}-Z\eta^{2/3})^{1/2}Z\eta^{-1/3}+\varepsilon(\frac{w_{\beta(\frac{\eta}{h}),k}}{w_k}-Z\eta^{2/3})^{1/2}\frac{\partial_\eta w_{\beta(\frac{\eta}{h}),k}}{w_k}.
		\end{align}
		Since in this case $Z<w_k^{-3/2}$ and $|\partial_\eta w_{\beta(\frac{\eta}{h}),k}|\lesssim \frac{h^{2/3}}{w_k}$, we have $|\Sigma|\lesssim w_k^{-3/2}$ or, we have non-stationary phase and decay. Thus, for $|\Sigma|\lesssim w_k^{-3/2}$ we bound directly
		\begin{align}
			&|H_{Z,w_k}(\Sigma)|\\
			&=\sum_{\varepsilon=\pm}\left|\frac{w_k^{3/2}}{h^{1/3}}\int \frac{\Psi_\varepsilon(Z\eta^{2/3}w_k-w_{\beta(\frac{\eta}{h}),k})\sqrt\psi_1(\eta)\eta^{1/3}}{\sqrt{L_{\beta(\eta/h)}'(w_{\beta(\frac{\eta}{h}),k})}} e^{iw_{k}^{3/2}(\Sigma\eta+\varepsilon(\frac{w_{\beta(\frac{\eta}{h}),k}}{w_k}-Z\eta^{2/3})^{3/2})}d\eta\right|\\
			&\le \sum_{\varepsilon=\pm}\frac{w_k^{3/2}}{h^{1/3}}\int\left| \frac{\Psi_\varepsilon(Z\eta^{2/3}w_k-w_{\beta(\frac{\eta}{h}),k})\sqrt\psi_1(\eta)\eta^{1/3}}{\sqrt{L_{\beta(\eta/h)}'(w_{\beta(\frac{\eta}{h}),k})}} \right|d\eta\lesssim \frac{w_k^{3/2}}{h^{1/3}w_k^{1/2}}\frac{1}{w_k^{1/4}}=\frac{w_k^{3/4}}{h^{1/3}}.
		\end{align}
		For $|\Sigma|\gg w_k^{-3/2}$, we have the estimates $|H_{Z,w_k}(\Sigma)|\lesssim \frac{w_k^{3/2}}{h^{1/3}w_k^{1/2}}\frac{1}{\Sigma^{n}w_k^{3n/2}}$.
		Therefore, in our case $Z<w_k^{-3/2}$,
		\begin{align}
			&\int |H_{Z,w_k}(\Sigma)|d \Sigma=\int_{|\Sigma|\lesssim w_k^{-3/2}} |H_{Z,w_k}(\Sigma)|d \Sigma+\int_{|\Sigma|\gtrsim w_k^{-3/2}} |H_{Z,w_k}(\Sigma)|d \Sigma\\
			&\lesssim\int_{|\Sigma|\lesssim w_k^{-3/2}} \frac{w_k^{3/4}}{h^{1/3}} d \Sigma+ \int_{|\Sigma|\gtrsim w_k^{-3/2}} \frac{w_k^{3/2}}{h^{1/3}w_k^{1/2}}\frac{1}{\Sigma^2w_k^{3}}\lesssim \frac{w_k^{3/4}}{h^{1/3}}w_k^{-3/2}+\frac{w_k^{3/2}}{h^{1/3}w_k^{1/2}} \frac{1}{w_k^3w_k^{-3/2}}\\
			&\lesssim \frac{1}{h^{1/3}w_k^{3/4}}.
		\end{align}
	\end{enumerate}
\end{proof}
Since this estimate is better than the other, we conclude that
\begin{equation}
	\sup_x||\Gamma_{x,k}||_{L^1_y}=\sup_{Z}||H_{Z,w_k}(\Sigma)||_{L^1_\Sigma}\lesssim h^{-1/3}w_k^{1/4}.
\end{equation}

\begin{proof}[Proof of \eqref{otherlemmagallery}]
	This is easier, we only need to check that the phase in $W_k$ is in condition to apply Van der Corput or stationary phase.
	\begin{align}
		\label{equlowmodes}
		&\phi_h(\eta)=\eta\sqrt{1+w_{\beta(\frac{\eta}{h}),k}(h/\eta)^{2/3}}+\frac{y}{t}\eta=\frac{y}{t}\eta+\sqrt{\eta^2+\eta^{4/3}w_{\beta(\frac{\eta}{h}),k}h^{2/3}}\\
		&\phi'_h(\eta)=\frac{y}{t}+\frac{2\eta+\frac{4}{3}\eta^{1/3}h^{2/3}w_{\beta(\frac{\eta}{h}),k}+\partial_\eta w_{\beta(\frac{\eta}{h}),k}h^{2/3}\eta^{4/3}}{2\sqrt{\eta^2+\eta^{4/3}w_{\beta(\frac{\eta}{h}),k}h^{2/3}}}\\
		&\phi''_h(\eta)=\frac{2+\frac{4}{9}\eta^{-2/3}h^{2/3}w_{\beta(\frac{\eta}{h}),k}+h^{2/3}\frac{4}{3}\eta^{1/3}\partial_\eta w_{\beta(\frac{\eta}{h}),k} +h^{2/3}\eta^{4/3}\partial^2_\eta w_{\beta(\frac{\eta}{h}),k}}{2\sqrt{\eta^2+\eta^{4/3}w_{\beta(\frac{\eta}{h}),k}h^{2/3}}}\\
		&-\frac{(2\eta+\frac{4}{3}\eta^{1/3}h^{2/3}w_{\beta(\frac{\eta}{h}),k}+\partial_\eta w_{\beta(\frac{\eta}{h}),k}h^{2/3}\eta^{4/3})^2}{4(\eta^2+\eta^{4/3}w_{\beta(\frac{\eta}{h}),k}h^{2/3})^{3/2}}\gtrsim h^{2/3}w_k.\\
	\end{align}
	Therefore by Van der Corput lemma, we obtain the desired result.
\end{proof}
theorem, \ref{stricharzgallery} can be generalized to $L^2$ functions by density.
\begin{cor}
	\label{corollarygallery}
	Let $u\in L^2(\Omega_2)$, we define the operator 
	\begin{align}
		T_{gall}(u)(t,x,y):=\frac{1}{h}\int \tilde{u}_k(\eta/h) \psi_1(\eta) e^{\pm it\sqrt{\lambda_{k,\iota}(\eta/h)}}e^{iy\eta/h}e_{k}(x,\eta/h) d\eta.
	\end{align}
With $\tilde{u}_k(\eta/h):=\int \mathcal{F}_y u(x,\eta/h) e_k(x,\eta/h,\iota)dx$
Then for every $\iota,q,r,k$ satisfying $q\ge 4, \frac{1}{q}=\frac{1}{2}\left(\frac{1}{2}-\frac{1}{r}\right)$, $h^{2/3}w_k\le 10$ and $\beta_{Str}=2\left(\frac{1}{2}-\frac{1}{r}\right)-\frac{1}{q}$, there exists a $C$ independent of $q,r,k$ and $u_0$ such that 
\begin{equation}
	||T_{gall}(u)(t,x,y)||_{L^q([0,t_{max}],L^r(\Omega_2))}\le C h^{-\beta_{Str}}||T_{gall}(u)(0,x,y)||_{L^2(\Omega_2)}.
\end{equation}
\end{cor}
Using this, we generalize theorem \ref{stricharzgallery} by a simple orthogonality argument with respect to those spaces as in \cite{Ivanovici2010Counterexamples}, which allows us to balance the number of terms in the sum over gamma with the gain in the spectral localization.
We consider the following operator: 
\begin{align}
	&T_{G_{h,\gamma}}(u)(t,x,y)=\frac{1}{h}\sum_{k\ge1}\int  \tilde{u}_k(\eta/h)\psi_0(h^2\lambda_{k,\iota}(\eta/h)) \psi_1(\eta) \psi_2\left(\frac{h^{2/3}w_{{\beta(\iota,\eta/h),k}}}{\eta^{2/3}\gamma}\right)e^{\pm it\sqrt{\lambda_{k,\iota}(\eta/h)}}e^{iy\eta/h}e_{k}(x,\eta/h) d\eta\\
	&=\sum_{k\ge 1} T_{G_{h,\gamma,k}}(u)(t,x,y).
\end{align}
\begin{obs}
	We note that this operator is essentially the operator of applying the Green's function we define later, def \ref{Goperator}.
\end{obs}
For every $k$, we apply the gallery modes theorem with $r=\infty, q=4$, proving the following
\begin{align}
	\left|\left|T_{G_{h,\gamma,k}}(u)(t,x,y) \right|\right|_{L^4_tL^\infty(\Omega_2)}\le C h^{-1+\frac{1}{4}} \left|\left| T_{G_{h,\gamma,k}}(u)(0,x,y)\right|\right|_{L^2(\Omega_2)}
\end{align}
Indeed, because of the cut-off $\psi_0$, all non-zero terms  $T_{G_{h,\gamma,k}}(u)(t,x,y)$ follow $h^{2/3}w_k\le 1/100$. Thus we can apply \ref{corollarygallery}.
With this, we can prove the following:
\begin{prop}
	There exists a universal constant $C$ depending only on $\iota$ and the fixed cutoffs, such that 
	\begin{align}
		||T_{G_{h,\gamma}}(u)(t,x,y)||_{L^4_t L^\infty(\Omega_2)}\le C h^{\frac{1}{4}-1} \lambda_{\gamma,h}^{1/2} ||T_{G_{h,\gamma}}(u)(0,x,y)||_{L^2(\Omega_2)}\le h^{\frac{1}{4}-1} \lambda_{\gamma,h}^{1/2}C ||u||_{L^2(\Omega_2)}.
	\end{align}
\end{prop}
\begin{proof}
	Because of the cut-off $\psi_2$, we know that the sum in the definition of $	T_{G_{h,\gamma}}(u)(t,x,y)$ is just a finite sum, $	T_{G_{h,\gamma}}(u)(t,x,y)=\sum_{\lambda_{\gamma,h} \gtrsim k\ge 1} T_{G_{h,\gamma,k}}(u)(t,x,y)$. By Minkowski's inequality we have,
	\begin{align}
			||T_{G_{h,\gamma,k}}(u)||_{L^4_t L^\infty(\Omega_2)}\le C h^{\frac{1}{4}-1} \sum_{k
			\lesssim \lambda_{\gamma,h}} ||T_{G_{h,\gamma,k}}(u)(0,\cdot)||_{L^2(\Omega_2)}.
	\end{align}
By discrete Cauchy-Schwarz, it follows
	\begin{align}
	||T_{G_{h,\gamma}}(u)||_{L^4_t L^\infty(\Omega_2)}\le C h^{\frac{1}{4}-1} \sum_{k
		\lesssim \lambda_{\gamma,h}} ||T_{G_{h,\gamma,k}}(u)(0,\cdot)||_{L^2(\Omega_2)}\le \left(\sum_{k
		\lesssim \lambda_{\gamma,h}} 1\right)^{1/2}\left(\sum_{k
		\lesssim \lambda_{\gamma,h}} ||T_{G_{h,\gamma,k}}(u)(0,\cdot)||^2_{L^2(\Omega_2)}\right)^{1/2}.
\end{align}
Since for every $k$, $T_{G_{h,\gamma,k}}(u)(0,\cdot)\in E_k^h$, and the $E_k^h$ are orthogonal spaces, we have \\
$\left(\sum_{k\ge 1} ||T_{G_{h,\gamma,k}}(u)(0,\cdot)||^2_{L^2(\Omega_2)}\right)^{1/2}\le ||T_{G_{h,\gamma}}(u)(0,\cdot)||_{L^2(\Omega_2)}\le ||u||_{L^2(\Omega_2)}$ which allows us to conclude.
\end{proof}
Lemma \ref{grossboundgallmodes}, by standard $TT^*$ arguments, gives, $||T_{G_{h,\gamma}}(u)||_{L^\infty_t L^\infty(\Omega_2)}\lesssim \frac{\gamma^{1/4}}{h}||u||_{L^2(\Omega_2)}$. By interpolating this bound with the one we have just obtained, we obtain the Strichartz for gallery modes.
\begin{prop}\label{gallerystr}
	For any $q\ge 4+\frac{12}{19}$, we obtain Strichartz estimates for the Green's function
	\begin{align}
		\left|\left|\sum_{\gamma\le h^{\frac{2}{3}-\frac{1}{30}}} T_{G_{\gamma,h}}(u_0)\right|\right|_{L^q_t L^\infty_{x,y}}\lesssim h^{\frac{1}{q}-1}||u_0||_{L^2}.
	\end{align}
\end{prop}
\section{Strichartz Estimates}
\label{strichartz}
We now use the preceding estimates to prove the desired Strichartz estimates.\\
Let $u\in L^2(\Omega_2)$ be arbitrary. We want to prove that there exists a constant $C$ independent of $u$ such that
\begin{align}
	\left|\left| \int G_h(t,x,y-b,a)u(a,b)dadb\right|\right|_{L^q([0,t_{max}]),L^\infty(\Omega_2)}\le C h^{\frac{1}{q}-1}||u||_{L^2(\Omega_2)}
\end{align}
for $\frac{1}{q}\le \left(\frac{1}{2}-\frac{1}{9}\right)\frac{1}{2}$.\\
In this case, it is more useful to see $G_h$ as an operator:
\begin{defi}
	\label{Goperator}
	We define the following functional linear operator from $L^2(\Omega_2)$ to $L^p_t(L^q_{x,y}(\Omega_2))$, for certain $p,q$:
	\begin{align}
		T_{G_h}(u)(t,x,y)=\int G_h(t,x,y-b,a)u(a,b)dadb.
	\end{align}
\end{defi}
We have some classical properties on this operator:
\begin{prop}
	The operator is $L^2$ bounded, that is, for every $t\in[0,t_{max}]$, we have
	\begin{align}
		\label{almostconsenergy}
		||T_{G_h}(v)(t,\cdot)||_{L^2(\Omega_2)}=||T_{G_h}(v)(0,\cdot)||_{L^2(\Omega_2)}\le ||v||_{L^2(\Omega_2)}.
	\end{align}
	Furthermore, for every $t_0\in [0,t_{max}]$ there exists a $v_{t_0}\in L^2(\Omega_2)$ such that $||v_{t_0}||_{L^2(\Omega_2)}=||v||_{L^2(\Omega_2)}$ and $T_{G_h}(v)(t_0+t,x,y)=T_{G_h}(v_{t_0})(t,x,y)$, for every $t_{max}-t_0\ge t\ge 0$.
\end{prop} 
The adjoint of the operator is as follows
\begin{prop}
	\label{adjoin}
	\begin{align}
		T^\ast_{G_h}(f)(x,y)=\int G_h(-s,x,y-b,a)f(s,a,b)dadbds.
	\end{align}
	In consequence, the operator $T_{G_h}T^\ast_{G_h}(f)$ is
	\begin{align}
		T_{G_h}T^\ast_{G_h}(f)(t,x,y)=\int G_h(t-s,x,y-b,a)f(s,a,b)dadbds.
	\end{align}
\end{prop}
\begin{proof}
	The adjoint is a simple calculation. We focus on proving that the composition have this expression. Let $f\in L^p_t(\mathbb{R}),L^q_{x,y}(\Omega_2)$. We have
	\begin{align}
		T^*_{G_h}(f)=\int G_{h}(-s,x,\cdot,a)\ast_y f(a,s,\cdot) dads,
	\end{align}
	so
	\begin{align}
		&\mathcal{F}_y(T^*_{G_h}(f))(x,\theta)=\int \mathcal{F}_y(G_{h})(-s,x,\theta,a) f(a,s,\theta) dads\\
		&=\sum_{k\ge 2} \psi_0(h^2 \lambda_k(\theta)) 	\psi_1(h\theta)e_k(x,\theta,\iota){}\int e_k(a,\theta,\iota)e^{-is\sqrt{\lambda_k(\theta)}}\mathcal{F}_y(f)(s,a,\theta)dads\\
		&=\sum_{k\ge 2} \psi_0(h^2 \lambda_k(\theta))\psi_1(h\theta)e_k(x,\theta,\iota)f_k(s,\theta,\iota).
	\end{align}
	Where $f_k(s,\theta,\iota)=\int e_k(a,\theta,\iota)e^{-is\sqrt{\lambda_k(\theta)}}\mathcal{F}_y(f)(s,a,\theta)dads$.
	Also, if we denote \\$F_m(x,t,\theta,\iota)= e_m(x,\theta,\iota)e^{it\sqrt{\lambda_m(\theta)}}$,
	\begin{align}
		&\mathcal{F}_y(T_{G_h}T^\ast_{G_h}(f))(t,x,\theta)=\int \mathcal{F}_y(G_h)(t,x,\theta,a)\mathcal{F}_y(T^*_{G_h}(f))(a,\theta)da\\
		&=\int \mathcal{F}_y(T^*_{G_h}(f))(\tilde{a},\theta)\sum_{m\ge 2} \psi_0(h^2 \lambda_m(\theta)) 	\psi_1(h\theta)e_m(x,\theta,\iota) e_m(\tilde{a},\theta,\iota)e^{it\sqrt{\lambda_m(\theta)}} d\tilde{a}\\
		&=\int \sum_{m\ge 2}\sum_{k\ge 2} \psi_0(h^2 \lambda_k(\theta))\psi_0(h^2 \lambda_m(\theta))\psi_1^2(h\theta) e_m(\tilde{a},\theta,\iota)e_k(\tilde{a},\theta,\iota)F_m(x,t,\theta,\iota)f_k(s,\theta,\iota)d\tilde{a}.
	\end{align}
	\begin{align}
		&=\sum_{m\ge 2}\sum_{k\ge 2} \psi_0(h^2 \lambda_k(\theta))\psi_0(h^2 \lambda_m(\theta))\psi_1^2(h\theta)F_m(x,t,\theta,\iota)f_k(s,\theta,\iota)\int e_m(\tilde{a},\theta,\iota)e_k(\tilde{a},\theta,\iota) d\tilde{a}
	\end{align}
	Since $e_k(x,\theta,\iota)$ is an orthonormal basis in $x$ for fixed $\theta,\iota$, that is $\int e_m(\tilde{a},\theta,\iota)e_k(\tilde{a},\theta,\iota) d\tilde{a}=0$, if $k\ne m$ and the integral is 1 in $k=m$, we have
	\begin{align}
		&\mathcal{F}_y(T_{G_h}T^\ast_{G_h}(f))(t,x,\theta)\\
		&=\sum_{k\ge 2} \psi_0^2(h^2 \lambda_k(\theta)) 	\psi_1^2(h\theta)F_k(x,t,\theta,\iota)f_k(s,\theta,\iota)\\
		&=\sum_{k\ge 2} \psi_0^2(h^2 \lambda_k(\theta)) 	\psi_1^2(h\theta)e_k(x,\theta,\iota)\int e_k(a,\theta,\iota)e^{-is\sqrt{\lambda_k(\theta)}}\mathcal{F}_y(f)(s,a,\theta)e^{it\sqrt{\lambda_k(\theta)}}dads\\
		&=\int \sum_{k\ge 2} \psi_0^2(h^2 \lambda_k(\theta)) 	\psi_1^2(h\theta)e_k(x,\theta,\iota) e_k(a,\theta,\iota)\mathcal{F}_y(f)(s,a,\theta)e^{i(t-s)\sqrt{\lambda_k(\theta)}}dads\\
		&=\int \mathcal{F}_y(G_h)(t-s,x,\theta,a)\mathcal{F}_y(f)(s,a,\theta) dads.
	\end{align}
	Therefore, since the inverse Fourier transform in $y$ transforms products in convolutions, we obtain the result.
\end{proof}
\begin{obs}
	Strictly speaking, the kernel $G$ which appears in the $TT^*$ argument is not exactly the same kernel as the one considered above: the symbol may differ by an equivalent cutoff. In particular, the kernel in the $TT^*$ argument contains the cutoff $\psi_1^2(h\theta)$ instead of $\psi_1(h\theta)$. This does not affect the estimates, since all the properties and asymptotics of $G$ were proved for an arbitrary cutoff
	$\psi_1\in \mathcal{C}^\infty_0\bigl(\sqrt{1-2^{-10}},\sqrt{1+2^{-10}}\bigr)$,
	and $\psi_1^2$ belongs to the same class. Therefore the constants in the estimates may change, but the bounds remain unchanged.
	
\end{obs}
\begin{obs}
	The operator $T_{G_{\gamma,h}}$ and similar operators satisfies the same definition and their adjoints have the same structure.
\end{obs}
With all these considerations, we can use our dispersive estimates to prove Strichartz.
We start by proving here the case $L^\infty(\Omega_2)$ and after that, we deduce the other estimates by interpolation with the energy.
\subsection{$L^\infty(\Omega_2)$: Strichartz estimates for small times }
The first statements of propositions (\ref{bigtangenctial}), (\ref{disptransverse}) tells us that, when $|t|\le C\sqrt{\gamma}$ (For a known fixed constant $C$), we have $|G_{\gamma,h}(t,x,y,a)|\lesssim \frac{\sqrt{\gamma}}{h^2} \inf\left(1,\frac{h^{1/2}}{t^{1/2}\gamma^{1/2}}\right)$. This expression of the dispersion has the same aspect as in the classical case, so the same argument applies:
\begin{align}
	&||T_{G_h}T^\ast_{G_{\gamma,h}}(f)||^q_{L^q([0,C\sqrt{\gamma}]),L^\infty(\Omega_2)}=\int_0^{C\sqrt{\gamma}} \left|\left|\int G_h(t-s,x,y-b,a)f(s,a,b)dadbds\right|\right|_{L^\infty_{x,y}(\Omega_2)}^q dt\\
	&\lesssim \int_0^{C\sqrt{\gamma}} \left(\int ||G_h(t-s,x,y-b,a)||_{L^\infty_{x,y}(\Omega_2)}|f(s,a,b)|dadbds \right)^q dt\\
	&\lesssim \int_0^{C\sqrt{\gamma}} \left(\int_0^{C\sqrt{\gamma}} \frac{h^{1/2}}{h^2}\frac{1}{|t-s|^{1/2}}|f(s,a,b)|dadbds \right)^q dt\\
	&\lesssim\frac{h^{q/2}}{h^{2q}} \int_0^{C\sqrt{\gamma}}\left(\int \frac{1}{|t-s|^{1/2}}||f(s,\cdot)||_{L^1_{a,b}}ds \right)^q dt\\
	&\lesssim\frac{h^{q/2}}{h^{2q}} ||\frac{1}{|t|^{1/2}}\ast_t |f|(t,a,b)||^q_{L^q([0,C\sqrt{\gamma}]),L^1_{x,y}}\lesssim \frac{h^{q/2}}{h^{2q}} ||f(t,a,b)||^q_{L^{(\frac{1}{q}+\frac{1}{2})^{-1}}([0,C\sqrt{\gamma}]),L^1_{a,b}}.
\end{align}
In order for $q$ and $(\frac{1}{q}+\frac{1}{2})^{-1}$ to be conjugated, we need $q=4$.\\
\begin{obs}
	We focus only on the minimum $q$ possible, since for larger $q$ we can simply use interpolation with the energy estimates. 
\end{obs}
Thus, for every $f\in L^{\frac{4}{3}}([0,C\sqrt{\gamma}]),L^\infty_{a,b}(\Omega_2)$, we obtain
\begin{align}
	||T_{G_{\gamma,h}}T^\ast_{G_{\gamma,h}}(f)||_{L^4([0,C\sqrt{\gamma}]),L^\infty(\Omega_2)}\lesssim \frac{h^{1/2}}{h^2}||f||_{L^{\frac{4}{3}}([0,C\sqrt{\gamma}]),L^\infty_{a,b}(\Omega_2)}.
\end{align}
By a $TT^*$ argument, it follows that
\begin{align}
	||T_{G_{\gamma,h}}(u)||_{L^4([0,C\sqrt{\gamma}]),L^\infty(\Omega_2)}\lesssim \frac{h^{1/4}}{h}||u||_{L^2_{a,b}(\Omega_2)}.
\end{align}
Now, we can use the energy estimate to extend this estimate to an estimate in all the interval $[0,t_{max}]$:\\
Let $[t_0,t_1]\subset[0,t_{max}]$ such that $t_1-t_0\le C\sqrt{\gamma}$. Then
\begin{align}
	||T_{G_{\gamma,h}}(u)||^4_{L^4([t_0,t_1]),L^\infty(\Omega_2)}=\int_{t_0}^{t_1} ||T_{G_{\gamma,h}}(u)(t,\cdot)||_{L^\infty(\Omega_2)}^4 dt=\int_{0}^{t_1-t_0} ||T_{G_{\gamma,h}}(u)(t-t_0,\cdot)||_{L^\infty(\Omega_2)}^4 dt.
\end{align}
We use our expression for the energy estimate and so, there exists a $v_{t_0}$ with the same $L^2$ norm as $u$ such that, for every $t$, $T_{G_{\gamma,h}}(u)(t-t_0,x,y)=T_{G_{\gamma,h}}(v_{t_0})(t,x,y)$. Using that we conclude that
\begin{align}
	&||T_{G_{\gamma,h}}(u)||^4_{L^4([t_0,t_1]),L^\infty(\Omega_2)}=\int_{0}^{t_1-t_0} ||T_{G_{\gamma,h}}(u)(t-t_0,\cdot)||_{L^\infty(\Omega_2)}^4 dt\\
	&=\int_{0}^{t_1-t_0} ||T_{G_{\gamma,h}}(v_{t_0})(t,\cdot)||_{L^\infty(\Omega_2)}^4 dt\lesssim \frac{h}{h^4}||v_{t_0}||^4_{L^2_{a,b}(\Omega_2)}=\frac{h}{h^4}||u_{t_0}||^4_{L^2_{a,b}(\Omega_2)}.
\end{align}
Now, we can divide the total space into parts with size $\le C\sqrt{\gamma}$. $[0,t_{max}]=\bigcup_{1\le N\le \lfloor \frac{t_{max}}{C\sqrt{\gamma}} \rfloor} [(N-1)C\sqrt{\gamma},NC\sqrt{\gamma}]\cup [C\sqrt{\gamma}\lfloor\frac{t_{max}}{ C\sqrt{\gamma}}\rfloor,t_{max}]$ and there are $\lfloor\frac{t_{max}}{ C\sqrt{\gamma}}\rfloor+1\lesssim\frac{t_{max}}{ C\sqrt{\gamma}}$ intervals. Then
\begin{align}
	\label{smalltobigtimes}
	&||T_{G_{\gamma,h}}(u)||^4_{L^4([0,t_{max}]),L^\infty(\Omega_2)}\\
	&=\sum_{1\le N\le \lfloor \frac{t_{max}}{C\sqrt{\gamma}} \rfloor} ||T_{G_{\gamma,h}}(u)||^4_{L^4([(N-1)C\sqrt{\gamma},NC\sqrt{\gamma}]),L^\infty(\Omega_2)}+||T_{G_{\gamma,h}}(u)||^4_{L^4([C\sqrt{\gamma}\lfloor\frac{t_{max}}{ C\sqrt{\gamma}}\rfloor,t_{max}]),L^\infty(\Omega_2)}\\
	&\lesssim \frac{t_{max}}{ C\sqrt{\gamma}} \frac{h}{h^4}||u_{t_0}||^4_{L^2_{a,b}(\Omega_2)}.
\end{align}
\subsection{$L^\infty(\Omega_2)$: Strichartz for longer times for tangential waves}
We study now what we have called the tangential waves, that is $\gamma/4\le a\le 2\gamma$.
\begin{obs}
	We  focus only in the tangential waves since the estimates for transverse waves are easier to get, since the bounds are already in terms of $T$ and not in terms of $N$, and they are strictly better than the tangential ones, thus the tangential waves are the bottleneck of the problem.
\end{obs}
First, by looking at \eqref{bigtangenctial}, we have at most three regimes: $1\le |N| \le \lambda_{a,h}^{1/3}$, $\lambda_{a,h}^{1/3}\le |N| \le \lambda_{a,h}$ and $|N| \ge \lambda_{a,h}$. However, we know that always, the relevant part is $|N|\le C_3 a^{-1/2}$, so if $\lambda_{a,h}^{1/3}\ge C_3 a^{-1/2}$ ($a\ge 10 C_3 h^{1/3}$) the last two regimes  never appear, and if $\lambda_{a,h}\ge C_3 a^{-1/2}$ ($a\ge 10 C_3 h^{1/2}$) the last regime never appears.
Now we compute the norm of $G_{\gamma,h}$ for tangential waves (We  use Young's inequality to obtain an estimate).\\
We compute then the integral, and we do the change of variables \eqref{tangenctialchange}, which we have used for the tangential waves $\hat{\boldsymbol{T}}$.
\begin{align}
	&\int_{C\sqrt{\gamma}}^{t_{max}} \sup_{x,y}\left|G_{\gamma,h}(t,x,y,a) \right|^p dt=\int_{C\sqrt{\gamma}}^{t_{max}} \sup_{x,y}\left|\sum_{N\in \mathcal{N}_1(\boldsymbol{t})} V_{h,\gamma,N}(\boldsymbol{t},a)+O(h^{\infty})\right|^pdt\\
	&\le\int_{C\sqrt{\gamma}}^{t_{max}} \sup_{x,y}\left|\sum_{N\in \mathcal{N}_1(\boldsymbol{t})} V_{h,\gamma,N}(\boldsymbol{t},a)\right|^pdt\\
	&+C(p)O(h^{\infty})=\sqrt{a}\sqrt{1+a}\int_{C\frac{\sqrt{\gamma}}{\sqrt{a}\sqrt{1+a}}}^{\frac{t_{max}}{\sqrt{a}\sqrt{1+a}}} \sup_{\hat{X},\hat{Y}}\left|\sum_{N\in \mathcal{N}_1(\boldsymbol{t}(\hat{\boldsymbol{T}},a))} W_{h,\gamma,N}(\boldsymbol{\hat{T}},a)\right|^pd\hat{T}+C(p)O(h^{\infty}).
\end{align}
We now consider the following:
we know that $\hat{T}$ and $N$ are closely related, however propositions \eqref{prop1} and \eqref{boundgreat} tell us that, for a fixed $\hat{T}$, the sum is bounded, and that $N\le 2K \hat{T}$. However we do not have the opposite $N\ge 2K \hat{T}$ (even if we know it must be true in some sense, since the stationary points are in $T\sim 4N$). Thus, instead of studying the regimes of $N$, we  study the regimes of $\hat{T}$. In order to do so, we  discretize this interval as follows:\\
Consider a partition $C\frac{\sqrt{\gamma}}{\sqrt{a}\sqrt{1+a}}=\hat{T}_0\le \hat{T}_1\le ...\le \hat{T}_m=\frac{t_{max}}{\sqrt{
		a}\sqrt{1+a}}$ such that $\frac{1}{4}= |T_i-T_{i-1}|$. One possible case could be $C\frac{\sqrt{\gamma}}{\sqrt{a}\sqrt{1+a}}\le C\frac{\sqrt{\gamma}}{\sqrt{a}\sqrt{1+a}}+1/2\le ...\le C\frac{\sqrt{\gamma}}{\sqrt{a}\sqrt{1+a}}+\frac{m-1}{2}\le \frac{t_{max}}{\sqrt{
		a}\sqrt{1+a}}$. In that case, we have 
\begin{align}
	\int_{C\frac{\sqrt{\gamma}}{\sqrt{a}\sqrt{1+a}}}^{\lambda_{a,h}^2} \sup_{\hat{X},\hat{Y}}\left|\sum_{N\in \mathcal{N}_1(\boldsymbol{t}(\hat{\boldsymbol{T}},a))} W_{h,\gamma,N}(\boldsymbol{\hat{T}},a)\right|^pd\hat{T}= \sum_{i=0}^{m} \int_{\hat{T}_i}^{\hat{T}_{i+1}} \sup_{\hat{X},\hat{Y}}\left|\sum_{N\in \mathcal{N}_1(\boldsymbol{t}(\hat{\boldsymbol{T}},a))} W_{h,\gamma,N}(\boldsymbol{\hat{T}},a)\right|^pd\hat{T}.
\end{align}
We note that $m$ depends on the parameters $\gamma,a$ and $m\sim \frac{t_{max}}{\sqrt{
		a}\sqrt{1+a}}$, we will call it $m_{a,\gamma}$.\\
Now, in each of the integrals, we have that for all $\hat{T}\in [\hat{T}_{i},\hat{T}_{i+1}]$, we have $\mathcal{N}_1(\boldsymbol{t}(\hat{\boldsymbol{T}},a))\subset \mathcal{N}_1(\boldsymbol{t}(\hat{T}_{i},a))\cup \mathcal{N}_1(\boldsymbol{t}(\hat{T}_{i+1},a))$ and this set has a cardinal bounded by $C+\frac{\hat{T}_{i+1}}{\lambda_{a,h}^2}$. In particular, for $\hat{T}_{i}\lesssim \lambda_{a,h}^2$, the cardinal is fixed by a constant (which is the same constant for every $j\le i$). We then define the following set.
\begin{defi}
	\label{N3}
	For every $0\le i\le m_{a,\gamma}$, we define $\mathcal{N}_{3}(\hat{T}_i):=(\mathcal{N}_1(\boldsymbol{t}(\hat{T}_{i},a))\cup \mathcal{N}_1(\boldsymbol{t}(\hat{T}_{i+1},a)))\cap \{N\in\mathbb{N}: N\le \hat{T}_{i+1}{K} \}$.
\end{defi}
Now we  study the different regimes of the theorem \ref{bigtangenctial} but substituting $N$ by $T_i$.
\begin{enumerate}
	\item $1\le \hat{T}_i\le K^{-1}\lambda_{a,h}^{1/3}$.\\
	In this case, we need to worry about the swallowtail, but we do not need to worry about the overlapping, this means that we can reconfigure the integral using that, for any $\hat{T}_i\le K^{-1}\lambda_{a,h}^{2}$, then $|\mathcal{N}_{3}(\hat{T}_i)|$ is bounded by a constant which does not depend on $i$. Thus we have
	\begin{align}
		&\int_{\hat{T}_i}^{\hat{T}_{i+1}} \sup_{\hat{X},\hat{Y}}\left|\sum_{N\in\mathcal{N}_3(\hat{T}_{i})} W_{h,\gamma,N}(\boldsymbol{\hat{T}},a)\right|^pd\hat{T}\\
		&\le\int_{\hat{T}_i}^{\hat{T}_{i+1}} C(p) \sum_{N\in\mathcal{N}_3(\hat{T}_{i})}\sup_{\hat{X},\hat{Y}}\left| W_{h,\gamma,N}(\boldsymbol{\hat{T}},a)\right|^pd\hat{T}\le\int_{\hat{T}_i}^{\hat{T}_{i+1}} C(p) \sum_{N\in \frac{K}{\sqrt{a}}}\sup_{\hat{X},\hat{Y}}\left| W_{h,\gamma,N}(\boldsymbol{\hat{T}},a)\right|^pd\hat{T},
	\end{align}
	so 
	\begin{align}
		&\sum_{i\in\mathbb{N}|\hat{T}_i\le K^{-1}\lambda_{a,h}^{1/3}} \int_{\hat{T}_i}^{\hat{T}_{i+1}} \sup_{\hat{X},\hat{Y}}\left|\sum_{N\in\mathcal{N}_3(\hat{T}_{i})} W_{h,\gamma,N}(\boldsymbol{\hat{T}},a)\right|^pd\hat{T} \\
		&\lesssim C(p) \sum_{|N|\lesssim T_{max}} \int_{C\frac{\sqrt{\gamma}}{\sqrt{a}\sqrt{1+a}}}^{T_{max}}\sup_{\hat{X},\hat{Y}}\left| \hat{W}_{h,\gamma,N}(\boldsymbol{\hat{T}},a)\right|^pd\hat{T}.
	\end{align}
	\begin{obs}
		$T_{max}$ is the maximum $T$ the integral goes, it is $T_{max}=K^{-1}\lambda_{a,h}^{1/3}$ unless $\lambda_{a,h}^{1/3}\ge \frac{1}{\sqrt{a}}$, in which case, $T_{max}=\frac{1}{\sqrt{a}}$.
	\end{obs}
	\begin{obs}
		The case $\lambda_{a,h}^{1/3}\ge \frac{1}{\sqrt{a}}$, corresponds with the case $h^{1/2}\le a$.
	\end{obs}
	We study this integral:
	\begin{align}
		&\int_{C\frac{\sqrt{\gamma}}{\sqrt{a}\sqrt{1+a}}}^{T_{max}}\sup_{\hat{X},\hat{Y}}\left| \hat{W}_{h,\gamma,N}(\boldsymbol{\hat{T}},a)\right|^pd\hat{T}\\
		&=\int_{[C\frac{\sqrt{\gamma}}{\sqrt{a}\sqrt{1+a}},T_{max}]-[4N-\frac{C}{N},4N+\frac{C}{N}]}\sup_{\hat{X},\hat{Y}}\left| \hat{W}_{h,\gamma,N}(\boldsymbol{\hat{T}},a)\right|^pd\hat{T}+\int_{4N-\frac{C}{N}}^{4N+\frac{C}{N}}\sup_{\hat{X},\hat{Y}}\left| \hat{W}_{h,\gamma,N}(\boldsymbol{\hat{T}},a)\right|^pd\hat{T}\\
		&\lesssim \int_{[C\frac{\sqrt{\gamma}}{\sqrt{a}\sqrt{1+a}},T_{max}]-[4N-\frac{C}{N},4N+\frac{C}{N}]}\frac{h^{\frac{p}{3}}}{h^{2p}(1+|N(T-4N)|^{1/4})^{p}}d\hat{T}\\
		&+\int_{4N-\frac{C}{N}}^{4N+\frac{C}{N}} \frac{h^{\frac{p}{3}}}{h^{2p}((\frac{N}{\lambda_{\gamma,h}^{1/3}})^{1/4}+|N(\hat{T}-4N)|^{1/6})^p}d\hat{T}\\
		&\lesssim \int_{0}^{4N-\frac{C}{N}}\frac{h^{\frac{p}{3}}}{h^{2p}(1+|N(T-4N)|^{1/4})^{p}}d\hat{T}+\int_{4N+\frac{C}{N}}^{T_{max}}\frac{h^{\frac{p}{3}}}{h^{2p}(1+|N(T-4N)|^{1/4})^{p}}d\hat{T}\\
		&+\int_{4N-\frac{C}{N}}^{4N+\frac{C}{N}} \frac{h^{\frac{p}{3}}}{h^{2p}((\frac{N}{\lambda_{\gamma,h}^{1/3}})^{1/4}+|N(\hat{T}-4N)|^{1/6})^p}d\hat{T}\\
		&\lesssim \int_{\frac{C}{N}}^{4N}\frac{h^{\frac{p}{3}}}{h^{2p}|N\theta|^{p/4}}d\theta+\int_{\frac{C}{N}}^{2T_{max}}\frac{h^{\frac{p}{3}}}{h^{2p}|N\theta|^{p/4}}d\theta+
		\int_{-\frac{C}{N}}^{\frac{C}{N}}  \frac{h^{\frac{p}{3}}}{h^{2p}((\frac{N}{\lambda_{\gamma,h}^{1/3}})^{1/4}+|N\theta|^{1/6})^p}d\theta\\
		&\lesssim C(p)\frac{h^{\frac{p}{3}}}{h^{2p}}\left( \frac{1}{\lambda_{a,h}^{\frac{p}{2}-2}N^{\frac{p}{4}}}+\frac{\log(N)^{\delta_{p,4}}}{N}+\frac{1}{N^{p/2}}+\frac{1}{N}+\frac{|\log(N^{\frac{3}{2}}/\lambda_{a,h}^{\frac{1}{2}})|^{\delta_{p,6}}N^{1/2}\lambda_{a,h}^{p/12}}{\lambda_{a,h}^{1/2}N^{p/4}}\right).
	\end{align}
	Now, if $p\ge 4$, we obtain the following bounds and adding all the terms, it follows, by using that $\sum_{N=1}^{m} \frac{\log(N)}{N}\lesssim \log(m)^2$, $\sum_{N=1}^{m} \frac{1}{N}\lesssim \log(m)$, $\sum_{N=1}^{m} \frac{1}{N^{p}}\le C(p)\frac{m}{m^{p}},p\ne 1$,
	\begin{align}
		\sum_{N \le KT_{max}}	\int_{C\frac{\sqrt{\gamma}}{\sqrt{a}\sqrt{1+a}}}^{T_{max}}\sup_{\hat{X},\hat{Y}}\left| W_{h,\gamma,N}(\boldsymbol{\hat{T}},a)\right|^pd\hat{T}\lesssim C(p) \frac{h^{p/3}}{h^{2p}}(1+\frac{|\log(KT_{max})|^{2\delta_{4,p}}}{(KT_{max})^{\frac{p}{4}-1}}).
	\end{align}
	This is the worst case possible, so in order to have Strichartz we need to choose $p$ as small as possible, however, in $p=4$, we obtain a $log$, so we have to be careful.
	We have then
	\begin{align}
		||G_{\gamma,h}(t,x,y,a)||_{L_t^4[C\sqrt{\gamma},\sqrt{a}\sqrt{1+a}T_{max}]L^\infty_{x,y}}\lesssim \sqrt{a}^{1/4}\frac{h^{1/3}}{h^2}|\log(KT_{max})|^{1/2}.
	\end{align}
	By Young's inequality, this gives
	\begin{align}	||T_{G_{\gamma,h}}T^\ast_{G_{\gamma,h}}(f)||_{L_t^8([C\sqrt{\gamma},\sqrt{a}\sqrt{1+a}T_{max}]),L^\infty(\Omega_2)}\lesssim \sqrt{\gamma}^{1/4}|\log(KT_{max})|^{1/2}\frac{h^{1/3}}{h^2}||f||_{L^{8/7}_t,L^1_{x,y}}.
	\end{align}
	But also, by using \eqref{smalltobigtimes}, we obtain
	\begin{align}
		||T_{G_{\gamma,h}}T^\ast_{G_{\gamma,h}}(f)||_{L_t^4([C\sqrt{\gamma},\sqrt{a}\sqrt{1+a}T_{max}]),L^\infty(\Omega_2)}\lesssim T_{max}^{1/2}\frac{h^{1/2}}{h^2}||f||_{L^{4/3}_t,L^1_{x,y}}.
	\end{align}
	We interpolate between these two expressions. For $\theta\in[0,1]$,
	\begin{align}
		&||T_{G_{\gamma,h}}T^\ast_{G_{\gamma,h}}(f)||_{L_t^{q_\theta}([C\sqrt{\gamma},\sqrt{a}\sqrt{1+a}T_{max}]),L^\infty(\Omega_2)}\\
		&\lesssim T_{max}^{\theta/2}\frac{h^{\theta/2}}{h^{2\theta}}\frac{h^{\frac{1-\theta}{3}}}{h^{2-2\theta}}\gamma^{\frac{1-\theta}{8}}|\log(KT_{max})|^{\frac{1-\theta}{2}}||f||_{L^{p_{\theta}}_t,L^1_{x,y}}.
	\end{align}
	We note that we always have $T_{max}h^{1/3}\lesssim \gamma^{1/2}$. (If $T_{max}\sim \lambda_{a,h}^{1/3}$ it follows that, if $T_{max}\sim \frac{1}{\sqrt{a}}$ is a consequence of being in the regime $\gamma \gtrsim h^{1/3}$). Furthermore, we have $h^{1/2}\gamma^{1/8}\lesssim T_{max}^{-3/4}$ (Again, by separating cases). This implies that $\gamma^{1/8}h^{1/3}\lesssim T_{max}^{-3/4}h^{1/4}$ and
	\begin{align}
		&||T_{G_{\gamma,h}}T^\ast_{G_{\gamma,h}}(f)||_{L_t^{q_\theta}([C\sqrt{\gamma},\sqrt{a}\sqrt{1+a}T_{max}]),L^\infty(\Omega_2)}\\
		&\lesssim T_{max}^{\frac{\theta}{2}-\frac{3}{4}(1-\theta)}\frac{h^{\frac{1-\theta}{4}+\frac{\theta}{2}}}{{h^2}}|\log(KT_{max})|^{\frac{1-\theta}{2}}||f||_{L^{p_{\theta}}_t,L^1_{x,y}}.
	\end{align}
	Since $\frac{1}{q_{\theta}}=\frac{\theta}{4}+\frac{1-\theta}{8}$, we have
	\begin{align}
		||T_{G_{\gamma,h}}T^\ast_{G_{\gamma,h}}(f)||_{L_t^{q_\theta}([C\sqrt{\gamma},\sqrt{a}\sqrt{1+a}T_{max}]),L^\infty(\Omega_2)}\lesssim T_{max}^{\frac{5\theta-3}{4}}\frac{h^{\frac{2}{q_{\theta}}}}{h^2}|\log(KT_{max})|^{\frac{1-\theta}{2}}||f||_{L^{p_{\theta}}_t,L^1_{x,y}}.
	\end{align}
	This means that, whenever $\theta<\frac{3}{5}$ (that is $q_{\theta}>5$), the exponent in $T_{max}$ is $<0$, so $T_{max}^{\frac{5\theta-3}{4}}|\log(KT_{max})|^{\frac{1-\theta}{2}}$ is bounded. Furthermore, it can be bounded by $T_{max}^{-\varepsilon(\theta)}$. And by duality we conclude. For $q>5$, we have
	\begin{align}
		||T_{G_{\gamma,h}}(u)||_{L_t^{q}([C\sqrt{\gamma},\sqrt{a}\sqrt{1+a}T_{max}]),L^\infty(\Omega_2)}\lesssim T_{max}^{-\varepsilon(q)}\frac{h^{\frac{1}{q}}}{h}||u||_{L^2_{x,y}}.
	\end{align}
	\begin{obs}
		If $q>8$ ($\theta<0$), instead of interpolating, we use directly
		\begin{align}
			||G_{\gamma,h}(t,x,y,a)||_{L_t^p[C\sqrt{\gamma},\sqrt{a}\sqrt{1+a}T_{max}]L^\infty_{x,y}}\lesssim \sqrt{\gamma}^{1/p}\frac{h^{1/3}}{h^2}\frac{|\log(KT_{max})|^{2\delta_{4,p}/p}}{KT_{max}^{\frac{1}{4}-\frac{1}{p}}},
		\end{align}
		getting the correct dispersion.
	\end{obs}
	\item $K^{-1}\lambda_{a,h}^{1/3}\le \hat{T}_i \le K^{-1}\lambda_{a,h}^{2/3}$.\\
	In this regime we start to have problems about the overlapping so we  study each integral in $\hat{T}_i$ individually.
	\begin{align}
		&\int_{\hat{T}_i}^{\hat{T}_{i+1}} \sup_{\hat{X},\hat{Y}}\left|\sum_{N\in\mathcal{N}_3(\hat{T}_{i})} W_{h,\gamma,N}(\boldsymbol{\hat{T}},a)\right|^pd\hat{T}\le C(p)\sum_{N\in \mathcal{N}_3(\hat{T}_{i})} \int_{\hat{T}_i}^{\hat{T}_{i+1}} \sup_{\hat{X},\hat{Y}}\left| W_{h,\gamma,N}(\boldsymbol{\hat{T}},a)\right|^pd\hat{T}.
	\end{align}
	For each $N$, we know that $N\le \lambda_{a,h}$ and by proposition \ref{bigtangenctial} we obtain
	\begin{align}
		|W_{h,\gamma,N}(\boldsymbol{\hat{T}},a)|\lesssim \frac{a^{1/2}\lambda_{\gamma,h}^2}{a^3\lambda_{\gamma,h}^{1/3}(\max\{(\frac{N}{\lambda_{\gamma,h}^{1/3}})^{1/2},1\}+|N(\hat{T}-4N)|^{1/4})}.
	\end{align}
	\begin{obs}
		\label{twosubdivisions}
		The main idea of this subdivision is that, for $K^{-1}\lambda_{a,h}^{1/3}\le \hat{T}_i \le K^{-1}\lambda_{a,h}^{2/3}$, it is more convenient to do
		\begin{align}
			&|W_{h,\gamma,N}(\boldsymbol{\hat{T}},a)|\lesssim \frac{a^{1/2}\lambda_{\gamma,h}^2}{a^3\lambda_{\gamma,h}^{1/3}(\max\{(\frac{N}{\lambda_{\gamma,h}^{1/3}})^{1/2},1\}+|N(\hat{T}-4N)|^{1/4})}\le \frac{a^{1/2}\lambda_{\gamma,h}^2}{a^3\lambda_{\gamma,h}^{1/3}(|N(\hat{T}-4N)|^{1/4})}.
		\end{align}
		While in $K^{-1}\lambda_{a,h}^{2/3}\le \hat{T}_i \le K^{-1}\lambda_{a,h}$, it is better to use the other denominator,
		\begin{align}
			&|W_{h,\gamma,N}(\boldsymbol{\hat{T}},a)|\lesssim \frac{a^{1/2}\lambda_{\gamma,h}^2}{a^3\lambda_{\gamma,h}^{1/3}(\max\{(\frac{N}{\lambda_{\gamma,h}^{1/3}})^{1/2},1\}+|N(\hat{T}-4N)|^{1/4})}\le \frac{a^{1/2}\lambda_{\gamma,h}^2}{a^3\lambda_{\gamma,h}^{1/3}(\frac{N}{\lambda_{\gamma,h}^{1/3}})^{1/2}}.
		\end{align}
	\end{obs}
	We get then the following: If $|\hat{T}_{i+1}-4N|\ge \frac{1}{3}$,then, for every $\hat{T}\in [\hat{T}_{i},\hat{T}_{i+1}]$, $|\hat{T}_{i+1}-4N|\ge \frac{1}{12}$ and we bound directly:
	\begin{align}
		&\int_{\hat{T}_{i}}^{\hat{T}_{i+1}}\sup_{\hat{X},\hat{Y}}\left| \hat{W}_{h,\gamma,N}(\boldsymbol{\hat{T}},a)\right|^pd\hat{T}\lesssim \int_{\hat{T}_{i}}^{\hat{T}_{i+1}} \frac{a^{p/2}\lambda_{\gamma,h}^{2p}}{a^{3p}\lambda_{\gamma,h}^{p/3}(|N(\hat{T}-4N)|^{1/4})^{p}}d\hat{T} \\
		&\lesssim C(p) \frac{a^{p/2}\lambda_{\gamma,h}^{2p}}{a^{3p}\lambda_{\gamma,h}^{p/3}N^{p/4}}.
	\end{align}
	If $|\hat{T}_{i+1}-4N|\le \frac{1}{3}$, then $|\hat{T}_{i}-4N|\le \frac{7}{12}$, and we can solve the integral:
	\begin{align}
		&\int_{\hat{T}_{i}}^{\hat{T}_{i+1}}\sup_{\hat{X},\hat{Y}}\left| \hat{W}_{h,\gamma,N}(\boldsymbol{\hat{T}},a)\right|^pd\hat{T}\lesssim \int_{\hat{T}_{i}}^{\hat{T}_{i+1}} \frac{a^{p/2}\lambda_{\gamma,h}^{2p}}{a^{3p}\lambda_{\gamma,h}^{p/3}|N|^{p/4}||\hat{T}-4N|^{p/4}}d\hat{T}\\
		&\lesssim C(p) \frac{a^{p/2}\lambda_{\gamma,h}^{2p}}{a^{3p}\lambda_{\gamma,h}^{p/3}N^{p/4}} \left(|\hat{T}_{i+1}-4N|^{\frac{p}{4}+1}+|\hat{T}_{i}-4N|^{\frac{p}{4}+1}\right)\lesssim C(p)\frac{a^{p/2}\lambda_{\gamma,h}^{2p}}{a^{3p}\lambda_{\gamma,h}^{p/3}N^{p/4}}.
	\end{align}
	In both cases, it follows that
	\begin{align}
		\int_{\hat{T}_{i}}^{\hat{T}_{i+1}}\sup_{\hat{X},\hat{Y}}\left| \hat{W}_{h,\gamma,N}(\boldsymbol{\hat{T}},a)\right|^pd\hat{T}\lesssim C(p)\frac{a^{p/2}\lambda_{\gamma,h}^{2p}}{a^{3p}\lambda_{\gamma,h}^{p/3}N^{p/4}}\lesssim C(p)\frac{a^{p/2}\lambda_{\gamma,h}^{2p}}{a^{3p}\lambda_{\gamma,h}^{p/3}T_{i}^{p/4}}.
	\end{align}
	See corollary \ref{TsimN}.\\
	Now, we can add all the terms, in our regime then, for $T_{max}\ge K^{-1}\lambda_{a,h}^{1/3}$
	\begin{align}
		&\int_{\hat{T}_i}^{\hat{T}_{i+1}} \sup_{\hat{X},\hat{Y}}\left|\sum_{N\in\mathcal{N}_3(\hat{T}_{i})} W_{h,\gamma,N}(\boldsymbol{\hat{T}},a)\right|^pd\hat{T}\le C(p)\sum_{N\in \mathcal{N}_3(\hat{T}_{i})} \int_{\hat{T}_i}^{\hat{T}_{i+1}} \sup_{\hat{X},\hat{Y}}\left| W_{h,\gamma,N}(\boldsymbol{\hat{T}},a)\right|^pd\hat{T}\\
		&\lesssim C(p)\frac{a^{p/2}\lambda_{\gamma,h}^{2p}}{a^{3p}\lambda_{\gamma,h}^{p/3}T_{i}^{p/4}}.
	\end{align}
	Therefore, for $2<p<4$
	\begin{align}
		&\sum_{i\in\mathbb{N}|\le K^{-1}\lambda_{a,h}^{1/3}\hat{T}_i\le T_{max}} \int_{\hat{T}_i}^{\hat{T}_{i+1}} \sup_{\hat{X},\hat{Y}}\left|\sum_{N\in\mathcal{N}_3(\hat{T}_{i})} W_{h,\gamma,N}(\boldsymbol{\hat{T}},a)\right|^pd\hat{T}\\
		&\lesssim C(p)\frac{a^{p/2}\lambda_{\gamma,h}^{2p}}{a^{3p}\lambda_{\gamma,h}^{p/3}} \sum_{i\in\mathbb{N}| K^{-1}\lambda_{a,h}^{1/3}\le \hat{T}_i\le T_{max}} \hat{T}_{i}^{-p/4}\lesssim C(p)\frac{a^{p/2}\lambda_{\gamma,h}^{2p}}{a^{3p}\lambda_{\gamma,h}^{p/3}} T_{max}^{1-\frac{p}{4}}.
	\end{align}
	Therefore, 
	\begin{align}
		||G_{\gamma,h}(t,x,y,a)||_{L_t^p[K^{-1}\lambda_{a,h}^{1/3},T_{max}]L_{x,y}^\infty}\lesssim \gamma^{\frac{1}{2p}} \frac{\gamma^{1/2}}{\gamma^{3}}\frac{\lambda_{\gamma,h}^2}{\lambda_{\gamma,h}^{1/3}}T_{max}^{\frac{1}{p}-\frac{1}{4}}\lesssim \gamma^{\frac{1}{2p}} \frac{h^{1/3}}{h^{2}}T_{max}^{\frac{1}{p}-\frac{1}{4}}.
	\end{align}
	Now we separate cases.\\
	If $T_{max}\sim \gamma^{-1/2}$, we have $\gamma\gtrsim h^{4/9}$. We group the terms like this:
	\begin{align}
		||G_{\gamma,h}(t,x,y,a)||_{L_t^p[K^{-1}\lambda_{a,h}^{1/3},T_{max}]L_{x,y}^\infty}\lesssim \frac{h^{1/p}}{h^2}\gamma^{1/8}h^{\frac{1}{3}-\frac{1}{p}}=\frac{h^{1/p}}{h^2}\gamma^{\frac{7}{8}-\frac{9}{4p}} \left(\frac{h}{\gamma^{9/4}}\right)^{{\frac{7}{8}-\frac{9}{4p}}}.
	\end{align}
	In order for this to be a good estimate, we need then $\gamma^{1/8}h^{\frac{1}{3}-\frac{1}{p}}\lesssim 1$. We compute then the following:
	\begin{align}
		\gamma^{1/8}h^{\frac{1}{3}-\frac{1}{p}}\lesssim \gamma^{1/8}\gamma^{\frac{9}{4}(\frac{1}{3}-\frac{1}{p})}=\gamma^{\frac{7}{8}-\frac{9}{4p}}.
	\end{align}
	If ${\frac{7}{8}-\frac{9}{4p}}\ge 0$, then we can guarantee that $\gamma^{1/8}h^{\frac{1}{3}-\frac{1}{p}}\lesssim 1$, and that coincides with $p\ge \frac{18}{7}$.\\\\
	If $T_{max}\sim \lambda_{a,h}^{2/3}\sim \lambda_{\gamma,h}^{2/3}$, we have $\gamma\lesssim h^{4/9}$. We group terms like this
	\begin{align}
		||G_{\gamma,h}(t,x,y,a)||_{L_t^p[K^{-1}\lambda_{a,h}^{1/3},T_{max}]L_{x,y}^\infty}\lesssim \frac{h^{1/p}}{h^2}\gamma^{\frac{3}{2p}-\frac{1}{4}}h^{\frac{1}{2}-\frac{5}{3p}}\lesssim \frac{h^{1/p}}{h^2}\left(\frac{\gamma}{h^{4/9}}\right)^{(\frac{3}{2p}-\frac{1}{4})}h^{\frac{4}{9}(\frac{3}{2p}-\frac{1}{4})}h^{\frac{1}{2}-\frac{5}{3p}}.
	\end{align}
	In order to $h^{\frac{4}{9}(\frac{3}{2p}-\frac{1}{4})}h^{\frac{1}{2}-\frac{5}{3p}}$ be small, we need $p\ge \frac{18}{7}$.
	Therefore, we have the following, for both cases of $T_{max}$.\\
	If $p\ge \frac{18}{7}$, we have 
	\begin{align}
		||G_{\gamma,h}(t,x,y,a)||_{L_t^p[K^{-1}\lambda_{a,h}^{1/3},T_{max}]L_{x,y}^\infty}\lesssim \left(\frac{\gamma}{h^{4/9}}\right)^{(\frac{3}{2p}-\frac{1}{4})}\frac{h^{1/p}}{h^2}.
	\end{align}
	This in turns, gives by the usual $TT\ast$ argument, the following.\\
	For $q\ge 5+\frac{1}{7}=\frac{36}{7}=\frac{2\cdot 18}{7}$, we have 
	\begin{align}
		||T_{G_{\gamma,h}}(u)||_{L_t^{q}([K^{-1}\lambda_{a,h}^{1/3},\sqrt{a}\sqrt{1+a}T_{max}]),L^\infty(\Omega_2)}\lesssim \min\{\left(\frac{\gamma}{h^{4/9}}\right),\left(\frac{h^{4/9}}{\gamma}\right)\}^{\frac{1}{2}(\frac{3}{q}-\frac{1}{4})} \frac{h^{\frac{1}{q}}}{h}||u||_{L^2_{x,y}}.
	\end{align}
	\item $K^{-1}\lambda_{a,h}^{2/3}\le \hat{T}_i \le K^{-1}\lambda_{a,h}$.
	As said in the remark \ref{twosubdivisions}, in this regime we consider the other denominator:
	\begin{align}
		&\int_{\hat{T}_{i}}^{\hat{T}_{i+1}}\sup_{\hat{X},\hat{Y}}\left| \hat{W}_{h,\gamma,N}(\boldsymbol{\hat{T}},a)\right|^pd\hat{T}\lesssim \int_{\hat{T}_{i}}^{\hat{T}_{i+1}} \frac{a^{p/2}\lambda_{\gamma,h}^{2p}}{a^{3p}\lambda_{\gamma,h}^{p/3}(\frac{N}{\lambda_{\gamma,h}^{1/3}})^{p/2}}d\hat{T} \lesssim \frac{a^{p/2}\lambda_{\gamma,h}^{2p}}{a^{3p}\lambda_{\gamma,h}^{p/3}(\frac{N}{\lambda_{\gamma,h}^{1/3}})^{p/2}} \\
		&\lesssim C(p)\frac{a^{p/2}\lambda_{\gamma,h}^{2p}\lambda_{\gamma,h}^{p/6}}{a^{3p}\lambda_{\gamma,h}^{p/3}{\hat{T}_{i}}^{p/2}}.\\
	\end{align}
	Thus we finalize obtaining the same bound as before for $2<p<4$
	\begin{align}
		&\sum_{i\in\mathbb{N}|\le K^{-1}\lambda_{a,h}^{2/3}\hat{T}_i\le T_{max}} \int_{\hat{T}_i}^{\hat{T}_{i+1}} \sup_{\hat{X},\hat{Y}}\left|\sum_{N\in\mathcal{N}_3(\hat{T}_{i})} W_{h,\gamma,N}(\boldsymbol{\hat{T}},a)\right|^pd\hat{T}\\
		&\lesssim C(p)\frac{a^{p/2}\lambda_{\gamma,h}^{2p}\lambda_{a,h}^{p/6}}{a^{3p}\lambda_{\gamma,h}^{p/3}} \sum_{i\in\mathbb{N}| K^{-1}\lambda_{a,h}^{2/3}\le \hat{T}_i\le T_{max}} \hat{T}_{i}^{-p/2}\lesssim C(p)\frac{a^{p/2}\lambda_{\gamma,h}^{2p}}{a^{3p}\lambda_{\gamma,h}^{p/3}} \lambda_{a,h}^{\frac{2}{3}-\frac{p}{6}}.
	\end{align}
	\begin{obs}
		Since $p>2$, $p/2>1$, so in the sum, the bound is taken using the first term, while before, since $p<4$, $p/4<1$, so in the sum, the bound is taken using the last term (in both cases $\lambda_{a,h}^{2/3}$).
	\end{obs}
	We obtain then, the same bound as the previous case for the case $T_{max}\sim \lambda_{a,h}^{2/3}$. Thus we obtain the exact same estimates:\\
	For $q\ge 5+\frac{1}{7}=\frac{36}{7}=\frac{2\cdot 18}{7}$, we have 
	\begin{align}
		||T_{G_{\gamma,h}}(u)||_{L_t^{q}([K^{-1}\lambda_{a,h}^{2/3},\sqrt{a}\sqrt{1+a}T_{max}]),L^\infty(\Omega_2)}\lesssim \min\{\left(\frac{\gamma}{h^{4/9}}\right),\left(\frac{h^{4/9}}{\gamma}\right)\}^{\frac{1}{2}(\frac{3}{2p}-\frac{1}{4})} \frac{h^{\frac{1}{q}}}{h}||u||_{L^2_{x,y}}.
	\end{align}
	\item $K^{-1}\lambda_{a,h}\le \hat{T}_i \le K^{-1}\lambda_{a,h}^2$.\\
	In this case, we can repeat the same argument as before, but with the bound
	\begin{align}
		|W_{h,\gamma,N}(\boldsymbol{\hat{T}},a)|\lesssim \frac{h^{1/3}\lambda_{a,h}^{2/3}}{h^2N},
	\end{align}
	obtaining
	\begin{align}
		&\sum_{i\in\mathbb{N}| K^{-1}\lambda_{a,h}\le \hat{T}_i\le K^{-1}\lambda_{a,h}^2} \int_{\hat{T}_i}^{\hat{T}_{i+1}} \sup_{\hat{X},\hat{Y}}\left|\sum_{N\in\mathcal{N}_3(\hat{T}_{i})} W_{h,\gamma,N}(\boldsymbol{\hat{T}},a)\right|^pd\hat{T}\\
		&\lesssim C(p) h^{\frac{p}{3}-2p}\lambda_{a,h}^{1-p}\lambda_{a,h}^{2p/3}.
	\end{align}
	That gives us that, when $p< 4$, we obtain
	\begin{align}
		||G_{\gamma,h}(t,x,y,a)||_{L_t^p[K^{-1}\lambda_{a,h},T_{max}]L_{x,y}^\infty}\lesssim \frac{h^{1/p}}{h^2}\gamma^{\frac{2}{p}-\frac{1}{2}}.
	\end{align}
	That is, 
	\begin{align}
		||T_{G_{\gamma,h}}(u)||_{L_t^{q}([K^{-1}\lambda_{a,h},\sqrt{a}\sqrt{1+a}T_{max}]),L^\infty(\Omega_2)}\lesssim \gamma^{\frac{1}{p}-\frac{1}{4}} \frac{h^{\frac{1}{q}}}{h}||u||_{L^2_{x,y}}.
	\end{align}
	The bound goes all the way until $C\frac{\sqrt{\gamma}}{\sqrt{a}\sqrt{1+a}}$, however, when $\hat{T}_i \ge K^{-1}\lambda_{a,h}^2$, we lose the nice bound of $\mathcal{N}_3$ by a constant, so the overlapping starts accumulating.
	\item $K^{-1}\lambda_{a,h}^2\le \hat{T}_i\lesssim \frac{t_{max}}{\sqrt{a}\sqrt{1+a}}$.\\
	In this case, we can take the bound of $\mathcal{N}_3$ directly in the integral and taking the smallest term:
	Obtaining
	\begin{align}	&\sum_{i\in\mathbb{N}|K^{-1}\lambda_{a,h}^2\le \hat{T}_i\le \frac{t_{max}}{\sqrt{a}\sqrt{1+a}}} \int_{\hat{T}_i}^{\hat{T}_{i+1}} \sup_{\hat{X},\hat{Y}}\left|\sum_{N\in\mathcal{N}_3(\hat{T}_{i})} W_{h,\gamma,N}(\boldsymbol{\hat{T}},a)\right|^pd\hat{T}\\
		&\lesssim \sum_{i\in\mathbb{N}|K^{-1}\lambda_{a,h}^2\le \hat{T}_i\le \frac{t_{max}}{\sqrt{a}\sqrt{1+a}}} \int_{\hat{T}_i}^{\hat{T}_{i+1}} \sup_{\hat{X},\hat{Y}}\left|\sum_{N\in\mathcal{N}_3(\hat{T}_{i})} |W_{h,\gamma,N}(\boldsymbol{\hat{T}},a)|\right|^pd\hat{T}\\
		&\lesssim \sum_{i\in\mathbb{N}|K^{-1}\lambda_{a,h}^2\le \hat{T}_i\le \frac{t_{max}}{\sqrt{a}\sqrt{1+a}}} \int_{\hat{T}_i}^{\hat{T}_{i+1}} \sup_{\hat{X},\hat{Y}}\left|(C+\frac{\hat{T}_{i}}{\lambda_{a,h}^2})\frac{h^{1/3}\lambda_{a,h}^{2/3}}{h^2\hat{T}_i} \right|^pd\hat{T}\\
		&\lesssim \sum_{i\in\mathbb{N}|K^{-1}\lambda_{a,h}^2\le \hat{T}_i\le \frac{t_{max}}{\sqrt{a}\sqrt{1+a}}}\left|(C+\frac{\hat{T}_{i}}{\lambda_{a,h}^2})\frac{h^{1/3}\lambda_{a,h}^{2/3}}{h^2\hat{T}_i} \right|^p \\
		&\lesssim \sum_{i\in\mathbb{N}|K^{-1}\lambda_{a,h}^2\le \hat{T}_i\le \frac{t_{max}}{\sqrt{a}\sqrt{1+a}}} \frac{h^{p/3}\lambda_{a,h}^{2p/3}}{h^{2p}\lambda_{a,h}^{2p}}\lesssim \frac{h^{p/3}\lambda_{a,h}^{2p/3}}{h^{2p}\lambda_{a,h}^{2p}\sqrt{a}}.\\
	\end{align}
	We get then
	\begin{align}
		||G_{\gamma,h}(t,x,y,a)||_{L_t^p[K^{-1}\lambda_{a,h}^2,\frac{t_{max}}{\sqrt{a}\sqrt{1+a}}]L_{x,y}^\infty}\lesssim \frac{h^{1/3}}{h^2}\lambda_{a,h}^{-4/3}.
	\end{align}
	This regime, gives us the good dispersion whenever we have $\lambda_{a,h}>>1$, that is in our regime $\lambda_{a,h}\ge h^{\varepsilon}$. The remaining regime is handled by gallery-mode analysis.\\
	We obtain, for any $\gamma\gtrsim h^{\frac{2}{3}-\varepsilon}$, $q\ge \frac{36}{7}$, there exists a bounded quantity $C(h,\gamma)$ such that
	\begin{align}
		||T_{G_{\gamma,h}}(u)||_{L_t^{q}([0,t_{max}]),L^\infty(\Omega_2)}\lesssim \frac{h^{\frac{1}{q}}}{h}||u||_{L^2_{x,y}}C(h,\gamma),
	\end{align}
	and such that $\sum_{2^{-7}\ge\gamma\gtrsim h^{\frac{2}{3}-\varepsilon}} C(h,\gamma)$ is bounded.
	Which in turn, means that
	\begin{align}
		\sum_{1\gtrsim\gamma\gtrsim h^{2/3-\varepsilon}}||T_{G_{\gamma,h}}(u)||_{L_t^{q}([0,t_{max}]),L^\infty(\Omega_2)}\lesssim \frac{h^{\frac{1}{q}}}{h}||u||_{L^2_{x,y}}.
	\end{align}
	This, together with proposition \ref{gallerystr}, give us
	\begin{align}
		\label{Linftyfinal}
		||T_{G_{h}}(u)||_{L_t^{q}([0,t_{max}]),L^\infty(\Omega_2)}\lesssim \frac{h^{\frac{1}{q}}}{h}||u||_{L^2_{x,y}}.
	\end{align}
This concludes the result for ${L_t^{q}([0,t_{max}]),L^\infty(\Omega_2)}$.
\end{enumerate}
Now we can conclude the result for all the admissible pairs stated in theorem \ref{mainteo}, 	
	$\frac{1}{q}\le \left(\frac{1}{2}-\frac{1}{9}\right)\left(\frac{1}{2}-\frac{1}{r}\right)$ by interpolating the expression we just obtained with the energy estimates, proposition \ref{almostconsenergy}, concluding the proof of $\ref{mainteo}$.
\appendix
\section{Airy Functions}
We introduce here the Airy functions $Ai$ and $Bi$.\\

\begin{defi}

	We define the functions $Ai,Bi:\mathbb{C}\longrightarrow\mathbb{C}$, as the solutions to the following differential equation
	\begin{equation}
		\label{airyeq} y''(z)-zy(z)=0.
	\end{equation}
	With the real initial conditions $Ai(0)=\frac{Bi(0)}{\sqrt{3}}=\frac{1}{3^{\frac{2}{3}}\Gamma\left(\frac{2}{3}\right)}$, $-Ai'(0)=\frac{Bi'(0)}{\sqrt{3}}=\frac{1}{3^{\frac{1}{3}}\Gamma\left(\frac{1}{3}\right)}$.
\end{defi}
In particular, since they solve a linear differential equation with real Cauchy data, we obtain the immediate corollary.
\begin{cor}
	$Ai,Bi$ are analytic functions and we have $Ai(\mathbb{R})\subset \mathbb{R}$, $Bi(\mathbb{R})\subset \mathbb{R}$, and they are both analytic in $\mathbb{C}$.
\end{cor}
We state here the following properties of the Airy functions which will use
\begin{prop}  
	For every $x\in\mathbb{R}$, we obtain
	\begin{equation}
		Bi(x):=\frac{1}{\pi}\left(\int_{0}^{\infty} e^{xp-\frac{p^3}{3}} dp+\int_{0}^{\infty} \sin\left(xp+\frac{p^3}{3}\right)dp\right),
	\end{equation}
	\begin{equation}
		Ai(x):=\frac{1}{2\pi}\int e^{i(\frac{t^3}{3}+tx)}dt=\frac{1}{\pi}\int_0^{\infty}\cos\left(\frac{t^3}{3}+xt\right)dt.
	\end{equation}
	Where the integrals involved converge in an improper sense.\\
	If we denote $j=e^{\frac{2}{3}\pi i}$, we get, for every $z\in\mathbb{C}$,
	\begin{equation}
		\label{bij}
		Bi(z)=ijAi(jz)-ij^2Ai(j^2z),
	\end{equation}
	\begin{equation}
		\label{propai}
		Ai(z)+jAi(jz)+j^2Ai(j^2z)=0,
	\end{equation}
	\begin{equation}
		W\{Ai,Bi\}(z):=Ai(z)Bi'(z)-Ai'(z)Bi(z)=\frac{1}{\pi}.
	\end{equation}
	\label{integral}
	For all $\alpha,\beta\in\mathbb{R}$, $\alpha\ne \beta$,
	\begin{equation}
		\label{intairy1}
		\int_{0}^{\infty} Ai^2(x+\beta)dx=-\beta Ai(\beta)^2+Ai'(\beta)^2,
	\end{equation}
	\begin{equation}
		\int_{0}^{\infty} Ai(x+\beta)Ai(x+\alpha)dx=\frac{1}{\alpha-\beta}\left(Ai(\alpha)Ai'(\beta)-Ai'(\alpha)Ai(\beta) \right).
	\end{equation}
\end{prop}
\begin{cor}
	$Ai\in L^2(\mathbb{R}_+)$ and $||Ai||^2_{L^2(\mathbb{R}_+)}=Ai'(0)^2$, $Bi\notin L^2(\mathbb{R}_+)$, $Bi\notin S'(\mathbb{R}_+)$.
\end{cor}
We state now some classical results about the zeroes of the Airy function $Ai$.
\begin{teo}
	\label{zerai}
	All the zeroes of $Ai$ are on the negative part of the real axis, and the set of zeroes are countable, discrete and infinite. If we denote them $-w_k$ in increasing modulus, we obtain the sequence $\{w_k\}_{k\in\mathbb{N}}$ and we have that  $\lim_{k\rightarrow\infty}\frac{w_k}{k^{\frac{2}{3}}}$ is finite and non-zero and $w_{k+1}-w_k\sim k^{-1/3}$.
\end{teo}
\begin{obs}
	For simplicity, since $Ai$ tends to zero at infinity, we will denote $-w_0=\infty$.
\end{obs}
We also have the following bound:
\begin{lema}
	\label{ineq}
	In $(-{w}_1,\infty)$, we have $xAi^2(x)< Ai'^2(x)$.
\end{lema} 
Taking the Weierstrass infinite product theorem, and its logarithm, we obtain, by differentiation
\begin{prop}
	\label{weiestrass}
	\begin{equation}
		Lai(z):=(\ln Ai)'(z)=\frac{Ai'(z)}{Ai(z)}=-\kappa+\sum_{n=1}^{\infty}\left(\frac{1}{z+w_n}-\frac{1}{w_n}\right).
	\end{equation}
\end{prop}
This function has some notable properties which will be very useful:
\begin{lema}
	\label{lemalai}
	The function $Lai$ is a smooth function defined in $\mathbb{C}-\{-w_i:i\in\mathbb{N}\}$. The function also satisfies the following differential equation:
	\begin{equation}
		Lai'(z)=z-Lai^2(z).
	\end{equation}
	When we reduce its domain to the real numbers $\bigcup_{i=0}^{\infty}(-w_{i+1},-w_i)$, the function is strictly decreasing for all real $x$ in its domain, and more specifically $Lai'(x)<0$.\\
	Furthermore, for all $i\in\mathbb{N}_0$, $Lai\colon (-w_{i+1},-w_i)\to \mathbb{R}$, is a decreasing diffeomorphism.
\end{lema}
\begin{cor}
	\label{corlai}
	For every $i\in\mathbb{N}_0$, $\lim_{x\rightarrow -w_i^+} Lai(x)=\infty$, $\lim_{x\rightarrow -w_i^-} Lai(x)=-\infty$.
\end{cor}
\begin{proof}
	As a quotient of smooth functions, $Lai$ is smooth when the denominator is not zero and by \ref{zerai}, that means $Lai$ is smooth for every $z\notin\{-w_i:i\in\mathbb{N}\}$.\\
	By using the differential equation that $Ai$ solves, we obtain
	\begin{equation}
		Lai'(x)=\frac{Ai''(x)Ai(x)-Ai'(x)Ai'(x)}{Ai(x)^2}=x-Lai^2(x).
	\end{equation}
	Since $Lai(0)>0$ (We can compute its exact value), for every $x\in\mathbb{R}$, $Lai'(x)=x-Lai^2(x)<0$, if $x>0$ because of \ref{ineq}, in $x<0$ because both terms in the sum are strictly negative.\\
	
\end{proof}
\section{The Robin function}
\label{approbin}
We introduce now the main  function for the study of the Robin boundary problem:
\begin{defi}
	\label{defirob}
	For every $\iota\in\mathbb{R}$, we define the Robin function $Rob_\iota:\mathbb{C}\longrightarrow\mathbb{C}$ as
	\begin{equation}
		Rob_\iota(z)=\cos(\iota)Ai(z)-\sin(\iota)Ai'(z).
	\end{equation}
	\begin{obs}
		Of course, this is an analytic function and it is also smooth if we consider $\iota$ a variable too. Also, the only relevant values of $\iota$ are in a $2\pi$ cycle (and in fact, since we are interested in making this function zero, we are only interested in [0,$\pi$) ).\\
		Also, $Ai=Rob_0$ and $-Ai'=Rob_{\pi/2}$.
	\end{obs}
\end{defi}
\subsection{Zeroes of $Rob$}
With this, we can state a theorem about the real zeroes of $Rob$.
\begin{prop}\label{zerosrob}
	\label{zeros}
	For $\iota\in[0,\pi)$, the real zeroes of $Rob_\iota$ are discrete, there are infinitely many of them and at most one is positive.\\
	If we denote by $-w_{\iota,n}$ the n-th smallest zero of $Rob_\iota$, we have:\\\\
	\begin{equation}
		w_{n-1}< w_{\iota,n}\le w_{n}.
	\end{equation}
	Also, for every natural, the function $w_{\iota,n}$ is a smooth strictly decreasing function at each point with respect to $\iota$ in $(0,\pi)$ and we have $\partial_\iota w_{\iota,n}=-\frac{1+\cot(\iota)^2}{w_{\iota,n}+\cot(\iota)^2}$. Specifically, $w_{\iota,n}=-Lai^{-1}(\cot(\iota))$. Where $Lai\colon (-w_{n},-w_{n-1})\to \mathbb{R}$.
\end{prop}
In order to prove that, we use the lemma \ref{lemalai} and corollary \ref{corlai}.

\begin{proof}
	For $\iota=0$, we have the function $Ai$, and the proposition \ref{zerai}. For $\iota\ne 0$, a real number $w$ is a zero of $Rob_\iota$ iff, $Lai(w)=\cot(\iota)$.\\
	By theorem \ref{corlai}, for every $i$, in the interval $(-w_{i},-w_{i-1})$, $Lai$ takes all the real numbers exactly once. In particular, it assumes the value $\cot(\iota)$ exactly once. Thus $-w_{\iota,i}\in\mathbb{R}$ is the value (which exists and is unique) such that,$-w_{\iota,i}\in (-w_{i},-w_{i-1})$ and  $Lai(-w_{\iota,i})=\cot(\iota)$.\\
	In other words, $-w_{\iota,i}=Lai^{-1}(\cot(\iota))$, with $Lai^{-1}$ chosen with image in $(-w_{i},-w_{i-1})$. This last expression, proves the smoothness and by differentiating it we obtain 	$\partial_\iota w_{\iota,n}=-\frac{1+\cot(\iota)^2}{w_{\iota,n}+\cot(\iota)^2}$.
\end{proof}
\begin{obs}
	We make no claim about the non-real zeroes. However, in \cite{rus} it is stated that all the zeroes are always real.
\end{obs}
\begin{cor}
	All real zeroes of $Rob_\iota$ are negative except maybe the first.
\end{cor}
\subsection{The functions $A_{\pm,\iota}$}
\begin{defi}
	\label{defia}
	We define the following functions $A_{\pm,\iota}$:
	\begin{align}
		&A_{\pm,\iota}(z)=\frac{1}{2}\left[(\cos(\iota)Ai(-z)-\sin(\iota)Ai'(-z))\mp(\cos(\iota)Bi(-z)-\sin(\iota)Bi'(-z))i\right]\\
		&=\cos(\iota)A_\pm(z)+\sin(\iota)A'_\pm(z).
	\end{align}
	With $A_\pm=A_{\pm,0}$.
\end{defi}
As a simple consequence, if follows that
\begin{prop}
	\label{arobbis}
	\begin{equation}
		A_{\mp,\iota}(z)=-e^{\pm \frac{4\pi i}{3}}\left[\cos(\iota)e^{\mp\frac{2\pi i}{3}}Ai\left(-e^{\pm \frac{2\pi i}{3}}z\right)-\sin(\iota)Ai'\left(-e^{\pm \frac{2\pi i}{3}}z\right)\right].
	\end{equation}
\end{prop}
\begin{obs}
	Here we can see that $A_{-,\iota}(z)$ is not just the Robin function evaluated in complex numbers, here the parameter is also complex, which introduces and additional complication. But at least, we can assert that in a neighborhood of $\mathbb{R}$, $A_{-,\iota}\ne 0$.
\end{obs}
\begin{prop}
	For every $\iota\in(0,\pi)$, there exists an $\varepsilon>0$ such that $A_{-,\iota}(z)\ne 0$ for all $z\in\{w\in \mathbb{C}| |Im(w)|<\varepsilon\}$.
\end{prop}
\begin{proof}
	
	Fix $\iota\in (0,\pi)$. The condition $A_{-,\iota}(z)=0$ is equivalent, by \ref{defia}, to \\$0=\cot(\iota)e^{-\frac{2\pi i}{3}}Ai\left(-e^{\frac{2\pi i}{3}}z\right)-Ai'\left(-e^{\frac{2\pi i}{3}}z \right)$ that is, since $Ai$ and $Ai'$ can never be both zero, 
	\begin{align}
		Lai\left(-e^{\frac{2\pi i}{3}}z \right)=\cot(\iota)e^{\frac{-2\pi i}{3}}=-\frac{1}{2}\cot(\iota)(1+i\sqrt{3}).
	\end{align}
	Also we know the value of $Lai(0)\in \mathbb{R}$ and $Lai(0)\ne 0$, so, for every $\iota$, $Lai(0)\ne -\frac{1}{2}\cot(\iota)(1+i\sqrt{3})$. By continuity, there exists an $\varepsilon>0$ such that, if $|z|<\varepsilon$, then $Lai\left(-e^{\frac{2\pi i}{3}}z \right)\ne -\frac{1}{2}\cot(\iota)(1+i\sqrt{3})$.\\
	Assume, for contradiction, that there exists a fixed $z$ such that $Lai\left(-e^{\frac{2\pi i}{3}}z \right)=\cot(\iota)e^{\frac{-2\pi i}{3}}$. Define $w:=-e^{\frac{2\pi i}{3}}z$, $w_r:=\Re(w)$, $w_i:=\Im(w)$. By \ref{weiestrass}, we have 
	\begin{figure}
		\caption{For $\cot(\iota)>0$, there are only zeroes in the green zone, so in the orange strip we can guarantee that there are no zeroes}
		\centering
		\begin{tikzpicture}[scale=1.5, >=Stealth]
			% Definir ángulos
			\def\angleA{120} % ángulo de la línea azul (2π/3 = 120°)
			\def\angleB{140} % ángulo de la línea roja (2π/3+0.1 rad ≈ 125.73°)
			
			% Ejes
			\draw[->] (-1.5,0) -- (1.5,0) node[right] {$\mathrm{Re}$};
			\draw[->] (0,-1.5) -- (0,1.5) node[above] {$\mathrm{Im}$};
			\node at (0,0) [below left] {};
			
			% Círculo de radio 0.5
			\draw[thick, purple] (0,0) circle (0.5);
			
			% Línea azul: arg(z)=2π/3
			\draw[thick, blue] (0,0) -- (\angleA:1.5) node[midway, above left] {};
			
			% Línea roja: arg(z)=2π/3+0.1
			\draw[thick, red] (0,0) -- (\angleB:1.5) node[midway, above left] {};
			
			% Punto P: intersección de la línea roja y el círculo de radio 0.5
			\coordinate (P) at (\angleB:0.5);
			\filldraw[red] (P) circle (1.5pt) node[above right] {};
			
			% Punto A en la línea azul para proyección
			\coordinate (A) at (\angleA:2);
			
			% Proyección de P sobre la línea azul
			\coordinate (Q) at ($(0,0)!(P)!(A)$);
			
			% Punto reflejado P'
			\coordinate (P') at ($2*(Q) - (P)$);
			
			% Línea L1: paralela a la azul que pasa por P
			\draw[thick, green!50!black] ($(P) + (\angleA:2)$) -- ($(P) + (\angleA:-2)$) node[midway, above right] {};
			
			% Línea L2: paralela a la azul que pasa por P' (lado opuesto)
			\draw[thick, green!50!black] ($(P') + (\angleA:2)$) -- ($(P') + (\angleA:-2)$) node[midway, below left] {};
			
			% Sombrear la banda entre L1 y L2 dentro del círculo de radio 1.5
			% Definir puntos lejanos en L1 y L2
			\coordinate (L1a) at ($(P) + (\angleA:10)$);
			\coordinate (L1b) at ($(P) + (\angleA:-10)$);
			\coordinate (L2a) at ($(P') + (\angleA:10)$);
			\coordinate (L2b) at ($(P') + (\angleA:-10)$);
			
			\begin{scope}
				\clip (0,0) circle (1.5);
				\fill[orange, opacity=0.3] (L1a) -- (L1b) -- (L2b) -- (L2a) -- cycle;
			\end{scope}
			
			% Círculo de radio 1.5 de referencia
			\draw[dashed] (0,0) circle (1.5);
			
			% Leyenda
			\node[blue, right] at (1.5,-1.2) {Blue line: $\arg(w)=2\pi/3$};
			\node[red, right] at (1.5,-1.5) {Red line: $\arg(w)=2\pi/3+\epsilon$};
			\node[purple, right] at (1.5,-1.8){Circle of Radius $\varepsilon$};
			\node[orange, right] at (1.5,-2.1) {Region  $\{\Im(z)\in(-\tilde{\varepsilon},\tilde{\varepsilon})\}$};
			\node[green!30!white, right] at (1.5,-2.4) {Region where the zeroes could be};
			\begin{scope}
				% Recortar para excluir el círculo interior
				\clip (0,0) circle (1.5);
				\fill[green!30!white, opacity=0.5, even odd rule] 
				(0,0) -- (\angleB:1.5) arc (\angleB:180:1.5) -- cycle
				(0,0) -- (\angleB:0.5) arc (\angleB:180:0.5) -- cycle;
			\end{scope}
		\end{tikzpicture}
	\end{figure}
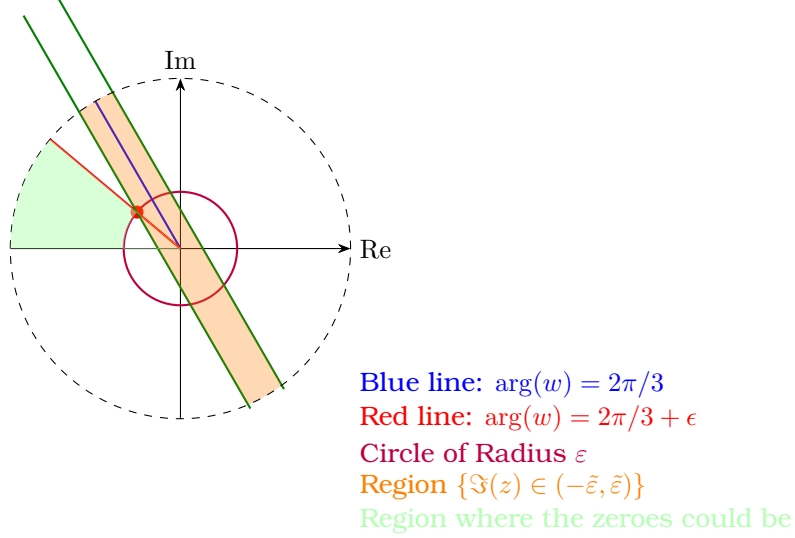
	\begin{align}
		-\frac{1}{2}\cot(\iota)(1+i\sqrt{3})=-\kappa+\sum_{n=1}^{\infty}\left(\frac{1}{w_r+w_ii+w_n}-\frac{1}{w_n}\right).
	\end{align}
	By computing the fraction and taking real and imaginary parts:
	\begin{align}
		&-\frac{\cot(\iota)}{2}=-\kappa+\sum_{n=1}^{\infty}\left(\frac{w_r+w_n}{w_i^2+(w_r+w_n)^2}-\frac{1}{w_n}\right)\\
		&-\frac{\sqrt{3}\cot(\iota)}{2}=\sum_{n=1}^{\infty}\left(\frac{w_i}{w_i^2+(w_r+w_n)^2}\right).
	\end{align}
	The second line implies that $\cot(\iota)$ has always the opposite sign to $w_i$ and in particular, for $w_i\ne 0$
	\begin{align}
		\sum_{n=1}^{\infty}\left(\frac{1}{w_i^2+(w_r+w_n)^2}\right)=-\frac{\sqrt{3}\cot(\iota)}{2w_i}>0
	\end{align}
We now consider the real part:
	\begin{align}
		&-\kappa+\sum_{n=1}^{\infty}\left(\frac{w_r+w_n}{w_i^2+(w_r+w_n)^2}-\frac{1}{w_n}\right)=	-\kappa-\sum_{n=1}^{\infty}\left(\frac{w_r}{w_i^2+(w_r+w_n)^2}+\frac{w_r^2}{(w_i^2+(w_r+w_n)^2)w_n}\right)\\
		&=-\kappa+\frac{w_r\sqrt{3}\cot(\iota)}{2w_i}-w_r^2\sum_{n=1}^{\infty}\frac{1}{(w_i^2+(w_r+w_n)^2)w_n}=-\frac{\cot(\iota)}{2}.
	\end{align}
	Now, since the sum is positive we can transform the equality in an inequality
	\begin{align}
		\label{ineqwei}
		2\kappa<\frac{w_r\sqrt{3}\cot(\iota)}{w_i}+\cot(\iota).
	\end{align}
	We now separate distinguish three cases depending on the sign of $\cot(\iota)$:\\
	If $\cot(\iota)>0$, by \eqref{ineqwei}, we obtain, $w_i<0$,  $\cot(\arg(w))=\frac{w_r}{w_i}>-\frac{1}{\sqrt{3}}+\frac{2\kappa}{\sqrt{3}\cot(\iota)}$ which means that, since $w_i<0$, that is $\arg(w)\in(\pi,2\pi)$, $\pi<\arg(w)<\frac{5\pi}{3}-\epsilon$ for some $\varepsilon>0$ depending on $\iota,\kappa$.\\\\
	If $\cot(\iota)<0$, by \eqref{ineqwei}, we obtain, $w_i>0$,  $\cot(\arg(w))=\frac{w_r}{w_i}<-\frac{1}{\sqrt{3}}+\frac{2\kappa}{\sqrt{3}\cot(\iota)}$ which means that, since $w_i>0$, that is $\arg(w)\in(0,\pi)$ $\pi>\arg(w)>\frac{2\pi}{3}+\epsilon$ for some $\varepsilon>0$ depending on $\iota,\kappa$.\\\\
	If $\cot(\iota)=0$, $w_i=0$, that is, $w$ must be real.\\
	In brief, for every $\iota$, there exists a $\epsilon>0$ such that $\arg(w)\notin (\frac{2\pi}{3}-\epsilon,\frac{2\pi}{3}+\epsilon)\cup (\frac{5\pi}{3}-\epsilon,\frac{5\pi}{3}+\epsilon)$ that is, $\arg(z)\notin (0,\epsilon)\cup(\pi-\epsilon,\pi+\epsilon)\cup (2\pi-\epsilon,2\pi)$, but also, by continuity, $|z|>\varepsilon$, so this implies that, there exists an $\tilde{\varepsilon}>0$ such that $\Im(z)\notin (-\tilde{\varepsilon},\tilde{\varepsilon})$, as we wanted.
\end{proof}
By this proposition, we obtain the following decomposition:
\begin{prop}
	\label{decomp}
	There exists $\rho_\iota,\theta_\iota\in\mathcal{C}^{\infty}$ such that, for every $x\in \mathbb{R}$, 
	\begin{equation}
		A_{-,\iota}(x)=\rho_\iota(x)e^{i\theta_\iota(x)}.
	\end{equation}
	With $\rho_\iota$ strictly positive
\end{prop}
\begin{proof}
	By the previous proposition, there exists a complex logarithm in this strip for $A_{-,\iota}$, that is, there exists $g$ analytic in $\mathcal{B}(\tilde\varepsilon)$ such that $A_{-,\iota}(z)=e^{g(z)}$. In particular, that is true in $\mathbb{R}$, so we can conclude.
\end{proof}
\begin{obs}
	The function $\theta_\iota$ determined only up to addition of $2\pi k$. So we must choose a branch of $\theta$. However, the derivative of $\theta$ does not depend on the choice.
\end{obs}
\begin{defi}
	We choose $\theta_\iota$, such that $-\theta_\iota\left(w_{1,\iota}\right)$ is the smallest possible positive number.
\end{defi}
\begin{prop}
	\label{inipoint}
	For every $\iota\in[0,\pi)$,
	$\theta_\iota(w_{1,\iota})=-\pi/2$
\end{prop}
\begin{proof}
	By definition 
	\begin{equation}
		A_{-}(w_{1,\iota})=\frac{1}{2}[\cos(\iota)Bi(-w_{1,\iota})-\sin(\iota)Bi'(-w_{1,\iota})]i=\rho_\iota(w_{1,\iota})e^{i\theta_\iota(w_{1,\iota})}.
	\end{equation}
	This implies that $\cos(\theta_\iota(w_{1,\iota}))=0$ which implies that $\theta_\iota(w_{1,\iota})=-\pi/2$ or $\theta_\iota(w_{1,\iota})=-3\pi/2$, depending on the sign of $[\cos(\iota)Bi(-w_{1,\iota})-\sin(\iota)Bi'(-w_{1,\iota})]$. (It would be $\pm\frac{\pi}{2}+2\pi k$, but we defined $\theta_\iota(w_{1,\iota})$ as the smallest positive value possible).\\
	If $\iota=0$, $[\cos(\iota)Bi(-w_{1,\iota})-\sin(\iota)Bi'(-w_{1,\iota})]=Bi(-w_{1})<0$. This means that $\theta_0(w_1)=-\frac{\pi}{2}$.\\
	We know that $[\cos(\iota)Bi(-w_{1,\iota})-\sin(\iota)Bi'(-w_{1,\iota})]$ is continuous with respect to $\iota\in (0,\pi)$ by \eqref{zeros} and that it cannot be zero by \eqref{decomp}, so it must always be negative, so, for every $\iota\in[0,\pi)$, $\theta_\iota(w_{1,\iota})=-\frac{\pi}{2}$. 
\end{proof}
As before, we can use the Wronskian to relate the modulus and the phase with a differential equation.
\begin{prop}
	\label{phase}
	For every  $x>w_{1,\iota}$ or $x>0$, we have that $\theta_\iota'(x)<0$, more specifically, for every real $x$,
	\begin{equation}
		\theta'_\iota(x)=-\frac{\cos^2(\iota)+\sin^2(\iota)x}{4\pi \rho^2_\iota(x)}.
	\end{equation}
\end{prop}
\begin{proof}
	We start by defining the functions $$f(x)=\Re(A_-(x))=\frac{1}{2}(\cos(\iota)Ai(-x)-\sin(\iota)Ai'(-x))=\rho_\iota(x)\cos(\theta_\iota(x)),$$
	$$g(x)=\Im(A_-(x))=\frac{1}{2}(\cos(\iota)Bi(-x)-\sin(\iota)Bi'(-x))=\rho_\iota(x)\sin(\theta_\iota(x)).$$
	We compute the derivatives 
	$$2f'(x)=-\cos(\iota)Ai'(-x)-x\sin(\iota)Ai(-x)=\rho_\iota'(x)\cos(\theta_\iota(x))-\rho_\iota(x)\sin(\theta_\iota(x))\theta_\iota'(x),$$
	$$2g'(x)=-\cos(\iota)Bi'(-x)-x\sin(\iota)Bi(-x)=\rho_\iota'(x)\sin(\theta_\iota(x))+\rho_\iota(x)\cos(\theta_\iota(x))\theta_\iota'(x).$$
	We compute now $g'f-gf'$
	\begin{align}
		&4(f(x)g'(x)-f'(x)g(x))\\
		&=-\cos^2(\iota)Ai(-x)Bi'(-x)+\sin^2(\iota)xAi'(-x)Bi(-x)\\
		&-x\cos(\iota)\sin(\iota)Ai(-x)Bi(-x)+\cos(\iota)\sin(\iota)Ai'(-x)Bi'(-x)\\
		&+\cos^2(\iota)Ai'(-x)Bi(-x)-x\sin^2(\iota)Ai(-x)Bi'(-x)\\
		&-\cos(\iota)\sin(\iota)Ai'(-x)Bi'(-x)+x\sin(\iota)\cos(\iota)Ai(-x)Bi(-x)\\
		&=-\frac{\cos^2(\iota)+\sin^2(\iota)x}{\pi}\\
		&=4\rho_\iota^2(x)\cos^2(\theta_\iota(x))\theta_\iota'(x)+8\rho_\iota(x)\rho_\iota'(x)\cos(\theta_\iota(x))\sin(\theta_\iota(x))\\
		&+4\rho_\iota^2(x)\sin^2(\theta_\iota(x))\theta_\iota'(x)-8\rho_\iota(x)\rho_\iota'(x)\sin(\theta_\iota(x))\cos(\theta_\iota(x))\\
		&=4\rho_\iota(x)\theta_\iota'(x).
	\end{align}
	We now just have to prove that, for every $\iota$, $\cos^2(\iota)+w_{1,\iota}\sin^2(\iota)>0$.\\\\
	If $\sin(\iota)=0$, the result is immediate, and if $\sin(\iota)\ne0$, we can rewrite the condition as
	$-w_{1,\iota}<\cot^2(\iota)$. Since $Lai$ is always strictly decreasing and has range in $(-w_1,\infty)$, by definition, we get that $-w_{1,\iota}=Lai^{-1}(\cot(\iota))$, and we obtain $Lai^{-1}(\cot(\iota))<\cot^2(\iota)$. Since we want to prove that for every $\iota$, it suffices to prove $Lai^{-1}(t)<t^2$ for every t, and since $t^2$ is positive, we can apply $Lai$, together with the fact that $Lai$ is strictly decreasing, so the inequality we have to prove transforms to $t>Lai(t^2)$. Since $Lai(0)<0$ and $Lai$ is decreasing, for $t\ge0$, it follows that the inequality is true $t\ge0>Lai(0)>Lai(t^2)$. If $t<0$, we have that writing the expression like this $-Lai(t^2)>-t$, both terms are positive, so we can square them and obtain an equivalent expression $Lai^2(t^2)>t^2$, and we conclude using \eqref{lemalai} $t^2-Lai^2(t^2)=Lai'(t^2)<0$ as we wanted.
\end{proof}

We also define the following function, which will allows us to show more explicitly some of the properties of $\theta$.
\begin{defi}
	\label{defL}
	We define the real function $L_\iota$ as $L_\iota(x)=\pi-2\theta_\iota(x)$.
\end{defi}
We get the following properties.
\begin{lema}
	\label{LemaL}
	For all $k\in\mathbb{N}$, $L_\iota(w_{\iota,k})=2\pi k$.\\
	Furthermore, for every $k\in\mathbb{N}$, we also have
	\begin{equation}
		\label{Lprima}
		L_\iota'(w_{\iota,k})=2\pi \int_{0}^{\infty} Ai^2(x-w_{\iota,k})dx.
	\end{equation}
\end{lema}
\begin{proof}
	We prove the first part. First, we do know, by proposition \ref{inipoint}, that $L_\iota(w_{\iota,1})=2\pi$, and that for every $t>w_{\iota,1}$, $L_\iota$ is an increasing function. We will prove the following:
	\begin{lema}
		$L_\iota(t)=2\pi n$, with $t\in\mathbb{R}$ and $n\in\mathbb{Z}$ if and only if $t=w_{\iota,k}$ for some $k$. (That is $-t$ is a zero of the Robin function).
	\end{lema} 
	\begin{proof}
		If $t=w_{\iota,k}$, then $A_-(t)$ is a pure imaginary number which means that $\theta_\iota(t)$ must be $\pm \frac{pi}{2}+2\pi n$. So $L_\iota(t)=\pi-2(\pm \frac{pi}{2}+2\pi n)=\pi \pm \pi -4\pi n=2\pi m$ for an $m\in\mathbb{Z}$.\\
		Reversely if  $L_\iota(t)=2\pi m$, that means that $\theta_\iota(t)=\frac{\pi}{2}-\pi m$, which means that $e^{i\theta_\iota(t)}$ is a pure imaginary number, which means that $Rob_\iota(-t)=\Re(A_-(t))=0$, so $-t$ is a zero of $Rob_\iota$, which concludes the proof do the lemma.
	\end{proof}
	Now, since $L_\iota(w_{\iota,1})=2\pi$, We will conclude by induction. If $L_\iota(w_{\iota,k})=2\pi k$, since there are infinite zeroes of the Robin function and $L$ is a increasing function, $L_\iota(w_{\iota,k+1})=2\pi m$, by the lemma, with $m>k$, but if $m\ne k+1$, there exists a $t$ such that $L_\iota(t)=2\pi (k+1)$, but by the lemma $t=w_{\iota,u}$, for a natural $u$, but by continuity and the fact that $L$ is increasing, $w_{\iota,k}<w_{\iota,u}<w_{\iota,k+1}$, which contradicts the definition of $w_{\iota,i}$.\\\\
	Now for the second statement we have, using \eqref{integral}, we have
	\begin{align}
		\int_{0}^{\infty} Ai^2(x-w_{\iota,k})dx=w_{\iota,k}Ai(-w_{\iota,k})^2+Ai'(-w_{\iota,k})^2.
	\end{align}
	By definition $0=\cos(\iota)Ai(-w_{\iota,k})-\sin(\iota)Ai'(-w_{\iota,k})$. We assume $\sin(\iota)\ne 0$ (the case $\sin(\iota)=0$ is analogous and it is the classical case in A.1 \cite{positive}).\\
	In that case $Ai'(-w_{\iota,k})=Ai(-w_{\iota,k})\frac{1}{\tan(\iota)}$.\\
	We now compute the following $Rob'(t)=\cos(\iota)Ai'(t)-\sin(\iota)tAi(t)$, so we get
	$Rob'_\iota(-w_{\iota,k})\\=\cos(\iota)Ai'(-w_{\iota,k})+\sin(\iota)w_{\iota,k}Ai(-w_{\iota,k})=\frac{\cos^2(\iota)+\sin^2(\iota)w_{\iota,k}}{\sin(\iota)}Ai(-w_{\iota,k})$.\\
	Also $Rob_\iota(t)=2\Re(A_{-,\iota}(-t))=2\rho_\iota(-t)\cos({\theta_\iota(-t)})$, so\\
	$Rob'_\iota(-w_{\iota,k})=-2\rho_\iota'(w_{\iota,k})\cos(\theta_\iota(w_{\iota,k}))+2\theta'_\iota(w_{\iota,k})\rho_\iota(w_{\iota,k})\sin(\theta_\iota(w_{\iota,k}))$. But by definition of zero of $Rob$, $\cos(\theta_\iota(w_{\iota,k}))=0$ and $\sin(\theta_\iota(w_{\iota,k}))=\pm 1$, so $(Rob'_\iota(-w_{\iota,k}))^2=4\theta'_\iota(w_{\iota,k})^2\rho_\iota(w_{\iota,k})^2$.
	\\
	With all the discussed in the expression above
	\begin{align}
		&\int_{0}^{\infty} Ai^2(x-w_{\iota,k})dx=Ai(-w_{\iota,k})^2 (w_{\iota,k}+\frac{1}{\tan(\iota)^2})\\
		&=\frac{\sin^2(\iota)}{(\cos^2(\iota)+\sin^2(\iota)w_{\iota,k})^2}(Rob'_\iota(w_{\iota,k})^2)(\sin^2(\iota)w_{\iota,k}+\cos^2(\iota))\frac{1}{\sin^2(\iota)}\\
		&=\frac{Rob'_\iota(w_{\iota,k})^2}{(\sin^2(\iota)w_{\iota,k}+\cos^2(\iota))}=4\frac{\theta'_\iota(w_{\iota,k})^2\rho_\iota(w_{\iota,k})^2}{(\sin^2(\iota)w_{\iota,k}+\cos^2(\iota))}=-\frac{ \theta'_\iota(w_{\iota,k})}{\pi}.
	\end{align}
	Where the last line comes from \eqref{phase}.\\
	With this, we conclude since $L'_\iota(t)=-2\theta_\iota'(t)$.
\end{proof}
Finally, we end this section with an adapted version of Poisson formula:
\begin{lema}
	\label{AiryPoisson}
	In $\mathcal{D}'(\mathbb{R}_+)$, one has
	\begin{equation}
		\label{eq:AiryPoissonrob1}
		\sum_{N\in \mathbb{Z}} e^{-i NL_\iota(w)}= 2\pi \sum_{k\in \mathbb{N}} \frac 1
		{L'_\iota(w_{\iota,k})}\, \delta(w-w_{\iota,k}).
	\end{equation}
	In other words, for $g\in \mathcal C_{0}^{\infty}(\mathbb{R}_+)$,
	\begin{equation}
		\label{eq:AiryPoissonrob}
		\sum_{N\in \mathbb{Z}} \int e^{-i NL_\iota(w)} g(w)\,dw = 2\pi \sum_{k\in \mathbb{N}} \frac 1
		{L_\iota'(w_{\iota,k})}\, g(w_{\iota,k}).
	\end{equation}
\end{lema}
\begin{proof}
	Let $g\in\mathcal{C}_0^\infty(\mathbb{R}_+)$ (which can be defined in all $\mathbb{R}$ extending by $0$). The classical Poisson summation formula reads that, for every $f\in\mathcal{C}_0^\infty(\mathbb{R})$, 
	\begin{equation}
		\sum_{k\in \mathbb{Z}} \varphi(2\pi k)=\frac 1 {2\pi} \sum_{N\in \mathbb{Z}} \int_{\mathbb{R}} e^{-i N x}\varphi(x)dx.
	\end{equation}
	Since $L_\iota$ has positive derivative in $\mathbb{R}_+$ and $\lim_{x\rightarrow \infty}L_\iota(x)=\infty$ (because $L_\iota$ grows and $\lim_{k\rightarrow \infty}L_\iota( w_{\iota,k})=\lim_{k\rightarrow \infty}2\pi k=\infty$), $L_\iota$ is diffeomorphism between $[0,\infty]$ and $\left(L_\iota(0),\infty\right)$, so we can define, for $x\in \left(L_\iota(0),\infty\right)$ the following function 
	\begin{equation}
		\varphi(x)=\frac{g(L_\iota^{-1}(x))}{L_\iota'(L_\iota^{-1}(x))}.
	\end{equation} 
	Since $g$ has its support in $\mathbb{R_+}$, $\varphi$ has its support in $\left(L_\iota(0),\infty\right)$ and we can extend it to all $\mathbb{R}$.\\\\
	Now, using the Poisson summation formula, we obtain
	\begin{equation}
		\sum_{k\in \mathbb{N}} \varphi(L_\iota( w_{b,k}))=\sum_{k\in \mathbb{N}} \varphi(2\pi k)=\frac 1 {2\pi} \sum_{N\in \mathbb{Z}} \int_{\mathbb{R}} e^{-i N x}\varphi(x)\,dx=\frac 1 {2\pi} \sum_{N\in \mathbb{Z}} \int_{L_\iota(0)}^\infty e^{-i N x}\varphi(x)\,dx.
	\end{equation}
	By doing the change of variables $x=L_\iota(t)$ in the integral, and using the definition of $\varphi$, we obtain 
	\begin{equation}
		\sum_{k\in \mathbb{N}} \frac{g(w_{\iota,k})}{L_\iota'( w_{\iota,k})}=\frac 1 {2\pi} \sum_{N\in \mathbb{Z}} \int_{0}^\infty e^{-i N L_\iota(t)}g(t)\,dt=\frac 1 {2\pi} \sum_{N\in \mathbb{Z}} \int_{\mathbb{R}} e^{-i N L_\iota(t)}g(t)\,dt.
	\end{equation}
\end{proof}
\subsection{Asymptotics}
We now compute some asymptotics which will be useful.\\
The asymptotics for $A_{-}, A'_{-}$ are well known (and can be obtained by stationary phase).
\begin{prop}[Asymptotics]
	\label{basicasympt}
	For $x>0$, we obtain the following asymptotics:
	\begin{align}
		&A_-(x)=\frac{e^{i(\frac{\pi}{4}-\frac{2}{3}x^{3/2})}}{2\sqrt{x}\sqrt{\pi}}\left(\sum_{k=0}^{m-1}x^{-\frac{3}{2}k}\frac{i^{k}(6k)! }{2^{6k}(2k)!(3k)!3^{2k}}+i^m\tilde{O}^\infty(x^{-\frac{3}{2}m})+i^{m+1}\tilde{O}^\infty(x^{-\frac{3}{2}{m-\frac{3}{2}}})\right)\\
		&=\frac{e^{i(\frac{\pi}{4}-\frac{2}{3}x^{3/2})}}{2\sqrt{x}\sqrt{\pi}}\Big(1+\frac{5}{2^43}\frac{i}{x^{\frac{3}{2}}}-\frac{5\cdot7\cdot11}{2^93^2}\frac{1}{x^{\frac{6}{2}}}-\frac{5\cdot7\cdot11\cdot13\cdot17}{2^{13}3^4}\frac{i}{x^{\frac{9}{2}}}+\tilde{O}^\infty(x^{-\frac{12}{2}})+i\tilde{O}^\infty(x^{-\frac{15}{2}})\Big).\\
		&\label{asympa-}
	\end{align}
	\begin{align}
		&A_-'(x)\\
		&=\frac{i\sqrt{x}e^{i(\frac{\pi}{4}-\frac{2}{3}x^{3/2})}}{2\sqrt{\pi}}\left(\sum_{k=0}^{m-1}x^{-\frac{3}{2}k}\frac{i^{k+1}(6k)!(6k+1) }{2^{6k}(6k-1)(2k)!(3k)!3^{2k}}+i^{m}\tilde{O}^\infty(x^{-\frac{3}{2}m-\frac{3}{2}})+i^{m+1}\tilde{O}^\infty(x^{-\frac{3}{2}m-\frac{3}{2}})\right)\\
		&\label{asympa-'}=\frac{i\sqrt{x}e^{i(\frac{\pi}{4}-\frac{2}{3}x^{3/2})}}{2\sqrt{\pi}}\left(-1+\frac{7}{5}\frac{5}{2^43}\frac{i}{x^{\frac{3}{2}}}-\frac{13}{11}\frac{5\cdot7\cdot11}{2^93^2}\frac{1}{x^{\frac{6}{2}}}-\frac{19}{17}\frac{5\cdot7\cdot11\cdot13\cdot17}{2^{13}3^4}\frac{i}{x^{\frac{9}{2}}}\right.\\
		&\left.+\tilde{O}^{\infty}(x^{-\frac{12}{2}})+i\tilde{O}^\infty(x^{-\frac{15}{2}})\right).
	\end{align}
\end{prop}
Since $A_{-,\iota}=\cos(\iota)A_{-}+\sin(\iota)A'_{-}$, and the asymptotics for $A_{-}$,$A'_{-}$ them are known, we have the asymptotics also for $A_{-,\iota}$. This allows us to compute also the asymptotics for $\rho^2_\iota(x)$.
\begin{align}
	&\sqrt{x}^{-1}4\pi \rho^2_\iota(x)=\left[x^{-1/2}\cos(\iota)\left(1+a_2x^{-6/2}+\tilde{O}^\infty(x^{-12/2})\right)+\sin(\iota)\left(a'_1x^{-3/2}+\tilde{O}^\infty (x^{-9/2})\right) \right]^2\\
	&+\left[x^{-1/2}\cos(\iota)\left(a_1x^{-3/2}+\tilde{O}^\infty(x^{-9/2})\right)+\sin(\iota)\left(-1+a'_2x^{-6/2}+\tilde{O}^\infty(x^{-12/2})\right) \right]^2\\
	&\label{asymptmodrob}=\cos^2(\iota)(x^{-1}+\tilde O^\infty (x^{-4}))+\sin^2(\iota)(1+\tilde O^\infty(x^{-3}))+2\sin(\iota)\cos(\iota)\left(\frac{2}{5} \frac{5}{2^{4}3}\right)(x^{-2}+\tilde O^\infty(x^{-5})).
\end{align}
\begin{nota}
	All the $\tilde{O}^\infty$ involved are real functions, stable under differentiation and bounded also from below, that is, if $f(w)=\tilde{O}^\infty(w^{-a})$, then, if $w>M$ for a certain $M$, there exists, $C_1,C_2$ such that $C_1w_{-a}\le|f(w)|\le C_2w^{-a}$, and for every $k$, $f^{(k)}(w)=O(w^{-a-k})$.
\end{nota}
\subsection{Derivatives with respect to $\iota$}
For now we have treated $\iota$ in the Robin function as a parameter, but we are also interested in considering $\iota$ as a variable, since in the spectral analysis, we obtain dependence of $\theta$ on the parameter. First the function $Rob$ is smooth if we consider $x$ and $\iota$ as variables. We also have that $\rho_\iota^2$ and $A_{\iota,-}(x)$ are, but it is also the case for $\theta_\iota(x)$.\\
We start by introducing the notation of the section:
\begin{nota}
	We denote $\partial_\iota\theta_\iota(x)$ as the derivative of $\theta_\iota(x)$ with respect to $\iota$, when we consider $\iota$ as a variable.\\
	We denote $\partial_x\theta_\iota(x)$ as the derivative of $\theta_\iota(x)$ with respect to $x$, when we consider $\iota$ as a variable. However, we will also maintain  the notation $\theta_\iota'(x)=\partial_x\theta_\iota(x)$.\\
	We will also use this convention for the other functions involved ($Rob$, $\rho$...) if needed.
\end{nota}
We now proceed to compute $\partial_\iota\theta_\iota(x)$.
\begin{prop}
	\label{partialalpha}
	\begin{equation}
		\partial_\iota\theta_\iota(x)=-\frac{1}{4\pi \rho^2_\iota(x)}=\frac{\theta'_\iota(x)}{\cos^2(\iota)+\sin^2(\iota)x}.
	\end{equation}
\end{prop}
We use the same ideas as the proof of proposition \ref{phase}.\\
We start by defining the functions
$$f(x,\iota)=\Re(A_{-,\iota}(x))=\frac{1}{2}(\cos(\iota)Ai(-x)-\sin(\iota)Ai'(-x))=\rho_\iota(x)\cos(\theta_\iota(x)),$$
and
$$g(x,\iota)=\Im(A_{-,\iota}(x))=\frac{1}{2}(\cos(\iota)Bi(-x)-\sin(\iota)Bi'(-x))=\rho_\iota(x)\sin(\theta_\iota(x)).$$
We compute the derivatives 
$$2\partial_\iota f(x,\iota)=-\cos(\iota)Ai'(-x)-\sin(\iota)Ai(-x)=2\partial_\iota\rho_\iota(x)\cos(\theta_\iota(x))-2\rho_\iota(x)\sin(\theta_\iota(x))\partial_\iota\theta_\iota(x).$$
$$2\partial_\iota g(x,\iota)=-\cos(\iota)Bi'(-x)-\sin(\iota)Bi(-x)=2\partial_\iota\rho_\iota(x)\sin(\theta_\iota(x))+2\rho_\iota(x)\cos(\theta_\iota(x))\partial_\iota\theta_\iota(x).$$
We compute now $(\partial_\iota g)f-g(\partial_\iota f)$,
\begin{align}
	&4(f(x,\iota)\partial_\iota g(x,\iota)-\partial_\iota f(x,\iota)g(x,\iota))\\
	&=-\cos^2(\iota)Ai(-x)Bi'(-x)+\sin^2(\iota)Ai'(-x)Bi(-x)\\
	&-\cos(\iota)\sin(\iota)Ai(-x)Bi(-x)+\cos(\iota)\sin(\iota)Ai'(-x)Bi'(-x)\\
	&+\cos^2(\iota)Ai'(-x)Bi(-x)-\sin^2(\iota)Ai(-x)Bi'(-x)\\
	&-\cos(\iota)\sin(\iota)Ai'(-x)Bi'(-x)+\sin(\iota)\cos(\iota)Ai(-x)Bi(-x)\\
	&=-\frac{\cos^2(\iota)+\sin^2(\iota)}{\pi}=-\frac{1}{\pi}\\
	&=4\rho_\iota^2(x)\cos^2(\theta_\iota(x))\partial_\iota \theta_\iota(x)+8\rho_\iota(x)\partial_\iota \rho_\iota(x)\cos(\theta_\iota(x))\sin(\theta_\iota(x))\\
	&+4\rho_\iota^2(x)\sin^2(\theta_\iota(x))\partial_\iota \theta_\iota(x)-8\rho_\iota(x)\partial_\iota \rho_\iota(x)\sin(\theta_\iota(x))\cos(\theta_\iota(x))\\
	&=4\rho^2_\iota(x)\partial_\iota \theta_\iota(x).
\end{align}
Thus we obtain the expression we wanted (and we get the second expression by using proposition \ref{phase})\\
\section{Derivatives of $\psi_{N,a,h}$}
\label{derivatives}
We compute here the derivatives of $\psi_{N,a,h}$ in order to make the tracking of the bounds clearer.
\begin{prop}[First derivative]
	\begin{align}
		&\partial_\Sigma \psi_{N,a,h}(\boldsymbol{ T},a,{\Sigma},{S},{A},\eta)=\eta \left(\Sigma^2+ X-A\right).\\
		&\partial_S \psi_{N,a,h}(\boldsymbol{ T},a,{\Sigma},{S},{A},\eta)=\eta \left(S^2+\frac{a}{\gamma}-A\right).\\
		&\partial_A \psi_{N,a,h}(\boldsymbol{ T},a,{\Sigma},{S},{A},\eta)=-\eta S-\eta \Sigma+\frac{\eta T}{2\sqrt{1+\gamma A}}-2N\eta A^{1/2}-\frac{3}{2}A^{1/2}N\eta B'_{\beta(\eta/h)}(A^{3/2}\lambda_{\gamma,h}{\eta}).\\
		&\partial_\eta \psi_{N,a,h}(\boldsymbol{ T},a,{\Sigma},{S},{A},\eta)=Y+\frac{\Sigma^3}{3}+\Sigma (X-A)+\frac{S^3}{3}+S(\frac{a}{\gamma}-A)+\frac{T}{\gamma}(\sqrt{1+\gamma A}-\sqrt{1+\gamma})\\
		&-\frac{4}{3}NA^{3/2}\left(1+\frac{3}{4}B'_{\beta(\eta/h)}(A^{3/2}\eta\lambda_{\gamma,h})\right)-N\lambda_{\gamma,h}^{-1}Resrob_1(\gamma A,\eta,h)
	\end{align}
\end{prop}
\begin{prop}[Second derivative]
	\begin{align}
		&\partial^2_\Sigma \psi_{N,a,h}(\boldsymbol{ T},a,{\Sigma},{S},{A},\eta)=2\eta\Sigma.\\
		&\partial^2_S \psi_{N,a,h}(\boldsymbol{ T},a,{\Sigma},{S},{A},\eta)=2\eta S.\\
		&\partial^2_A \psi_{N,a,h}(\boldsymbol{ T},a,{\Sigma},{S},{A},\eta)=-\frac{\eta T\gamma}{4\sqrt{1+\gamma A}^3}-N\eta A^{-1/2}-A^{-1/2}\frac{3}{4}N\eta B'_{\beta(\eta/h)}(A^{3/2}\lambda_{\gamma,h}{\eta})\\
		&-\frac{9}{4}A N\eta^2 \lambda_{\gamma,h} B''_{\beta(\eta/h)}(A^{3/2}\lambda_{\gamma,h}{\eta}).\\
		&\partial^2_\eta \psi_{N,a,h}(\boldsymbol{ T},a,{\Sigma},{S},{A},\eta)=-N\lambda_{\gamma,h}{A^{3}} B''_{\beta(\eta/h)}(\lambda_{\gamma,h}\eta{A^{3/2}})-N\lambda_{\gamma,h}^{-1}Resrob_2(\gamma A,\eta,h).\\
		&\partial^2_{\Sigma,S}\psi_{N,a,h}(\boldsymbol{ T},a,{\Sigma},{S},{A},\eta)=\partial^2_{\Sigma,S}\psi_{N,a,h}(\boldsymbol{ T},a,{\Sigma},{S},{A},\eta)=0.\\
		&\partial^2_{\Sigma,A}\psi_{N,a,h}(\boldsymbol{ T},a,{\Sigma},{S},{A},\eta)=\partial^2_{S,A}\psi_{N,a,h}(\boldsymbol{ T},a,{\Sigma},{S},{A},\eta)=-\eta.\\
		&\partial^2_{\Sigma,\eta}\psi_{N,a,h}(\boldsymbol{ T},a,{\Sigma},{S},{A},\eta)=\Sigma^2+X-A. \;\;\;\;\;\;\;\;\partial^2_{S,\eta}\psi_{N,a,h}(\boldsymbol{ T},a,{\Sigma},{S},{A},\eta)=S^2+\frac{a}{\gamma}-A.\\
		&\partial^2_{A,\eta}\psi_{N,a,h}(\boldsymbol{ T},a,{\Sigma},{S},{A},\eta)= -\Sigma-S+\frac{T}{2\sqrt{1+\gamma A}}-2NA^{1/2}-NA^{1/2}\frac{3}{2}B'_{\beta(\eta/h)}(A^{3/2}\eta \lambda_{\gamma,h})\\
		&-\frac{3}{2}NA^2\eta \lambda_{\gamma,h}B''_{\beta(\eta/h)}(A^{3/2}\eta \lambda_{\gamma,h})-N\lambda_{\gamma,h}^{-1}\gamma\partial_A Resrob_1(\gamma A,\eta,h)\\
		&=\eta^{-1} \partial_A \psi_{N,a,h}(\boldsymbol{ T},a,{\Sigma},{S},{A},\eta)-\frac{3}{2}NA^2\eta \lambda_{\gamma,h}B''_{\beta(\eta/h)}(A^{3/2}\eta\lambda_{\gamma,h})-N\lambda_{\gamma,h}^{-1} \gamma \partial_A Resrob_1(\gamma A,\eta,h)
	\end{align}
\end{prop}
\begin{prop}[Third derivative]
	\label{thirderiv1}
	\begin{align}
		&\partial^3_\Sigma \psi_{N,a,h}(\boldsymbol{ T},a,{\Sigma},{S},{A},\eta)=\partial^3_S \psi_{N,a,h}(\boldsymbol{ T},a,{\Sigma},{S},{A},\eta)=2\eta.\\
		&\partial^3_A \psi_{N,a,h}(\boldsymbol{ T},a,{\Sigma},{S},{A},\eta)=\frac{3\eta \gamma^2 T}{8\sqrt{1+\gamma A}^5}+\frac{1}{2}N\eta A^{-3/2}+A^{-3/2}\frac{3}{8}N\eta B'_{\beta(\eta/h)}(A^{3/2}\lambda_{\gamma,h}{\eta})\\
		&- \lambda_{\gamma,h}\frac{27}{8}N\eta^2 B''_{\beta(\eta/h)}(A^{3/2}\lambda_{\gamma,h}{\eta})-\frac{27}{8}A^{3/2} N\eta^3 \lambda_{\gamma,h}^2 B'''_{\beta(\eta/h)}(A^{3/2}\lambda_{\gamma,h}{\eta}).\\
		&\partial^3_\eta \psi_{N,a,h}(\boldsymbol{ T},a,{\Sigma},{S},{A},\eta)=-N{A^{9/2}}\lambda_{\gamma,h}^2B'''_{\beta(\eta/h)}(\lambda_{\gamma,h}\eta{A^{3/2}})-N\lambda_{\gamma,h}^{-1}Resrob_3(\gamma A,\eta,h).\\
		&\partial^3_{\Sigma,\Sigma,S}\psi_{N,a,h}(\boldsymbol{ T},a,{\Sigma},{S},{A},\eta)=\partial^3_{\Sigma,S,S}\psi_{N,a,h}(\boldsymbol{ T},a,{\Sigma},{S},{A},\eta)=\partial^3_{S,S,A}\psi_{N,a,h}(\boldsymbol{ T},a,{\Sigma},{S},{A},\eta)=0.\\
		&\partial^3_{A,A, \Sigma}\psi_{N,a,h}(\boldsymbol{ T},a,{\Sigma},{S},{A},\eta)=\partial^3_{A,A,S}\psi_{N,a,h}(\boldsymbol{ T},a,{\Sigma},{S},{A},\eta)=\partial^3_{\Sigma,\Sigma,A}\psi_{N,a,h}(\boldsymbol{ T},a,{\Sigma},{S},{A},\eta)=0.\\
		&\partial^3_{A,\Sigma, \eta}\psi_{N,a,h}(\boldsymbol{ T},a,{\Sigma},{S},{A},\eta)=\partial^3_{A,\eta,S}\psi_{N,a,h}(\boldsymbol{ T},a,{\Sigma},{S},{A},\eta)=-1.\\
		&\partial^3_{\eta,\eta,A}\psi_{N,a,h}(\boldsymbol{ T},a,{\Sigma},{S},{A},\eta)=-3N{A^2}\lambda_{\gamma,h} B''_{\beta(\eta/h)}(\lambda_{\gamma,h}{A^{3/2}}\eta)-\frac{3}{2}N{A^{7/2}}\eta\lambda_{\gamma,h}^2B'''_{\beta(\eta/h)}(\lambda_{\gamma,h}\eta{A^{3/2}})\\
		&-N\gamma\lambda_{\gamma,h}^{-1}\partial_A Resrob_2(\gamma A,\eta,h).\\
		&\partial^3_{\eta,A,A}\psi_{N,a,h}(\boldsymbol{ T},a,{\Sigma},{S},{A},\eta)=\eta^{-1} \partial^2_{A,A}\psi_{N,a,h}(\boldsymbol{ T},a,{\Sigma},{S},{A},\eta)-3NA \eta \lambda_{\gamma,h}B''_{\beta(\eta/h)}(\lambda_{\gamma,h}{A^{3/2}}\eta)\\
		&-\frac{9}{4}NA^{5/2}\eta^2\lambda_{\gamma,h}^2 B'''_{\beta(\eta/h)}(\lambda_{\gamma,h}{A^{3/2}}\eta)-N\lambda_{\gamma,h}^{-1}\gamma^2\partial^2_A Resrob_1(\gamma A,\eta,h).\\
	\end{align}
\end{prop}
\bibliography{bib}
\bibliographystyle{plain}
\end{document}